\documentclass[11pt]{amsart}
\usepackage{amsmath,amsfonts,amsthm,amssymb}
\usepackage[pdfborder={0 0 0 [0 0 ]}]{hyperref}
\usepackage{enumerate}
\usepackage{mathrsfs}
\usepackage{enumerate}
\usepackage{graphicx}
\usepackage{color}
\usepackage[ps,all,arc,rotate,cmtip]{xy}
\usepackage{changepage}
\usepackage[capitalise]{cleveref}

\long\def\symbolfootnote[#1]#2{\begingroup
\def\thefootnote{\fnsymbol{footnote}}\footnote[#1]{#2}\endgroup}

\newtheorem*{theorem*}{Theorem}
\newtheorem{theorem}{Theorem}[section]

\newtheorem{lem}[theorem]{Lemma}
\newtheorem{thm}[theorem]{Theorem}

\newtheorem{prop}[theorem]{Proposition}

\newtheorem{cor}[theorem]{Corollary}

\theoremstyle{definition}

\newtheorem{rmk}[theorem]{Remark}

\newtheorem{claim}{Claim}

\newtheorem{dfn}[theorem]{Definition}

\newcommand{\R}{\mathbb{R}}
\newcommand{\Z}{\mathbb{Z}}

\newcommand{\Out}{\mathrm{Out}}
\newcommand{\Aut}{\mathrm{Aut}}
\newcommand{\ord}{\mathrm{ord}}
\newcommand{\st}{\mathrm{st}}
\newcommand{\lk}{\mathrm{lk}}
\renewcommand{\L}{\mathcal{L}}
\newcommand{\X}{\mathcal{X}}
\newcommand{\Y}{\mathcal{Y}}
\renewcommand{\S}{\mathcal{S}}
\renewcommand{\P}{\mathcal{P}}

\def\link-preserving{}

\renewcommand{\leq}{\leqslant}

\def\im{\mathrm{im}}
\def\into{\hookrightarrow}
\def\iff{if and only if }
\def\wrt{with respect to }

\def\GL{\mathrm{GL}}

\def\s-{\smallsetminus}

\def\t{\widetilde}
\def\U{\mathrm{U}}

\begin{document}

\title[Nielsen realisation for RAAGs]{Nielsen realisation for untwisted automorphisms of right-angled Artin groups}
\author{Sebastian Hensel and Dawid Kielak}


\maketitle

\begin{abstract}
\noindent We prove Nielsen realisation for finite subgroups of the groups of untwisted outer automorphisms of RAAGs in
the following sense: given any graph $\Gamma$, and any finite group $G\leqslant \U^0(A_\Gamma) \leqslant \Out^0(A_\Gamma)$, we find a non-positively curved cube complex
with fundamental group $A_\Gamma$ on which $G$ acts by
isometries, realising the action on $A_\Gamma$.
\end{abstract}

\section{Introduction} 
A right-angled Artin group (RAAG) $A_\Gamma$ is a group given by a very simple
presentation, which is defined by a graph $\Gamma$: the group $A_\Gamma$ has one
generator for each vertex of $\Gamma$, and
two generators commute if and only if the corresponding vertices
are joined by an edge in $\Gamma$.

RAAGs have been an object of intense study over the last years, and
indeed seem to be ubiquitous in geometry and topology. The most
striking example is possibly the role they played in the recent
solution of the virtual Haken conjecture by Agol~\cite{Agol2013}.
They also possess a rich intrinsic structure, maybe most visibly so in
the variety of surprising
properties their subgroups can exhibit (see
e.g. the work of Bestvina and Brady~\cite{BestvinaBrady1997}).

\smallskip
A general RAAG $A_\Gamma$ can be seen as interpolating
between a non-abelian free group $F_n$ (corresponding to the graph
with $n$ vertices and no edges) and a free Abelian group $\mathbb{Z}^n$
(defined by the complete graph on $n$ vertices). If a property holds
for both $F_n$ and $\mathbb{Z}^n$, it is then natural to look for an analogue that
works for all RAAGs.

In this article we investigate \emph{Nielsen realisation} from this
point of view. For free groups this takes the following form: suppose
one is given a finite subgroup $H < \Out(F_n)$. Is there a graph $X$
with $\pi_1(X) = F_n$ on which the group $H$ acts by isometries,
inducing the given action on the fundamental group?
The answer turns out to be yes (as shown independently by
Culler~\cite{culler1984}, Khramtsov~\cite{khramtsov1985}, and
Zimmermann~\cite{Zimmermann1981}; see also \cite{Henseletal2014} for a more recent,
topological proof).

Let us note here that Nielsen realisation for free groups is equivalent to the statement that in the action of $\Out(F_n)$ on the Culler--Vogtmann Outer Space every finite subgroup fixes a point. The result is also an essential tool in the work of Bridson--Vogtmann~\cite{bridsonvogtmann2011} and the second-named author~\cite{kielak2013, kielak2015a}, and is used to prove certain rigidity phenomena for $\Out(F_n)$.

The
corresponding statement for free abelian groups follows from the (classical) fact that any finite (in fact compact) subgroup of $\GL_n(\R)$ can be conjugated to be a subgroup of the orthogonal group. This implies that any finite $H < \Out(\Z^n) = \GL_n(\Z)$ acts isometrically on an $n$-torus, and the induced action on the fundamental group is the given one.

\smallskip
For RAAGs the natural analogue is as follows: suppose
one is given a finite subgroup $H < \Out(A_\Gamma)$ in the outer
automorphism group of a RAAG. Is there a compact non-positively curved metric
space on which $H$ acts by isometries, realising the action on the
fundamental group?

The close relationship between RAAGs and cube complexes tempts one to
ask the above question with cube complexes in place of
metric spaces. This is however bound to lead to a negative answer,
since already for general finite subgroups of $\GL_n(\mathbb{Z})$ the action
on the torus described above cannot be made cubical and cocompact simultaneously.

The main result of this article proves Nielsen Realisation
for a large class of RAAGs. The restrictions are chosen in a way allowing us to use cube complexes, and we obtain
\begin{theorem*}
  Suppose $\Gamma$ is a simplicial graph,
  and let $H<\U^0(A_\Gamma)$ be finite. Then there is a compact non-positively curved cube
  complex realising the action of $H$. Moreover, the dimension of the complex is the same as the dimension of the Salvetti complex of $A_\Gamma$.
\end{theorem*}
The group $\U^0(A_\Gamma)<\Out(A_\Gamma)$ is the intersection of the
group $\U(A_\Gamma)$ of untwisted outer automorphisms (introduced by
Charney--Stambaugh--Vogtmann~\cite{charneyetal2012}) with the
finite index subgroup $\Out^0(A_\Gamma)$.

Observe that when $\Gamma$ has no symmetries, and the link of any
vertex in $\Gamma$ is not a cone, then
$\Out(A_\Gamma) = \U^0(A_\Gamma)$. In particular, our result holds for
connected triangle- and symmetry-free defining graphs.

\subsection{Outline of the proof}

Since this article is rather substantial in length, let us offer here
a somewhat informal outline of the proof of the main theorem.


\smallskip
The proof (Sections~\ref{sec:proof-main} through~\ref{sec:proof-main part II}) is
inductive on the dimension of $\Gamma$, that is the maximal size of a maximal
clique in $\Gamma$.

We proceed by identifying maximal proper subgraphs $\Gamma_1, \dots, \Gamma_n$ of $\Gamma$, which are \emph{invariant}, that is we have an induced action of $H$ on (conjugacy classes in) $A_{\Gamma_i}$, the subgroup of $A_\Gamma$ generated by vertices of $\Gamma_i$, for each $i$. We assume that the result holds for $\Gamma_i$.

Now one of the following situations occurs: the subgraphs $\Gamma_i$ might be disjoint, or they might intersect in some non-trivial fashion.
The first of these cases is essentially covered by Relative Nielsen Realisation for free products, the main theorem in our previous article~\cite[Theorem 5.4]{HenselKielak2016a}; the precise statement we use here is \cref{prop: sticking complexes together}.

The second situation requires a different type of argument. Here we take some maximal invariant proper subgraph $\Gamma'$, and build the cube complex for $\Gamma$ from the one for $\Gamma'$ (given by induction), by gluing to it cube complexes for invariant subgraphs which intersect $\Gamma'$ and its complement non-trivially. This process presents two types of difficulties: firstly, we need the complexes to agree on the overlap; this requirement forces us to study Relative Nielsen Realisation, which is a way of building cube complexes for our actions having some prescribed subcomplexes.
To make sense of such statements we introduce \emph{cubical systems} (Section~\ref{sec: cubical systems}), which are precisely cube complexes with subcomplexes realising the induced actions on relevant invariant subgraphs.

The second difficulty arises when we try to glue two cube complexes over a subcomplex they both posses; here care needs to be taken to make sure that the object we obtain from the gluing realises the given action, and not some other action related to the given one by a partial conjugation. This is the content of Section~\ref{sec: gluing}.

For the reasons outlined here, our paper is rich with technical details. The (positive) side effect of this is that the cube complexes realising the action of a finite group $H \to \U^0(A_\Gamma)$ we start with come equipped with a plethora of invariant subcomplexes, which gives our realisation an extra layer of potential applicability. 

\bigskip
\textbf{Acknowledgments.} The authors would like to thank Piotr
Przytycki and Ruth Charney for many helpful comments and
discussions. The authors would furthermore like to thank Karen
Vogtmann for discussions and suggesting the statement and use of
adapted realisation.  

\section{Preliminaries}
\label{sec: prelims}

\subsection{Graphs and RAAGs}
\label{subsec: graphs and raags}
Throughout the paper $\Gamma$ will denote a fixed simplicial graph.
We define the associated RAAG $A_\Gamma$ to be the group generated by
the vertices of $\Gamma$, and with a presentation in which the only
relations occurring are commutators of vertices adjacent in $\Gamma$.

The only subgraphs of $\Gamma$ we will encounter will be induced subgraphs; such a subgraph is uniquely determined by its vertex set (since $\Gamma$ is fixed).
Hence we will
use $\cap, \cup, \s-$ etc.  of two graphs to denote the induced subgraph
spanned by the corresponding operation applied to the vertices of the two graphs.

\begin{dfn}
Let $\Delta, \Sigma$ be two induced subgraphs of $\Gamma$. We say that they form a \emph{join} \iff each vertex in $\Delta$ is connected by an edge (in $\Gamma$) to each vertex of $\Sigma$.
The induced subgraph spanned by all vertices of $\Delta$ and $\Sigma$ will be denoted by $\Delta \ast \Sigma$.

An induced subgraph $\Theta$ is a \emph{join} \iff we have $\Theta = \Delta \ast \Sigma$ for some non-empty induced subgraphs $\Delta$ and $\Sigma$; furthermore $\Theta$ is a cone if $\Delta$ can be taken to be a singleton.
\end{dfn}
Note that, in accordance with our convention, we have $\Delta \ast \Sigma = \Delta \cup \Sigma$, with the join notation indicating the presence of the relevant edges.

Note that induced subgraphs (their vertices to be more specific) generate subgroups of $A_\Gamma$; given such a subgraph $\Delta$ we will call the corresponding subgroup $A_\Delta$. This subgroup is abstractly isomorphic to the RAAG defined by $\Delta$. We adopt the convention $A_{\emptyset} = \{1\}$.

Throughout the paper we use the (standard) convention of denoting the normaliser, centraliser, and centre of a subgroup $H \leqslant A_\Gamma$ by, respectively, $N(H), C(H)$ and $Z(H)$.
We will also use $c(x) \in \Aut(A_\Gamma)$ to denote conjugation by $x \in A_\Gamma$.

We will need the following definitions throughout the paper. Some of them are new; others may be non-standard.
\begin{dfn}
Suppose $\Delta \subseteq \Gamma$ is an induced subgraph.
\begin{enumerate}[i)]
\item The \emph{link} of $\Delta$ is
  \[ \lk(\Delta) = \bigcap_{v \in \Delta} \lk(v) \]
\item The \emph{star} of $\Delta$ is
  \[ \st(\Delta) = \lk(\Delta) \ast \Delta \]
\item The \emph{extended star} of $\Delta$ is
  \[\widehat \st(\Delta) = \lk(\Delta) \ast \lk( \lk(\Delta)) =
  \st(\lk(\Delta)) \]
\item Given a second full subgraph $\Theta \subseteq \Gamma$ with $\Delta
  \subseteq \Theta$ we define the \emph{restricted link} and
  \emph{restricted star} of $\Delta$ in $\Theta$ to be respectively
  \[ \lk_\Theta(\Delta) = \lk(\Delta) \cap \Theta \textrm{ and }
  \st_\Theta(\Delta) = \st(\Delta) \cap \Theta \]
\end{enumerate}
\end{dfn}

Let us observe the following direct consequences of the definition.

\begin{lem}
\label{lem: link and star calculus}
Let $\Delta$ and $\Theta$ be two induced subgraphs of $\Gamma$, and let $v$ be a vertex of $\Gamma$. Then
\begin{enumerate}
\item $\Delta \subseteq \lk(\Theta) \Leftrightarrow \Theta \subseteq \lk(\Delta)$
\item $\Delta \subseteq \st(v) \Rightarrow v \in \st(\Delta)$
\item $\lk(\Delta) \subseteq \st(v) \Rightarrow v \in \widehat \st(\Delta)$
\end{enumerate}
\end{lem}
\begin{proof}
\noindent
\begin{enumerate}
 \item Both statements are equivalent to saying that each vertex in $\Delta$ is connected to each vertex in $\Theta$.
  \item If $v \in \Delta$ then the result follows trivially. If not, then $\Delta \subseteq \lk(v)$ and the result follows from the previous one.
  \item $\lk(\Delta) \subseteq \st(v) \Rightarrow v \in \st(\lk(\Delta)) = \widehat \st(\Delta)$. \qedhere
\end{enumerate}
\end{proof}

\begin{dfn}[Join decomposition]\label{def:join-decomp}
Let $\Delta \subseteq \Gamma$ be an induced subgraph. We say that
\[ \Delta = \Delta_1 \ast \dots \ast \Delta_k\]
is a \emph{join decomposition} of $\Delta$ \iff each $\Delta_i$ is an induced subgraph of $\Gamma$ which is not a join.

We define $Z(\Delta)$ to be the union all subgraphs $\Delta_i$ which are singletons.
\end{dfn}

Such a decomposition is unique up to reordering the factors.

\begin{prop}[{\cite[Proposition~2.2]{charneyetal2012}}]
\label{prop: ccv}
Given $\Delta \subseteq \Gamma$ we have the following identifications
\begin{itemize}
  \item $N(A_\Delta) = A_{\st(\Delta)} = A_\Delta \times A_{\lk(\Delta)}$
  \item $Z(A_\Delta) = A_{Z(\Delta)}$
 \item $C(A_\Delta) = A_{Z(\Delta)} \times A_{\lk(\Delta)}$
\end{itemize}
Given another induced subgraph $\Sigma \subseteq \Gamma$ we also have
\[ x^{-1} A_\Delta x \leqslant A_\Sigma \Longleftrightarrow x \in N(A_\Delta) N(A_\Sigma) \textrm{ and } \Delta \subseteq \Sigma\]
\end{prop}

\begin{dfn}
Given a simplicial graph $\Gamma$ we define its \emph{dimension} $\dim \Gamma$ to be the number of vertices in a largest clique in $\Gamma$.
\end{dfn}
The dimension of $\Gamma$ coincides with the dimension of the Salvetti
complex of $A_\Gamma$.

%
%
\subsection{Words in RAAGs}
\label{subsec: words in raags}

Since $A_\Gamma$ is given in terms of a presentation, its elements are
equivalence classes of words in the alphabet formed by vertices of
$\Gamma$ (which we will refer to simply as the alphabet
$\Gamma$). There is a robust
notion of normal form based on reduced and cyclically reduced
words. We will only mention the results necessary for our arguments;
for further details see the work of Servatius~\cite{Servatius1989}.

\begin{dfn}
Given a word $w=v_1 \cdots v_n$, where each $v_i$ is a letter, i.e. a vertex of $\Gamma$ or its inverse, we define two \emph{basic moves}:
\begin{itemize}
 \item \emph{reduction}, which consists of removing $v_i$ and $v_j$ from $w$ (with $i<j$), provided that $v_i = v_j^{-1}$, and that $v_k$ commutes with $v_i$ in $A_\Gamma$ for each $i<k<j$.
  \item \emph{cyclic reduction}, which consists of removing $v_i$ and $v_j$ from $w$ (with $i<j$), provided that $v_i = v_j^{-1}$, and that $v_k$ commutes with $v_i$ in $A_\Gamma$ for each $k<i$ and $j<k$.
\end{itemize}
A word $w$ which does not allow for any reduction is called \emph{reduced}; if in addition it does not allow for any cyclic reduction, it is called \emph{cyclically reduced}.
\end{dfn}

Servatius shows that, starting with a word $w$, there is a unique reduced word obtainable from $w$ by reductions, and a unique cyclically reduced word obtainable from $w$ by basic moves. It is clear that the former gives the same element of $A_\Gamma$ as $w$ did, and the latter gives the same conjugacy class.

He also shows that two reduced words give the same element in $A_\Gamma$ \iff they differ by a sequence of moves replacing a subword $v v'$ by $v' v$ with $v,v'$ being commuting letters; let us call those \emph{swaps}.

\begin{lem}
\label{lem: cyclically reduced word}
Let $\Sigma \subseteq \Gamma$ be an induced subgraph, and suppose that an element $x \in A_\Gamma$ satisfies $x \in y^{-1} A_\Sigma y$ for some $y \in A_\Gamma$. Then the elements of $A_\Gamma$ given by cyclically reduced words representing the conjugacy class of $x$ lie in $A_\Sigma$.
\end{lem}
\begin{proof}
Take a reduced word $w$ in $\Sigma$ representing $yxy^{-1}$; let $w_y$ be a word in $\Gamma$ representing $y$. Then $w_y w w_y^{-1}$ represents $x$, and it is clear that there is a series of cyclic reductions taking this word to $w$. Further reductions and cyclic reduction will yield another word in $\Sigma$ representing the conjugacy class of $x$. Hence there exists a cyclically reduced word representing the conjugacy class of $x$ as required.

Now suppose that we have two cyclically reduced words, $w$ and $w'$, representing the same conjugacy class in $A_\Gamma$, and such that $w$ is a word in $\Sigma$. In $A_\Gamma$ we have the equation
\[ w = z^{-1} w' z\]
where $z$ is some reduced word in $\Gamma$. Let us take $z$ of minimal word length.

Suppose that the word $z^{-1} w' z$ is not reduced. Then there is a reduction allowed, and it cannot happen within $w'$, since $w'$ is reduced. Thus there exists a letter $v$ such that, without loss of generality, it occurs in $z$, its inverse occurs in $w'$, and the two can be removed.
Note that if the inverse of $v$ occurred only in $z^{-1}$ then performing the reduction would yield a word $z'$, shorter than $z$, which satisfies the equation
\[ w = (z')^{-1} w' z'\]
over $A_\Gamma$. This contradicts the minimality of $z$.

Since $v$ can be removed,
we can perform a number of swaps to $w'$ and obtain a reduced word $ w'' v^{-1}$; we can do the same for $z$ and obtain a reduced word $v z' $. We now have the following equality in $A_\Gamma$
\[ w = z'^{-1} v^{-1} w'' v^{-1} v z' = z'^{-1} v^{-1} w''  z' \]
with $v^{-1} w''$ consisting of exactly the same letters as $w'$ (it is a cyclic conjugate of $ w'' v^{-1}$), and $z'$ shorter than $z$. We repeat this procedure until we obtain a reduced word. But then we know that it differs from $w$ by a sequence of swaps, and hence is a word in $\Sigma$. Thus $w'$ must have been a word in $\Sigma$ as well.
\end{proof}

Now we can prove the following proposition.

\begin{prop}
\label{prop: conj elemnts in RAAGs}
Let $\phi \in \Aut(A_\Gamma)$. Let $\Sigma \subseteq \Gamma$ be such that for all $x \in A_\Sigma$, the element $\phi(x)$ is conjugate to some element of $A_\Sigma$. Then there exists $y \in A_\Gamma$ such that
\[ \phi(A_\Sigma) \leqslant y^{-1} A_\Sigma y \]
\end{prop}
\begin{proof}
Let the vertex set of $\Sigma$ be  $\{v_1, \dots, v_m \}$; these letters are then generators of $A_\Sigma$.
Let $Z_\Gamma = H_1(A_\Gamma;\Z)$ denote the abelianisation of $A_\Gamma$, and let $Z_\Sigma \leqslant Z_\Gamma$ denote the image of $A_\Sigma$ in the abelianisation. Note that $Z_\Sigma$ is also generated by $\{v_1, \dots, v_m \}$ in a natural way.

Let $\phi_* \colon Z_\Gamma \to Z_\Gamma$,  be the induced isomorphism on abelianisations.
Now, by assumption, $\phi_*$ induces a surjection $Z_\Gamma /
Z_\Sigma  \to Z_\Gamma /  Z_\Sigma $; observe that $Z_\Gamma /  Z_\Sigma$ is
isomorphic to $\Z^n$ for some $n$, and such groups are Hopfian, so
this induced morphism is an isomorphism. Hence
$\phi_*\vert_{Z_\Sigma}$ is an isomorphism (since $\phi_*$ is), and so there exists an
element $w \in A_\Sigma$, such that its image in $Z_\Sigma$ is mapped
by $\phi_*$ to the element $v_1\cdots v_m$. This implies that $w$ is
mapped by $\phi$ to a conjugate (by some element $y^{-1}$) of a
cyclically reduced word $x$, which contains each letter
$v_i$. Crucially, an element $z \in A_\Gamma$ commutes with $x$ \iff
$z \in C(A_\Sigma)$ (as two reduced words define the same element \iff one can be obtained from the other by a sequence of swaps described above).

Consider $\psi =   c(y) \phi$, so that $\psi(w) = x$. We now aim to show that $\psi(A_\Sigma) \leqslant A_\Sigma$.

Suppose for a contradiction that there exists
$u \in A_\Sigma$ such that $\psi(u) \not\in A_\Sigma$.

It could be possible that $\psi(u) \in A_{\st(\Sigma)} = A_\Sigma \times A_{\lk(\Sigma)}$. But the only elements in $A_\Sigma \times A_{\lk(\Sigma)}$ conjugate to elements in $A_\Sigma$ are in fact the elements of $A_\Sigma$. Hence we can assume that $\psi(u) \not\in A_{\st(\Sigma)}$.

Since $\psi(u)$ is conjugate to an element in $A_\Sigma$, yet is not
in $A_{\st(\Sigma)}$, we can write
\[ \psi(u) = a^{-1} b^{-1} v^{-1} x' v b a \]
where the word is reduced, $a$ is a subword containing only letters in
\[Z(\Sigma) \ast \lk(\Sigma)\]
the subword $b$ contains only letters in $\Sigma \s- Z(\Sigma)$, the
letter $v$ does not lie in $\st(\Sigma)$, and $x'$ is any subword. We
will obtain a contradiction from this form of the word.

By assumption $\psi(wu) = x a^{-1} b^{-1} v^{-1} x' v b a $ is conjugate to an element of $A_\Sigma$. In particular it lies in $A_\Sigma$ after a sequence of reductions and cyclic reductions, by Lemma~\ref{lem: cyclically reduced word}. We are going to visualise the situation as follows.
We take a polygon with the number of vertices matching the length of
the word describing $\psi(wu)$; now we label the vertices by letters
so that going around the polygon clockwise and reading the labels
gives us a cyclic conjugate of our word.

In this picture, a reduction or cyclic reduction consists of a deletion of two
vertices $V_1,V_2$ labeled by the same letter but with opposite signs, and such that
there is a path between $V_1$ and $V_2$ whose vertices are only
labeled by letters commuting with the letter labeling $V_1$.

Recall that we have a finite sequence of basic moves which takes the word
\[ x a^{-1} b^{-1} v^{-1} x' v b a\] to a word in $A_\Sigma$.
Note that after every successive move we can still identify which part of our new word came from $x$, and which from $a^{-1} b^{-1} v^{-1} x' v b a$.

We claim that we can never remove two occurrences of letters in
\[a^{-1} b^{-1} v^{-1} x' v b a\]
 along a path lying in this
subword (and the same is true for the subword $x$)

Consider the first time we use a move which deletes the
occurrences of a letter and its inverse in (what remains of) the
subword $a^{-1} b^{-1} v^{-1} x' v b a$ along a path lying in this
subword. Let $q$ denote the letter we are removing. Since the subword
is reduced, such a move was not possible until a prior removal of a
letter $q'$ lying between the two letters we are removing, and such
that $q$ and $q'$ do not commute. But now $q'$ must have been removed
by a path not contained in our subword, since the removal of $q$ is
the first move of this type. Therefore the path used to remove $q'$
must contain $q$, and so $q$ and $q'$ must commute. This is a
contradiction which shows the claim.

\smallskip
It is clear that we can remove all letters in $a$ and $a^{-1}$ with a
path going over $x$; let us perform these moves first.

Since our finite sequence of moves takes us to a word in $A_\Sigma$, it must at some point remove the letters $v$ and $v^{-1}$. Consider the first move removing occurrences of this letter. Note that both these occurrences must lie in the subword $v^{-1} x' v$. Hence our move removes this occurrences along a path containing all of $x$. We have however assumed that $v$ does not commute with $x$, and so we must have first removed an occurrence of a letter $q$ from $x$, where $q$ does not commute with $v$.

In fact all occurrences of $q$ must be removed from $b x b^{-1}$ by reductions before we can remove $v$. This implies that $x$ contains at least two occurrences of $q$ and its inverse, which is a contradiction.
\end{proof}

\subsection{Automorphisms of a RAAG}
\label{subsec: aut(raag)}

Let us here briefly discuss a generating set for the group $\Aut(A_\Gamma)$.

By work of Servatius \cite{Servatius1989} and Laurence \cite{Laurence1995},
$\Aut(A_\Gamma)$ is generated by the following
classes of automorphisms:
\begin{enumerate}[i)]
\item Inversions
\item Partial conjugations
\item Transvections
\item Graph symmetries
\end{enumerate}
Here, an \emph{inversion} maps one generator of $A_\Gamma$ to its
inverse, fixing all other generators.

A \emph{partial conjugation} requires a vertex $v$ in $\Gamma$ whose
star disconnects $\Gamma$. For such a $v$, a partial conjugation is an
automorphism which conjugates all generators in one of the
complementary components of $\st(v)$ by $v$ and fixes all other generators.

A \emph{transvection} requires vertices $v, w$ with $\st(v)\supseteq
\lk(w)$. For such $v,w$, a transvection is the automorphism which maps
$w$ to $w v$, and fixes all other generators.
Transvections come in two types: \emph{folds} (or \emph{type I}) occur when \[\lk(v)\supseteq
\lk(w)\] and \emph{twists} (or \emph{type II}) when $v \in \lk(w)$.

The group $\mathrm{U{A}ut}(A_\Gamma)$ is defined to be the subgroup generated by all generators from our list except the twists. The group $\U(A_\Gamma)$ is its quotient by inner automorphisms.

A \emph{graph symmetry} is an automorphism of $A_\Gamma$ which permutes
the generators according to a combinatorial automorphism of $\Gamma$.

The group $\Aut^0(A_\Gamma)$ is defined to be the subgroup generated by generators of the first four types, i.e. without graph symmetries. Again, $\Out^0(A_\Gamma)$ is its quotient by inner automorphisms.

The group $\U^0(A_\Gamma)$ is defined to be the quotient by inner automorphisms of the subgroup $\mathrm{UAut}^0(A_\Gamma)$ of $\Aut(A_\Gamma)$ generated by all generators from our list except the twists and graph symmetries.

Note that $\U^0(A_\Gamma) = \U(A_\Gamma) \cap \Out^0(A_\Gamma)$, since conjugating any of our generators by a graph symmetry gives a generator of the same kind.

\begin{lem}
\label{lem: link non-cone}
Suppose that $\lk(v)$ is not a cone for all vertices $v$ of $\Gamma$. Then $\mathrm{U{A}ut}(A_\Gamma) = \Aut(A_\Gamma)$ and $\mathrm{U{A}ut^0}(A_\Gamma) = \Aut^0(A_\Gamma)$.
\end{lem}
\begin{proof}
It is enough to show that the assumption prohibits the existence of twists. Let us suppose (for a contradiction) that such a transvection exists; this is equivalent to assuming that there exist vertices $v$ and $w$ such that $v \in \lk(w) \subseteq \st(v)$. But in this case we have
\[ \lk(w) = (v \cap \lk(w)) \ast (\lk(v) \cap \lk(w)) = v \ast (\lk(w) \s- v) \]
which is a cone.
\end{proof}

\subsection{Markings}
\label{subsec: markings}
We also need the notion of marked topological spaces in the sense introduced
in~\cite{HenselKielak2016a}.
\begin{dfn}
We say that a path-connected topological space $X$ with a universal covering $\widetilde X$ is \emph{marked} by a group $A$ \iff it
comes equipped with an isomorphism between $A$ and the deck transformation group of $\t X$.
\end{dfn}
\begin{rmk}
Given a space $X$ marked by a group $A$, we obtain an isomorphism $A
\cong \pi_1(X,p)$ by choosing a basepoint $\widetilde p \in \widetilde
X$ (where $p$ denotes its projection in $X$). We adopt the notation
that the image of a point or set under the universal covering map will
be denoted as its \emph{projection}. To keep the notation uniform, we
will also call $X$ the projection of $\t X$.

Conversely, an isomorphism $A \cong \pi_1(X,p)$ together with a choice
of a lift $\widetilde p \in \widetilde X$ of $p$ determines the
marking in the sense of the previous definition.
\end{rmk}

\begin{dfn}
Suppose that we are given an embedding $\pi_1(X) \into \pi_1(Y)$ of fundamental groups of two path-connected spaces $X$ and $Y$, both marked. A map $\iota \colon X \to Y$ is said to \emph{respect the markings via the map $\widetilde \iota$} \iff $\widetilde \iota \colon \widetilde X \to \widetilde Y$ is $\pi_1(X)$-equivariant (\wrt the given embedding $\pi_1(X) \into \pi_1(Y)$), and satisfies the commutative diagram
\[ \xymatrix{ \widetilde X \ar[r]^{\widetilde \iota} \ar[d] & \widetilde Y \ar[d] \\
X \ar[r]^\iota & Y } \]

We say that $\iota$ \emph{respects the markings} \iff such an $\widetilde \iota$ exists.
\end{dfn}

To keep then notation (slightly) more clean, given a space $X_\Delta$ we will denote its universal cover by $\t X_\Delta$ (rather than $\t{X_\Delta}$). 

Next we describe a construction that allows us to glue two marked spaces.
\begin{lem}
\label{lem: markings}
Suppose that for each $i \in \{0,1,2\}$ we are given a group $A_i$ and
a space $X_i$ marked by this group. Suppose further that for each $i
\in \{1,2\}$ we are given a monomorphism $\phi_i \colon A_0 \into
A_i$, and a continuous embedding  $\iota_i \colon X_0 \into X_i$
which respect the markings via a map $\t \iota_i$.

Then there exists a group $A$, a space $X$ marked by $A$, and maps
$\phi'_i \colon A_i \to A$ and $\iota'_i \colon X_i \to X$, the latter
respecting the markings via maps $\t \iota'_i$, such that the
following diagrams commute
\[ \xymatrix{
A_0 \ar[r]^{\phi_1} \ar[d]^{\phi_2} & A_1 \ar[d]^{\phi'_1} & X_0 \ar[r]^{\iota_1} \ar[d]^{\iota_2} & X_1 \ar[d]^{\iota'_1} & \t X_0 \ar[r]^{\t \iota_1} \ar[d]^{\t \iota_2} & \t X_1 \ar[d]^{\t \iota'_1} \\
A_2 \ar[r]^{\phi'_2} & A & X_2 \ar[r]^{\iota'_2} & X & \t X_2 \ar[r]^{\t \iota'_2} & \t X
} \]
\end{lem}
We will refer to the construction in this lemma as \emph{obtaining $\t
  X$ from $\t X_1$ and $\t X_2$ by gluing $\im (\t \iota_1)$ to $\im
  (\t \iota_2)$}. Note that the projection $X$ of $\t X$ is in fact obtained by an honest gluing of the projections $X_1$ and $X_2$ along the respective subspaces.
\begin{proof}
We define $A$ and $X$ to be the push-outs of the appropriate
diagrams. Take a point $p \in X_0$, and its lift $\t p \in \t
X_0$. The point $p$ has its copies in $X_1, X_2$ and $X$; let $\t q$
denote some lift of $p$ in $\t X$. We also have copies of the point
$\t p$ in $\t X_1$ and $\t X_2$. We define the maps $\t \iota'_i$ to
be the unique maps satisfying the following commutative diagrams of
pointed spaces
\[ \xymatrix{
 \t X_i, \t p \ar[r]^{\t \iota'_i} \ar[d] & \t X, \t q \ar[d] \\
 X_i, p \ar[r]^{\iota_i}  &  X, p
} \]
where the vertical maps are the coverings. The verification that these
maps satisfy the third commutative diagram above is easy.
\end{proof}


\section{Relative Nielsen Realisation}
\label{sec: adapted realisation}

One of the crucial tools used in this article is the Relative Nielsen Realisation theorem, proven by the authors in~\cite{HenselKielak2016a}. In this section we will look at the necessary definitions, the statement of the theorem, and then a special case thereof.

\begin{dfn}
Let $A$ and $H$ be groups, and let $\phi \colon H \to \Out(A)$ be a
homomorphism. We say that a metric space $X$, on which $H$ acts by
isometries, \emph{realises} the action $\phi$ \iff $X$ is marked by
$A$, and the action of $H$ on conjugacy classes in $\pi_1(X) \cong A$
is equal to the one induced by $\phi$.

If $X$ is a (metric) graph or a cube complex, we require the action to respect the combinatorial structure as well.
\end{dfn}

\begin{rmk}
The action of $H$ on $X$ induces a group extension
\[ A \to \overline{A} \to H \]
When $A$ is centre-free, this extension carries precisely the same information as the action $H \to \Out(A)$; when $A$ has non-trivial centre however, then the extension contains more information.
\end{rmk}

\begin{dfn}
Suppose that $A$ contains a subgroup $A_1$ which is invariant under the (outer) action of $H$ (up to conjugation). Then the extension $\overline A$ contains a subgroup, such that the map $\overline A \to H$ restricted to this subgroup is onto $H$, and its kernel is equal to the subgroup. We call this subgroup  $\overline{A_1}$.
\end{dfn}

\begin{thm}[Relative Nielsen Realisation~{\cite[Theorem 5.4]{HenselKielak2016a}}]
\label{rel NR}
Let \[\phi \colon H \to \Out(A)\]
be a homomorphism with a finite domain, and let
\[A = A_1 \ast \dots \ast A_n \ast B\]
be a free-product decomposition, with $B$ a (possibly trivial) finitely generated free group, such that $H$ preserves the conjugacy class of each $A_i$.

Suppose that for each $i \in \{1, \dots, m\}$ we are given an NPC space $X_i$ marked by $A_i$, on which $H$ acts in such a way that the associated extension of $A_i$ by $H$ is isomorphic (as an extension) to the extension $\overline{A_i}$ coming from $\overline A$.
Then there exists an NPC space $X$ realising the action $\phi$, and such that for each $i \in \{1, \dots, m\}$ we have an $H$-equivariant embedding $\iota_i \colon X_i \to X$ which preserves the marking.

Moreover, the images of the spaces $X_i$ are disjoint, and collapsing each $X_i$ individually to a point yields a graph with fundamental group abstractly isomorphic to the free group $B$.
\end{thm}

Note that the above is a slightly restricted version of the Relative Nielsen Realisation, since in its most general form $H$ is allowed to permute the factors $A_i$.

\begin{rmk}[{\cite[Remark 5.6]{HenselKielak2016a}}]
\label{cubical rmk}
When each $X_i$ is a cube complex, we may take $X$ to be a cube complex too; the embeddings $X_i \into X$ are now maps of cube complexes, provided that we allow ourselves to cubically barycentrically subdivide the complexes $X_i$.
\end{rmk}

\smallskip
Now let us state a version of the above theorem stated for graphs.

Throughout, we consider only graphs without
vertices of valence 1, that is without leaves.
\begin{theorem}[Adapted Realisation]
\label{thm:adapted-realisation}
  Let $H$ be finite and $\phi \colon H \to \Out(F_n)$ be a homomorphism. Suppose that
  \[F_n = A_1 * \dots * A_k * B\]
  is a free splitting such that $\phi(H)$ preserves the conjugacy class of
  $A_i$ for each $1\leq i\leq k$. Let $\phi_i\colon H\to \Out(A_i)$ denote the
  induced actions. For each $i$ let $X_i$ be a marked metric graph
  realising $\phi_i$.

  Then there is a marked metric graph $X$ with the following properties.
  \begin{enumerate}[i)]
  \item $H$ acts on $X$ by combinatorial isometries, realising $\phi$.
  \item There are isometric embeddings $\iota_i \colon X_i \to X$ which respect the markings via maps $\widetilde \iota_i$.
  \item The embedded subgraphs $X_i$ are preserved by $H$, and the restricted action
  induces the action $\phi_i$ on $X_i$ up to homotopy.
  \end{enumerate}
  When $\pi_1(X_i) =A_i \not\cong \Z$, the map $\iota_i$ is actually $H$-equivariant (and not just $H$-equivariant up to homotopy).

We also arrange for the images $\iota_i(X_i)$ to be pairwise disjoint, unless we have $F_n = A_1 \ast A_2$, in which case we require $X = \im(\iota_1) \cup \im(\iota_2)$, and $\im(\widetilde \iota_1) \cap \im(\widetilde \iota_2)$ to be a single point. 
\end{theorem}

Note that since we only require the embeddings $\iota_i$ to be isometric, the presence of vertices of valence 2 is completely irrelevant. They might appear in the graphs $X_i$ as well as $X$ to make the action of $H$ combinatorial, but we do not require these appearances to agree under $\iota_i$.

\begin{proof}
Without loss of generality let us assume that the factors $A_i$ are non-trivial.

When $n=1$ we either have $F_n = B$, in which case we take $X$ to be the circle, or we have $F_n = A_1$, in which case we take $X = X_1$.

When $n \geqslant 2$, the group $F_n$ has trivial centre, and so the action $H \to \Out(F_n)$ yields a finite extension
\[ F_n \to \overline{F_n} \to H \]

For any $A_i \not\cong \Z$, the extension yielded by the action of $H$ on $X_i$ and the extension $\overline {A_i}$ are isomorphic as extensions. For convenience let us set $Y_i = X_i$ in this case.

When $A_i \cong \Z$, there exists an action of $H$ on a circle $S_1$ such that the induced extension agrees with $\overline {A_i}$. We define $Y_i$ to be precisely this circle with the $H$-action. Note that the action of $H$ on $Y_i$ and $X_i$ agree up to homotopy.

We now apply \cref{rel NR}, using the graphs $Y_i$.

When $F_n = A_1 \ast A_2$, we know that collapsing the images of $Y_1$ and $Y_2$ yields a graph with trivial fundamental group, and hence a tree. The preimage of this tree lifts is a forest in $X$ in such a way that each connected component of the forest intersects each $Y_i$ in at most one point. Collapsing each of these components individually to a point yields the result. 
\end{proof}

\begin{lem} \label{lem: fixed point in AR with two graphs}
In the context of Theorem~\ref{thm:adapted-realisation}, the
projection of the intersection point 
\[
\im(\widetilde \iota_1) \cap \im(\widetilde \iota_2)
\]
in $X$ is $H$-fixed (if it exists).
\end{lem}
\begin{proof}
The Seifert--van Kampen Theorem tells us that the intersection
\[\im(\iota_1) \cap \im(\iota_2)\]
 is simply-connected, and so in particular path connected. But this means that it can be lifted to the universal cover in such a way that the lift lies within
\[\im(\widetilde \iota_1) \cap \im(\widetilde \iota_2)\]
 which is just a singleton. Hence so is $\im(\iota_1) \cap \im(\iota_2)$. Now this point is the intersection of two $H$-invariant subspaces, and so is itself $H$-invariant, and thus $H$-fixed.
\end{proof}

\section{Systems of subgraphs and invariance}
\label{sec: systems of graphs}

\begin{dfn}
  Given a homomorphism $\phi \colon H \to \Out(A_\Gamma)$ we
  define \emph{the system of invariant subgraphs} $\mathcal L^\phi$
  to be the set of induced subgraphs $\Delta \subseteq \Gamma$
  (including the empty one) such that $\phi(H)$ preserves the
  conjugacy class of $A_\Delta \leqslant A_\Gamma$. Note that
  $\L^\phi$ is partially ordered by inclusion.
\end{dfn}

We next show that $\L^\phi$ is closed under taking intersections and
certain unions.

\begin{lem}
\label{lem: intersections in L}
Let $\Delta, \Sigma \in  \L^\phi$. Then
\begin{enumerate}[i)]
\item $\Delta \cap \Sigma \in \L^\phi$.
\item If $\lk(\Delta \cap \Sigma) \subseteq \st (\Delta)$ then $\Delta \cup
  \Sigma \in \L^\phi$.
\end{enumerate}
\end{lem}
\begin{proof}
  \begin{enumerate}[i)]
  \item Pick $h \in H$. Since $\Delta \in \L^\phi$, there exists a
    representative
    \[h_1 \in \Aut(A_\Gamma)\] of $\phi(h)$ such that $h_1(A_\Delta) =
    A_\Delta$. Analogously, there exists $h_2 \in \Aut(A_\Gamma)$
    representing $\phi(h)$ such that $h_2(A_\Sigma) = A_\Sigma$.

    Since $h_1$ and $h_2$ represent the same element in
    $\Out(A_\Gamma)$, we have
    \[ h_1^{-1} h_2 = c(r) \] with $r \in A_\Gamma$.

    Let $\Theta = \Delta \cap \Sigma$, and
    take $x \in A_\Theta$. Then $h_1(x) \in A_\Delta$, and so any
    cyclically reduced word in the alphabet $\Gamma$ representing the conjugacy class of
    $h_1(x)$ lies in $A_\Delta$ (by Lemma~\ref{lem: cyclically reduced word}). Now
    \[c(r) h_1(x) = h_2(x) \in A_\Sigma\] and thus any reduced word
    representing the conjugacy class of $h_1(x)$ lies in
    $A_\Sigma$. Hence such a word lies in $A_\Theta$. But this implies
    that $h_1(x)$ is a conjugate of an element of $A_\Theta$ for each $x \in A_\Theta$, and so
    we apply Proposition~\ref{prop: conj elemnts in RAAGs} and conclude that
\[ h_1(A_\Theta) \leqslant y^{-1} A_\Theta y \]
for some $y \in A_\Gamma$.

We repeat the argument for $h_1^{-1}$, which is a representative of $\phi(h^{-1})$, and obtain
\[ A_\Theta = h_1^{-1} h_1(A_\Theta) \leqslant h_1^{-1} (y^{-1} A_\Theta y) \leqslant y'^{-1} A_\Theta y'\]
for some $y' \in A_\Gamma$. Now Proposition~\ref{prop: ccv} tells us that $y'^{-1} A_\Theta y' = A_\Theta$ and so
both inequalities in the expression above are in fact equalities. Thus
\[ h_1(A_\Theta) = y^{-1} A_\Theta y \]
and so $\Theta \in \L^\phi$ as claimed.

    \item Now suppose that $\lk(\Theta) \subseteq \st
    (\Delta)$. Since $\Theta \in \L^\phi$, there exists a
    representative $h_3$ of $\phi(h)$ which fixes $A_\Theta$. Since
    $\Delta \in \L^\phi$, the subgroup $h_3(A_\Delta)$ is a conjugate
    of $A_\Delta$ by an element $r \in A_\Gamma$. Now we have
\[r A_\Theta r^{-1} \leq A_\Delta\]
 and so by Proposition~\ref{prop: ccv} we know that
\[r \in
    N(A_\Theta) N(A_\Delta) = A_{\st(\Theta)} A_{\st(\Delta)} =
    A_{\st(\Delta)}\]
 since $\lk(\Theta) \subseteq \st(\Delta)$ by
    assumption. But then $r^{-1} A_\Delta r = A_\Delta$, and so
    $h_3(A_\Delta) = A_\Delta$.

    Let us apply the same argument to $h_3(A_\Sigma)$ -- it must be a
    conjugate of $A_\Sigma$ by some $s \in A_\Gamma$, and we conclude
    as above that
    \[s \in N(A_\Theta) N(A_\Sigma) \leq A_{\st(\Delta)}
    A_{\st(\Sigma)}\] We take a new representative $h_4$ of $\phi(h)$,
    which differs from $h_3$ by the conjugation by the
    $A_{\st(\Delta)}$-factor of $s$. This way we get $h_4(A_\Delta) =
    A_\Delta$, and $h_4(A_\Sigma)$ equal to a conjugate of $A_\Sigma$
    by an element of $A_{\st(\Sigma)} = N(A_\Sigma)$. Hence we have
    $h_4(A_\Sigma) = A_\Sigma$, and the result follows. \qedhere
  \end{enumerate}
\end{proof}
The following lemma is very much motivated by the work of Charney--Crisp--Vogtmann~\cite{charneyetal2012}.

\begin{lem}
\label{lem: extended stars in L}
Suppose that $\phi (H) \leqslant \Out^0(A_\Gamma)$. Then $\mathcal L^\phi$ contains
\begin{enumerate}
\item each connected component of $\Gamma$ which contains at least one edge;
\item the extended star of each induced subgraph;
\item the link of each subgraph $\Delta$, such that $\Delta$ is not a cone;
\item the star of each subgraph in $\L^\phi$.
\end{enumerate}
\end{lem}
\begin{proof}
We will prove the first three points on our list for
$\Out^0(A_\Gamma)$ (and therefore for any subgroup).
It is enough to verify that each type of generator of $\Out^0(A_\Gamma)$ preserves the listed subgroups up to conjugacy. It is certainly true for all inversions, and thus we only need to verify it for transvections and partial conjugations.

We will make a rather liberal use of Lemma~\ref{lem: link and star calculus}.

\smallskip \noindent \textbf{Transvections.} Take two vertices in $\Gamma$, say $v$ and $w$, such that
\[\lk(w) \subseteq \st(v)\]
In this case we have a transvection $w \mapsto wv$. To prove our assertion we need to check that whenever $w$ belongs to the subgraph defining our subgroup, so does $v$.

\begin{enumerate}
\item If $w$ belongs to a connected component of $\Gamma$ which is not a singleton, then $\lk(w) \neq \emptyset$ lies in the same component, and so our assumption $\lk(w) \subseteq \st(v)$ forces $v$ to lie in the component as well.
\item Take $\Delta \subseteq \Gamma$ with $w \in \widehat \st(\Delta)$. If $w \in \lk(\lk(\Delta))$ then
\[\lk(\Delta) \subseteq \lk(w) \subseteq \st(v)\]
 and so $v \in \widehat \st(\Delta)$. If $w \in \lk(\Delta)$ then $\Delta \subseteq \lk(w) \subseteq \st(v)$ and so $v \in \st(\Delta)$.
\item Take $\Delta \subseteq \Gamma$ which is not a cone, and such that $w \in \lk(\Delta)$. Then $\Delta \subseteq \st(v)$, and so $v \in \st(\Delta)$ as above. However, if $v \in \Delta$, then $\Delta \s- \{v \}  \subseteq \st(v) \s- \{v\} = \lk(v)$, and so $\Delta$ is a cone over $v$, which is a contradiction. Thus $v \in \lk(\Delta)$.
\end{enumerate}
Note that in the last part the assumption of $\Delta$ not being a cone is used only to guarantee that $v \not\in \Delta$. If the transvection under consideration was of type I, we would now this immediately since $\Delta \subseteq \lk(w) \subseteq \lk(v)$ in this case, and so the assumption of $\Delta$ not being a cone would be unnecessary.

\smallskip \noindent \textbf{Partial conjugations.} Take a vertex $v$ in $\Gamma$, such that its star disconnects $\Gamma$. In this case we have a partial conjugation of the subgroup $A_\Theta$ by $v$, with $\Theta \subseteq \Gamma \s- \st(v)$ being a connected component. Let $\Sigma = \Gamma \s- (\Theta \cup \st(v))$. To show that this automorphism preserves the desired subgroups up to conjugation, we need to show that if the subgraphs defining the subgroups do not contain $v$, then they cannot intersect both $\Theta$ and $\Sigma$.

\begin{enumerate}
\item If a connected component intersects (and so contains) $\Theta$ but does not contain $v$, then it is in fact equal to $\Theta$, and so intersects $\Sigma$ trivially.
\item Take $\Delta \subseteq \Gamma$ with $v \not\in \widehat\st(\Delta)$, and such that $\widehat\st(\Delta)$ intersects $\Theta$ and $\Sigma$ non-trivially. If $\lk(\Delta) \subseteq \st(v)$, then $v \in \widehat \st(\Delta)$, which is a contradiction. Hence $\lk(\lk(\Delta))$ cannot intersect both $\Theta$ and $\Sigma$, since if it did, then we would have $\lk(\Delta) \subseteq \st(v)$.
    Similarly, if $\Delta \subseteq \st(v)$, then $v \in \st(\Delta)$. This is again a contradiction, and again we conclude that $\lk(\Delta)$ cannot intersect both $\Theta$ and $\Sigma$. Hence, without loss of generality, we have $\lk(\lk(\Delta))$ intersecting $\Theta$ and $\lk(\Delta)$ intersecting $\Sigma$. But then the two cannot form a join. This is a contradiction.
\item Take $\Delta \subseteq \Gamma$, with $\lk(\Delta)$ intersecting both $\Theta$ and $\Sigma$ non-trivially. This condition forces $\Delta \subseteq \st(v)$. In fact we see that $\Delta \subseteq \lk(v)$, since otherwise we would have $\lk(\Delta) \subseteq \st(v)$, and thus the link would intersect neither $\Theta$ nor $\Sigma$.
Now $\Delta \subseteq \lk(v)$ gives $v \in \lk(\Delta)$.
\end{enumerate}
Note that in the last part we did not use the assumption on $\Delta$ not being a cone.

\smallskip
Now we need to prove (4). Take $\Delta \in \L^\phi$. Pick $h \in H$ and let
\[h_1 \in \Aut(A_\Gamma)\]
be a representative of $\phi(h)$ such that $h_1(A_\Delta) = A_\Delta$. Then also
\[h_1(N(A_\Delta)) = N(A_\Delta) = A_{\st(\Delta)} \qedhere \]
\end{proof}

\begin{dfn}
We say that $\phi \colon H \to \Out(A_\Gamma)$ is \emph{link-preserving}
 \iff
$\L^\phi$ contains links of all induced subgraphs of $\Gamma$.
\end{dfn}
Note that if $\dim \Gamma = 1$, then every action $\phi$ is link-preserving.

\begin{rmk}
\label{rmk: link-preserving}
Note that when $\phi$ is link-preserving, then so is every induced action $H \to
\Out(A_\Sigma)$ for each $\Sigma \in \L^\phi$.
\end{rmk}

\begin{lem}
\label{lem: no adjacent transvections}
If $\phi(H) \leqslant \mathrm{U^0}(A_\Gamma)$ then $\phi$ is link-preserving.
\end{lem}
\begin{proof}
Let $\Delta \subseteq
\Gamma$ be given. Again we will show that each generator of $\mathrm{U^0}(A_\Gamma)$ sends $A_{\lk(\Delta)}$ to a conjugate of itself.  The inversions clearly have the desired property;
so do partial conjugations and transvections of type I, since the proof of (3) above did not
use the assumption on $\Delta$ not being a cone (as remarked). But these are the generators of $\mathrm{U^0}(A_\Gamma)$ and so we are done.
\end{proof}

\begin{cor}\label{cor:dim-2-case}
  If $\dim \Gamma = 2$ and $\Gamma$ has no leaves, then any $\phi \colon H \to \Out^0(A_\Gamma)$ is link-preserving.
\end{cor}
\begin{proof}
Since $\dim \Gamma =2$, the links of vertices are discrete graphs; since $\Gamma$ has no leaves, they contain at least 2 vertices. Hence such links are never cones, and Lemma~\ref{lem: link non-cone} tells us that
\[ \Out^0(A_\Gamma) = \mathrm{U^0}(A_\Gamma)\]
Lemma~\ref{lem: no adjacent transvections} completes the proof.
\end{proof}

\begin{dfn}
  \begin{enumerate}[i)]
  \item Any subset $\S$ of $\L^\phi$ closed under taking intersections of
    its elements will be called a \emph{subsystem of invariant
      subgraphs}.

  \item Given such a subsystem $\S$, and any induced subgraph $\Theta \in
    \L^\phi$, we define
    \begin{itemize}
    \item $\S_\Theta = \{ \Sigma \cap \Theta \mid \Sigma \in \S \} $
    \item $\bigcup \S = \bigcup \limits_{\Sigma \in \S} \Sigma$
    \item $\bigcap \S = \bigcap \limits_{\Sigma \in \S} \Sigma$
    \end{itemize}
  \end{enumerate}
\end{dfn}

\begin{lem}
Let $\P \subseteq \L^\phi$ be a subsystem of invariant graphs, and let $\Theta \in \P$. Then
\[ \P_\Theta = \{ \Delta \in \P \mid \Delta \subseteq \Theta \} \]
\end{lem}
\begin{proof}
This follows directly from the fact that $\P$ is closed under taking intersections.
\end{proof}

Note that given $\Delta \in \L^\phi$, we get an induced action $ \psi \colon H \to \Out(A_\Delta)$. This follows from the fact that the normaliser $N(A_\Delta)$ satisfies
\[ N(A_\Delta) = A_{\st(\Delta)} = A_\Delta \times A_{\lk(\Delta)}  \]
and $A_{\lk(\Delta)}$ centralises $A_\Delta$.

It is immediate that $\L^\psi = \L^\phi_\Delta$.

\subsection{Boundaries of subgraphs}

Let us record here a useful fact about boundaries of subgraphs.

\begin{dfn}
 Let $\Sigma \subseteq \Gamma$ be a subgraph. We define its \emph{boundary} $\partial \Sigma$ to be the set of all vertices of $\Sigma$ whose link is not contained in $\Sigma$.
\end{dfn}

\begin{lem}
\label{lem: boundary}
Suppose that $\Sigma$ is a maximal (\wrt inclusion) proper subgraph of $\Gamma$ such that $\Sigma \in \L^\phi$. Then for every $w \in \partial \Sigma$ we have $\Gamma \s- \Sigma \subseteq \lk(w)$.
\end{lem}
\begin{proof}
Let $\Theta = \Gamma \s- \Sigma$, and let $w \in \partial \Sigma$. We have $\widehat \st(w) \in \L^\phi$. Now
\[\lk(\widehat \st(w) \cap \Sigma) \subseteq \lk(w) \subseteq \widehat \st(w)\]
and so $\Sigma \cup \widehat \st(w) \in \L^\phi$ by Lemma~\ref{lem: intersections in L}. Thus the maximality of $\Sigma$ implies that $\Sigma \cup \widehat \st(w) = \Gamma$. Hence $\Theta \subseteq \widehat \st(w)$. Let $\Theta_1 = \Theta \cap \lk(w)$ and  $\Theta_2 = \Theta \cap \lk(\lk(w))$. If the latter is empty then we are done, so let us assume that it is not empty.

If $\widehat \st(w) = \st(w)$ then $\Theta_2 = \Theta \cap \{w\} = \emptyset$ and we are done. Otherwise
$\lk(\lk(w))$ is not a cone (since $w$ is isolated in it), and thus the link of $\lk(\lk(w))$ is contained in $\L^\phi$. But this triple link is in fact equal to $\lk(w)$, and so we have $\lk(w) \in \L^\phi$.
Since $\L^\phi$ is closed under taking intersections, we also have $\Sigma \cap \lk(w) \in \L^\phi$ and thus also $\st(\Sigma \cap \lk(w)) \in \L^\phi$. We have
\[ \Theta_2 = \Theta \cap \lk(\lk(w)) \subseteq \lk(\lk(w)) \subseteq \lk(\Sigma \cap \lk(w)) \subseteq \st(\Sigma \cap \lk(w))\]
and so $\st(\Sigma \cap \lk(w))$ is not contained in $\Sigma$.
We also have
\[\Sigma \cup \st(\Sigma \cap \lk(w)) \in \L^\phi\]
since \[\lk(\Sigma \cap \st(\Sigma \cap \lk(w))) \subseteq  \lk(\Sigma \cap \lk(w)) \subseteq \st(\Sigma \cap \lk(w))\]
 as before. The graph $\Sigma \cup \st(\Sigma \cap \lk(w))$ must contain $\Theta_1$ as well, by the maximality of $\Sigma$. Thus $\Theta_1 \subseteq \st(\Sigma \cap \lk(w))$, and so
 \[\Theta_1 \subseteq \lk (\lk_\Sigma(w))\]

We have $\Theta_1 \subseteq \lk(w) = \lk(\lk(\lk(w)))$ and so, combining this with the previous observation, we get
\begin{eqnarray*}
\Theta_1 & \subseteq& \lk\big(\lk_\Sigma(w)\big) \cap \lk\big(\lk(\lk(w))\big) \\
& =& \lk\Big(\lk(\lk(w)) \cup (\lk(w) \cap \Sigma)\Big) \\
& \subseteq& \lk\Big(\big(\lk(\lk(w)) \cup \lk(w)\big) \cap \Sigma\Big) \\
& =& \lk\big(\widehat \st(w) \cap \Sigma \big) \\
& \subseteq&  \st\big(\widehat \st(w) \cap \Sigma\big)
\end{eqnarray*}

The last subgraph is a star of a subgraph in $\L^\phi$, and so is itself in $\L^\phi$. As before we have
\[\lk\Big( \st\big(\widehat \st(w) \cap \Sigma\big) \cap \Sigma \Big) \subseteq \lk(w) = \lk_\Sigma (w) \cup \Theta_1 \subseteq \st\big(\widehat \st(w) \cap \Sigma\big)\]
and so $\st\big(\widehat \st(w) \cap \Sigma\big) \cup \Sigma \in \L^\phi$. It contains $\Sigma$ and  $\Theta_1$, and therefore it must contain all of $\Theta$. But then $\Theta \subseteq \st\big(\widehat \st(w) \cap \Sigma\big)$, which is only possible when
\[ \Theta \subseteq \lk\big(\widehat \st(w) \cap \Sigma\big) \subseteq \lk(w)\]
which is the desired statement.
\end{proof}

\section{Cubical systems}
\label{sec: cubical systems}

In this section we give the definition of the most fundamental object
in the paper. We then begin proving some central properties that will
be used throughout.

\begin{dfn}[Metric cube complex]
  \begin{enumerate}[i)]
  \item A \emph{metric cube complex} is a (realisation of a)
    combinatorial cube complex, which comes equipped with a metric
    such that every $n$-cube in $X$ (for each $n$) is isometric to a
    Cartesian product of $n$ closed intervals in $\R$.

  \item Given two such complexes $X$ and $Y$, we say that $Y$ is a
    \emph{subdivision} of $X$ \iff the combinatorial cube complex
    underlying $Y$ is a subdivision of $X$, and the induced map (on
    the realisations) $Y \to X$ is an isometry.

  \item We say that $Z \subseteq X$ is a \emph{subcomplex} of $X$ \iff
    there exists a subdivision $Y$ of $X$ such that, under the
    identification $X = Y$, the subspace $Z \subseteq Y$ is a
    subcomplex in the combinatorial sense.

  \item A connected metric cube complex is called \emph{non-positively curved}, or
    NPC, whenever its universal cover is CAT(0).
  \end{enumerate}
\end{dfn}
At this point we want to warn the reader that our notion of subcomplex
is more relaxed than the usual definition.

Note further that we will make no distinction between a metric cube complex $X$ and its subdivisions; as metric spaces they are isomorphic; moreover given any group action (by isometries) $H \curvearrowright X$ which respects the combinatorial structure of $X$ and any subdivision $Y$ of $X$, there exists a further subdivision $Z$ of $Y$ such that the inherited action of $H$ on $Z$ respects the combinatorial structure of $Z$.

\begin{dfn}[Cubical systems]
\label{dfn: cubical systems}
Suppose we have a subsystem of invariant subgraphs $\P$, such that $\P$ is closed under taking restricted links, that is for all $\Delta, E \in \P$ with $\Delta \subseteq E$ we have $\lk_E(\Delta)\in \P$. A \emph{cubical system} $\X$ (for $\P$) consists of the following data.
\begin{enumerate}
\item For each $\Delta \in \P$ a marked metric NPC cube complex $X_\Delta$, of the same dimension as $A_\Delta$, realising $H \to \Out(A_\Delta)$. We additionally require $X_\Delta$ not to have leaves when $\Delta$ is 1-dimensional.
\item For each pair $\Delta, \Theta \in \P$ with $\Delta \subseteq \Theta$, an $H$-equivariant isometric embedding
\[\iota_{\Delta, \Theta} \colon X_\Delta \times X_{\lk_\Theta(\Delta)} \to X_{\Theta}\] whose image is a subcomplex, and
 which respects the markings via a map $\widetilde \iota_{\Delta, \Theta}$, where the product is given the product marking.  
We set $\iota_{\Delta, \Delta}$ and $\widetilde \iota_{\Delta, \Delta}$ to be the respective identity maps.
\end{enumerate}
Given $\Delta, \Theta \in \P$ we will refer to the map $\widetilde
\iota_{\Delta, \Theta}\vert_{\widetilde{X}_\Delta \times \{x\}}$ for
any $x \in \widetilde{X}_{\lk_\Theta(\Delta)}$ as the \emph{standard
  copy of $\widetilde{X}_\Delta$ in $\widetilde{X}_\Theta$ determined
  by $x$}, or simply \emph{a standard copy of $\widetilde{X}_\Delta$
  in $\widetilde{X}_\Theta$}.

We say that such a standard copy is \emph{fixed} \iff its projection in $X_\Theta$ is $H$-invariant.

We require our maps to satisfy four conditions.

\smallskip \noindent \textbf{Product Axiom.} Given $\Delta, \Theta \in \P$ such that $\Theta = \Delta \ast \lk_\Theta(\Delta)$, we require
 $\t \iota_{\Delta, \Theta}$ to be surjective.

\smallskip \noindent \textbf{Orthogonal Axiom.} Given $\Delta \subseteq \Theta$, both in $\P$, for each $\t x \in \t  X_\Delta$ we require $\t \iota_{\Delta, \Theta}(\{\t x\} \times \t X_{\lk_\Theta(\Delta)})$ to be equal to the image of some standard copy of $\t X_{\lk_\Theta(\Delta)}$ in $\t X_\Theta$.

\smallskip \noindent \textbf{Intersection Axiom.} Let $\Sigma_1,
\Sigma_2, \Theta \in \P$ be such that $\Sigma_i \subseteq \Theta$ for both values of $i$. Suppose that we are given standard copies of $\widetilde X_{\Sigma_i}$ in $\widetilde X_{\Theta}$ whose images intersect non-trivially. Then the intersection of the images is equal to the image of a standard copy of $\widetilde X_{\Sigma_1 \cap \Sigma_2}$ in $\t X_\Theta$. Moreover, this intersection is also the image of a standard copy of  $\widetilde X_{\Sigma_1 \cap \Sigma_2}$ in $\t X_{\Sigma_i}$, under the given standard copy $\t X_{\Sigma_i} \to \t X_\Theta$, for both values of $i$

\smallskip \noindent \textbf{System Intersection Axiom.} Let $\S
\subseteq \P$ be a subsystem of invariant graphs closed under taking
unions of its elements.
 Then for each $\Sigma \in \S$
there exists a  standard copy of $\widetilde X_\Sigma$ in $\widetilde X_{\bigcup \S}$ such that the images of all of these copies intersect non-trivially.
\end{dfn}

\begin{rmk}
%
We will often make no distinction between a standard copy and its
image. Let us remark here that any standard copy is a
subcomplex (with our non-standard definition of a subcomplex; see above).
\end{rmk}

Now we will list some implications of the definition.

\begin{rmk}
Suppose we are given a cubical system $\X$ for $\L^\phi$, and let $\Delta \in \L^\phi$. Then the subsystem $\X_\Delta$, consisting of all complexes $X_\Sigma \in \X$ with $\Sigma \subseteq \Delta$ together with all relevant maps, is a cubical system for $\L^\phi_\Delta$.
\end{rmk}

\begin{lem}
\label{lem: extending intersections}
Let $\X$ be a cubical system for $\L^\phi$, and let $\P \subseteq
\L^\phi$ be a subsystem of invariant graphs which is closed under
taking unions. Suppose that $\P$ contains another subsystem $\P'$,
also closed under taking unions, such that $\Sigma \cup \bigcap \P' \in \P'$ for each $\Sigma \in \P$. Suppose further that for each $\Sigma' \in \P'$ we are given a standard copy $\widetilde Y_{\Sigma'}$ of $\widetilde X_{\Sigma'}$ in $\widetilde X_{\bigcup \P}$ such that
\[ \bigcap_{\Sigma' \in \P'} \widetilde Y_{\Sigma'} \neq \emptyset \]
Then for each $\Sigma \in \P \s- \P'$ there exists a standard copy $\widetilde Y_\Sigma$ of $\widetilde X_\Sigma$ in $\widetilde X_{\bigcup \P}$ such that
\[ \bigcap_{\Sigma \in \P} \widetilde Y_\Sigma \neq \emptyset \]
\end{lem}
\begin{proof}
Let us first set some notation: we let $\Delta = \bigcap \P'$, and for each $\Sigma \in \P$ we set $\Sigma' = \Sigma \cup \Delta \in \P'$.

The System Intersection Axiom gives us for each $\Sigma \in \P$ a
standard copy $\widetilde Z_\Sigma$ of $\widetilde X_\Sigma$ in
$\widetilde X_{\bigcup \P}$ such that there is a point $\widetilde b
\in \widetilde X_{\bigcup \P}$ with
\[ \widetilde b \in \bigcap_{\Sigma \in \P} \widetilde Z_\Sigma. \]
In particular $\widetilde b \in \widetilde Z_\Delta$, and thus there exists a point $\widetilde b' \in \widetilde Y_\Delta$ such that both $\t b$ and $\t b'$ lie in the same standard copy $\t W$ of $\t X_{\lk_{\bigcup \P}(\Delta)}$ (due to the Orthogonal Axiom). Note that the Intersection Axiom guarantees that $\widetilde b' \in \widetilde Y_{\Sigma'}$ for each $\Sigma' \in \P'$.

We want to construct standard copies $\t Y_\Sigma$ for each $\Sigma \in \P$ which contain $\t b'$.

Let $\delta$ denote the geodesic in the complete CAT(0) space $\widetilde X_{\bigcup \P}$ connecting $\widetilde b$ to $\widetilde b'$. Since all maps $\t \iota$ are isometric embeddings, the geodesic $\delta$ lies in $\t W$.

Take $\Sigma \in \P$. Then $\delta$ connects two standard copies of
$\t X_{\Sigma'}$, namely $\t Y_{\Sigma'}$ and $\t Z_{\Sigma'}$, and hence it
must lie in a standard copy of $\t X_{\st_{\bigcup \P}(\Sigma')}$;
this copy is unique since the link of $\st_{\bigcup \P}(\Sigma')$ in $\bigcup \P$ is trivial.
Thus the geodesic lies in the intersection of this copy and $\t W$, which itself is a standard copy of
\[\t X_{\lk_{\bigcup \P}(\Delta) \cap \st_{\bigcup \P}(\Sigma')} = \t X_{\lk_{\Sigma'}(\Delta) \ast \lk_{\bigcup \P}(\Sigma')} \]
by the Intersection Axiom. In particular, this implies that $\t b$ lies in this standard copy. 
But $\t b$ also lies in $\t Z_\Sigma$, and hence in the unique standard copy of $\t X_{\st_{\bigcup \P}(\Sigma)}$. Therefore $\t b$ lies in the intersection of this copy and the copy of $\t X_{\lk_{\Sigma'}(\Delta) \ast \lk_{\bigcup \P}(\Sigma')}$. However
\[\lk_{\Sigma'}(\Delta) \ast \lk_{\bigcup \P}(\Sigma')
\subseteq \st_{\bigcup \P}(\Sigma)\]
 (since $\Sigma \subseteq \Sigma'$ and $\Sigma' = \Sigma \cup \Delta$), and thus
the Intersection Axiom implies that the copy of $\t
X_{\lk_{\Sigma'}(\Delta) \ast \lk_{\bigcup \P}(\Sigma')}$ which
contains $\delta$ lies within the unique copy of $\t X_{\st_{\bigcup
    \P}(\Sigma)}$. This must be also true for $\delta$ itself, and so
there exists a standard copy $\t Y_\Sigma$ of $\t X_\Sigma$ which
contains $\t b'$, the other endpoint of $\delta$.
\end{proof}

\begin{lem}[Matching Property]
\label{lem: mathcing prop}
Let $\X$ be a cubical system for $\L^\phi$, and let $\Sigma_1,
\Sigma_2 \in \L^\phi$ be such that $\Theta = \Sigma_1 \cup \Sigma_2
\in \L^\phi$ as well. Let $\widetilde Y_i$ be a given standard copy of
$\widetilde X_{\Sigma_i}$ in $\widetilde X_\Theta$. Then their images
intersect in a standard copy of $\widetilde X_\Delta$, with $\Delta =
\Sigma_1 \cap \Sigma_2$.
\end{lem}
\begin{proof}
Let
\[ \P = \{ \Delta, \Sigma_1, \Sigma_2, \Theta \} \]
and
\[\P' = \{ \Sigma_1, \Theta \} \subseteq \P\]
Note that $\P'$ satisfies the assumptions of the previous lemma, and so there exists a standard copy $\t Z_{\Sigma_2}$ of $\t X_{\Sigma_2}$ in $\t X_\Theta$ such that $\t Y_{\Sigma_1}$ and  $\t Z_{\Sigma_2}$ intersect.

Now there exists a geodesic $\delta$ connecting standard copies $\t Z_{\Sigma_2}$ and $\t Y_{\Sigma_2}$, such that $\delta$ lies in a standard copy of $\t X_{\lk_\Theta(\Sigma_2)}$, and one of the endpoints of $\delta$ lies in $\t Y_{\Sigma_1} \cap \t Z_{\Sigma_2}$. But
\[ \lk_\Theta(\Sigma_2) \subseteq \Sigma_1 \]
and so $\delta$ lies in $\t Y_{\Sigma_1}$ entirely (by the Intersection Axiom). Thus $\t Y_{\Sigma_1}$ and $\t Y_{\Sigma_2}$ intersect non-trivially (at the other endpoint of $\delta$).
\end{proof}

\begin{lem}[Composition Property]
\label{lem: composition prop}
Let $\X$ be a cubical system for $\P$, and let $\Delta, \Sigma, \Theta \in \P$ satisfy $\Delta \subseteq \Sigma \subseteq \Theta$. Suppose that we are given a standard copy $\t Y_\Delta$ of $\t X_\Delta$ in $\t X_\Sigma$, and a standard copy $\t Z_\Sigma$ of $\t X_\Sigma$ in $\t X_\Theta$. Then there exists a standard copy of  $\t X_\Delta$ in $\t X_\Theta$, whose image is equal to the image of $\t Y_\Delta$ in $\t Z_\Sigma$.
\end{lem}

\begin{figure}
\label{fig: composition}
\begin{center}
\includegraphics[scale = 0.8]{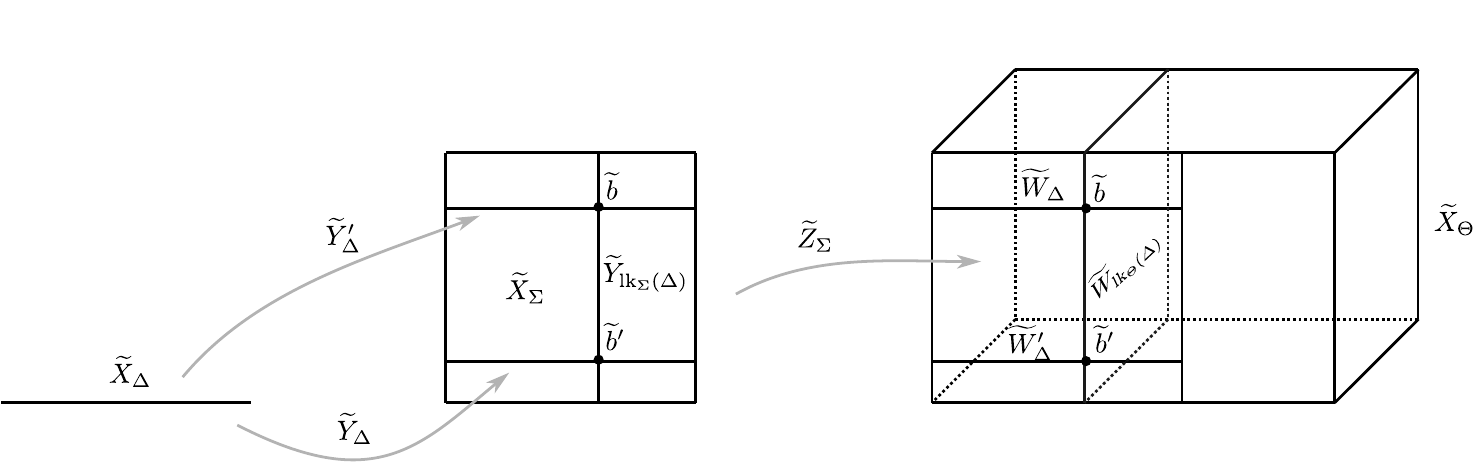}
\caption{A schematic for the proof of Lemma~\ref{lem: composition prop}}
\end{center}
\end{figure}

\begin{proof}
The System Intersection Axiom, Lemma~\ref{lem: extending intersections}, and the Intersection Axiom give us a standard copy $\t W_\Delta$ of $\t X_\Delta$ in $\t X_\Theta$ contained in $\t Z_\Sigma$, which is an image of a standard copy $\t Y'_\Delta$ of $\t X_\Delta$ in $\t X_\Sigma$. Now the Orthogonal Axiom gives us a standard copy $\t Y_{\lk_\Sigma(\Delta)}$ of $\t X_{\lk_\Sigma(\Delta)}$ in $\t X_{\Sigma}$ which intersects both $\t Y_\Delta$ and $\t Y'_\Delta$, the latter in a point $\t b$.

The Orthogonal Axiom also gives us a standard copy $\t W_{\lk_\Theta(\Delta)}$ of $\t X_{\lk_\Theta(\Delta)}$ in $\t X_{\Theta}$ which intersects $\t W_\Delta$ at the image of point $\t b$.
The Intersection Axiom implies that $\t W_{\lk_\Theta(\Delta)}$ intersects $\t Z_\Sigma$ in a copy of $\t X_{\lk_\Sigma(\Delta)}$, which is also the image of a copy of $\t X_{\lk_\Sigma(\Delta)}$ in $\t X_\Sigma$. But this copy contains $\t b$, and two standard copies of a given complex intersect \iff they coincide. Therefore it is equal to $\t Y_{\lk_\Sigma(\Delta)}$.
Thus it intersects $\t Y'_\Delta$ in a point $\t b'$, and the image of this point in $\t X_\Theta$ lies in
the image of $\t Y_\Delta$ in $\t Z_\Sigma$ and $\t W_{\lk_\Theta(\Delta)}$.

The Orthogonal Axiom gives us a copy $\t W'_\Delta$ of $\t X_\Delta$ in $\t X_\Theta$ which contains $\t b'$. Hence $\t W'_\Delta$ intersects $\t Z_\Sigma$, and the Intersection Axiom implies that this intersection is a standard copy equal to the image of a standard copy of $\t X_\Delta$ in $\t X_\Sigma$. But this standard copy intersects $\t Y_\Delta$ in $\t b'$, and so coincides with $\t Y_\Delta$. This finishes the proof.
\end{proof}

\begin{dfn}
Let $\Delta \subseteq \Sigma$ be two elements of $\L^\phi$, and suppose that $\X'$ is a cubical system for $\L^\phi_\Delta$. Let $\Delta = \Delta_1 \ast \dots \ast \Delta_k$ be its join decomposition. We say that a cubical system $\X$ for $\L^\phi_\Sigma$ \emph{extends} $\X'$ \iff
\begin{itemize}
 \item when $\vert \Delta_i \vert \geqslant 2$, for every $E \in \L^\phi_{\Delta_i}$ we have an $H$-equivariant isometry $j_E \colon X'_E \to X_E$ which preserves the markings via a map $\widetilde j_E$;
\item when $\Delta_i$ is a singleton we have an isometry $j_{\Delta_i} \colon X'_{\Delta_i} \to X_{\Delta_i}$, which preserves the markings via a map $\widetilde j_{\Delta_i}$, and is $H$-equivariant up to homotopy;
\item the maps $\widetilde \iota', \widetilde \iota$ and $\widetilde j$ make the obvious diagrams commute.
\end{itemize}

We say that $\X$ \emph{strongly extends} $\X'$ \iff all the maps $j$ are $H$-equivariant.
\end{dfn}


\begin{prop}
\label{prop: product system}
Suppose that $\Gamma = \Gamma_1 \ast \Gamma_2$ with $\Gamma_i \in \L^\phi$ for both values of $i$. Let $\X^i$ be a cubical system for $\L^\phi_{\Gamma_i}$. Then there exists a cubical system $\X$ for $\L^\phi$ which extends both $\X^1$ and $\X^2$ strongly.
\end{prop}
\begin{proof}
This construction is fairly straightforward. Given $\Sigma \in
\L^\phi$ we define $\Sigma_i = \Sigma \cap \Gamma_i$. Note that we
have $\Sigma = \Sigma_1 \ast \Sigma_2$.

Take such a $\Sigma$. We define $\t X_\Sigma = \t X^1_{\Sigma_1}
\times \t X^2_{\Sigma_2}$. We mark it with the product marking, and
immediately see that $\t X_\Sigma$ and its projection $X_\Sigma$ are of
the form required in the definition of a cubical system.

Now let us also take $\Theta \in \L^\phi$ such that $\Sigma \subseteq \Theta$. Crucially,
\[\lk_\Theta(\Sigma) = \lk_{\Theta_1}(\Sigma_1) \ast \lk_{\Theta_2}(\Sigma_2)\]

We define $\t \iota_{\Sigma, \Theta} = \t \iota_{\Sigma_1, \Theta_1}
\times \t \iota_{\Sigma_2, \Theta_2}$. Again, it is clear that these
maps have the form required in the definition.

The four axioms are immediate; they all follow from the observation that any standard copy of $\t X_\Sigma$ in $\t X_\Theta$ for any $\Sigma, \Theta \in \L^\phi$ with $\Sigma \subseteq \Theta$ is equal to a product of standard copies of $\t X_{\Sigma_i}$ in $\t X_{\Theta_i}$ for both values of $i$, and vice-versa -- taking a product of two such copies yields a copy of the former kind.
%
%
%
%
\end{proof}

We will say that the cubical system $\X$ obtained above is the \emph{product} of systems $\X^1$ and $\X^2$.

\begin{lem}
\label{lem: replacing circle}
Suppose that $\Gamma$ is a cone over $s \in \Gamma$, with $\{s\}, \Gamma \s- \{s\} \in \L^\phi$. Let $\X'$ be a cubical system for $\L^\phi_{\{s\} }$, and let $\X$ be a cubical system for $\L^\phi$ which extends $\X'$. Then there exists a cubical system $\X''$ for $\L^\phi$ which extends $\X$ and extends $\X'$ strongly.
\end{lem}
\begin{proof}
We define $\X''$ to be the product of $\X_{\Gamma \s- \{s\}}$ and
$\X'$. It is then clear that $\X''$ extends $\X'$ strongly. To verify
that $\X''$ extends $\X$ we only need to observe that
$X''_{\{s\}} = X'_{\{s\}}$, which in turn is isometric to $X_{\{s\}}$
in a way $H$-equivariant up to homotopy since $\X$ extends $\X'$.
\end{proof}

\begin{lem}
 \label{replacing circle for extensions}
Let $\X$ be a cubical system for $\L^\phi$, and suppose that the action $\phi$ is link-preserving, and its domain $H$ is finite.
Suppose that we have a group extension
\[ A_\Gamma \to \overline H \to H \]
yielding the given action $\phi$.
Then there exists a cubical system $\Y$ for $\L^\phi$ which extends $\X$, and such that the extension of $A_\Gamma$ by $H$ given by the action of $H$ on $X_\Gamma$ is isomorphic to $\overline H$ (as an extension).
\end{lem}
\begin{proof}
Let $\Gamma = \Gamma_1 \ast \dots \ast \Gamma_m$ be a join decomposition of $\Gamma$, and suppose that $\Gamma_i$ is a  singleton \iff $i \leqslant k$.
Then $E = \Gamma_{i+1} \ast \dots \ast \Gamma_m \in \L^\phi$, since $\phi$ is link-preserving, and so we have a cubical system $\Y^1 = \X_E$ for $\L^\phi_E$.

We also have $\Delta = \Gamma_1 \ast \dots \ast \Gamma_k \in \L^\phi$, and thus the extension $\overline H$ yields an extension
\[ A_\Delta \to \overline {A_\Delta} \to H \]
We easily build a cubical system $\Y^2$ for $\L^\phi_\Delta$ such that the action of $H$ on $Y^2_\Delta$ gives $\overline {A_\Delta}$ -- it is enough to do it for $\Delta$ being a singleton, in which case we only need to observe that any finite extension of $\Z$
acts on $\R$ by isometries preserving the integer points and without a global fixed point, and taking the quotient by $\Z$ yields an action of $H$ on a circle which gives the desired extension.

Note that $\Y^2$ extends $\X_\Delta$.
\cref{prop: product system} finishes the proof.
\end{proof}

\section{Cubical systems for free products}
\label{sec: AR for RAAGs}

In this section we use Relative Nielsen Realisation (\cref{rel NR}) to build cubical systems for RAAGs whose defininig graph $\Gamma$ is disconnected. The construction will assume that we have already built cubical systems for some unions of connected components of $\Gamma$.

\begin{prop}
\label{prop: sticking complexes together}
Suppose that $\Gamma$ is the disjoint union 
\[\Gamma = \Sigma_1 \sqcup \dots \sqcup \Sigma_n \sqcup \Theta\]
 where $\Sigma_i \in \L^\phi$, each $\Sigma_i$ is a union of connected components, and $\Theta$ is discrete. Let $\X^i$ be a cubical system for $\L^\phi_{\Sigma_i}$ for each $i$. Then there exists a cubical system $\X$ for $\L^\phi_{\Sigma_1} \cup \dots \cup \L^\phi_{\Sigma_n} \cup \{\Gamma\}$ extending each $\X_i$.

Moreover, the (unique) standard copies in $\t X_\Gamma$ of $\t X_{\Sigma_i}$ and $\t X_{\Sigma_j}$ are disjoint when $i \neq j$, unless $\Gamma = \Sigma_1 \sqcup \Sigma_2$, in which case their standard copies intersect in a single point, and the union of their projections covers $X_\Gamma$. 
\end{prop}
\begin{proof}
If each $\Sigma_i$ is empty, then the classical Nielsen Realisation for graphs yields the result. We may thus assume that $\Sigma_1$ is not empty. We may also assume that $\Gamma \neq \Sigma_1$. Therefore we may assume that $A_\Gamma$ has no centre, and hence the action $H \to \Out(A_\Gamma)$ gives us an extension
\[ A_\Gamma \to \overline H \to H \]
This extension in turn gives us extensions
\[ A_{\Sigma_i} \to \overline{ A_{\Sigma_i}} \to H \]
for each $i$.

We use \cref{replacing circle for extensions} and modify each $\X^i$ so that these extensions agree with the ones given by the action of $H$ on $X^i_{\Sigma_i}$.

We now apply \cref{rel NR}, together with \cref{cubical rmk}, using the cube complexes $X^i_{\Sigma_i}$ as input. This way we construct the cube complex $X_\Gamma$.
Now we define the cubical system $\X$.

Given $\Sigma \in \L^\phi_{\Sigma_i}$ we define $X_\Sigma = X^i_\Sigma$. We have already defined $X_\Gamma$. Since the subgraphs $\Sigma_i, \Sigma_j$ are disjoint when $i \neq j$, there are no choices involved in this definition.

Given $\Sigma, \Sigma' \in \L^\phi_{\Sigma_i}$ we set $\t \iota_{\Sigma, \Sigma'}$ to be the corresponding map in $\X^i$. Given $\Sigma \in \L^\phi_{\Sigma_i}$ we define
\[ \t \iota_{\Sigma, \Gamma} = \t \iota_{\Sigma_i, \Gamma} \circ \t \iota_{\Sigma, \Sigma_i} \]
This map is of the desired kind since $\lk(\Sigma) \subseteq \Sigma_i$, as this last subgraph is a union of components of $\Gamma$. 
We also set $\t \iota_{\Gamma, \Gamma}$ to be the identity.

What remains is the verification of the four axioms.

\subsubsection*{Product and Orthogonal Axioms}
Suppose that $\Sigma' = \st_{\Sigma'}(\Sigma)$, with 
\[\Sigma, \Sigma' \in \L = \L^\phi_{\Sigma_1} \cup \dots \cup \L^\phi_{\Sigma_n} \cup \{\Gamma\}\]
Then either $\Sigma = \Sigma'$, which is the trivial case, or $\Sigma \subseteq \Sigma_i$ for some $i$. But then also
$\Sigma' \subseteq \Sigma_i$, since $\Sigma_i$ is a union of connected components. In this latter case we only need the Product or Orthogonal Axiom in $\X^i$.

\subsubsection*{Intersection Axiom}
Take
$\Sigma, \Sigma', \Omega \in \L$ such that $\Sigma \subseteq
\Omega$ and $\Sigma' \subseteq \Omega$, 
and let $\t Y_{\Sigma}$ and $\t Y_{\Sigma'}$ be standard copies of, respectively, ${\t
  X}_{\Sigma}$ and $\t X_{\Sigma'}$ in ${\t X}_\Omega$ with non-empty intersection. We need
to show that the intersection is the image of a standard copy of
$\Sigma \cap\Sigma'$ in each.

The situation becomes trivial when $\Sigma = \Omega$ or $\Sigma' = \Omega$, so let us assume neither of these situations occurs.

If there exists an $i$ such that $\Sigma, \Sigma' \subseteq \Sigma_i$, then either $\Omega \subseteq \Sigma_i$, in which case we use the Intersection Axiom of $\X^i$, or $\Omega = \Gamma$, in which case we use the Intersection Axiom for the triple $\Sigma, \Sigma', \Sigma_i$, noting that the standard copies $\t Y_{\Sigma}$ and $\t Y_{\Sigma'}$ are just images of standard respective standard copies in $\t X_{\Sigma_i}$.

The remaining case occurs when $\Sigma \subseteq \Sigma_i$ and $\Sigma' \subseteq \Sigma_j$ for $i \neq j$. Then $\Sigma \cap \Sigma' = \emptyset$, and so we need to show that $\t Y_{\Sigma} \cap \t Y_{\Sigma'}$ is a single point. Since the standard copies intersect non-trivially, it means that the images of $Y_i$ and $Y_j$ intersected. Thus we must have had $\Gamma = \Sigma_i \sqcup \Sigma_j$, and $\t Y_{\Sigma} \cap \t Y_{\Sigma'}$ is precisely the unique point at which the images of $\t Y_i$ and $\t Y_j$ intersect.

\subsubsection*{System Intersection Axiom}
Take a subsystem $\P \subseteq \L$ closed under taking unions. Suppose that we have
$\Sigma, \Sigma' \in \P$ with $\Sigma \subseteq \Sigma_i$ and $\Sigma' \subseteq \Sigma_j$ with $i \neq j$. Then $\Sigma \cup \Sigma' \in \P \subseteq \L$, and so $\Sigma \cup \Sigma' = \Gamma$.
An analogous reasoning immediately implies that $\P = \{ \emptyset, \Sigma_i, \Sigma_j, \Gamma\}$. Now the last three subgraphs have unique corresponding standard copies, and these intersect at a single point; we take this point to be the standard copy of $\t X_\emptyset$ and the argument is finished.

If all subgraphs in
$\P \s- \{ \Gamma\}$ lie in $\Sigma_i$ for some $i$, then
we are done by the System Intersection Axiom of $\X^i$ applied to $\P \s- \{ \Gamma\}$ -- either this covers all of $\P$, or we need to observe that standard copies in $\t X_\Gamma$ are images of copies in $\t X_{\Sigma_i}$.
\end{proof}

\section{Gluing}
\label{sec: gluing}

In this section we deal with the situation in which we are given cube complexes realising the induced action on $A_{\Sigma}$ and $A_{\Delta}$, where $\Sigma \cap \Delta \neq \emptyset$, and we build a cube complex realising the action on $A_{\Sigma \cup \Delta}$.

Suppose that $\phi \colon H\to\Out(A_\Gamma)$ is link-preserving, and that $H$ is finite.
Let $\Sigma, \Theta \in \L^\phi$ be such that $\Sigma \cup \Theta = \Gamma$. Take $E = \Sigma \cap \Theta$. Let us set $E' = E \s- Z(E)$, and $Z(E) = \{ s_1, \dots, s_k \}$. Since $\phi$ is link-preserving, we have $E' \in \L^\phi$ and $\{s_i\} \in \L^\phi$ for each $i$.

Suppose that we have cubical systems $\X$ and  $\X'$ for $\L^\phi_\Sigma$ and $\L^\phi_\Theta$ respectively, such that $\X'$ extends $\X_E$.

The main goal of this section is to show that (under mild assumptions)
one can equivariantly glue $X_\Sigma$ to $X'_\Theta$ so that the result
realises the correct action $\phi:H\to\Out(A_\Gamma)$. This will be done in two steps;
first showing that the loops $X_{s_i}$ and $X'_{s_i}$ are actually equal (and hence that $\X'$ extends $\X_E$ strongly), and
then constructing a gluing.

We will repeatedly use the following construction.

\begin{dfn}
\label{dfn: geometric rep}
Given a cube complex $X_\Gamma$ realising an action \[\phi \colon H \to \Out(A_\Gamma)\] we say that an element $h_p \in \Aut(A_\Gamma)$ is a \emph{geometric representative} of $h\in H$ \iff it is obtained by the following procedure: take a basepoint $\t p \in \t X_\Gamma$ with a projection $p \in X_\Gamma$, and a path $\gamma$ from $p$ to $h.p$. The choice of $\t p$ induces an identification $\pi(X_\Gamma,p) = A_\Gamma$. We now take $h_p \in \Aut(\pi_1(X_\Gamma,p))$ to be the automorphism induced on the fundamental group by first applying $h$ to $X_\Gamma$, and then pushing the basepoint back to $p$ via $\gamma$.
\end{dfn}

Suppose that $H$ acts on a graph (without leaves) of rank 1 (i.e. on a subdivided circle). Let us fix an orientation on the circle.

\begin{dfn}
Given an element $h \in H$ we say that it \emph{flips} the circle \iff
it reverses the circle's orientation; otherwise we say that it
\emph{rotates} the circle. In the latter case we say that it rotates
by $k$ \iff the simple path from some vertex to its image
under $h$, going along the orientation of the circle, has
combinatorial length $k$.
\end{dfn}

\begin{lem}
\label{lem: rotating circles}
Suppose that $\Sigma \neq \st_\Sigma(s_i)$ and $\Theta \neq
\st_\Theta(s_i)$ for some $i$. Then $X_{\{s_i\}}$ and $X'_{\{s_i\}}$ are
$H$-equivariantly isometric.
\end{lem}
\begin{proof}
To simplify notation set $s = s_i$, $Y = X_{s_i}$ and $Z = X'_{s_i}$. Note that $Y$ and $Z$ are isometric.
Let us subdivide the edges of $Y$ and $Z$ so that both actions of $H$ are combinatorial, and so that the two loops can be made combinatorially isomorphic.

Let $m$ denote the
number of vertices in the subdivided loop $Y$.
Fix an orientation on both $Y$ and $Z$ so that going around the loops once in the positive direction yields $s \in A_\Gamma$.

We first focus on those $h\in H$ which
map the conjugacy class of $s \in A_\Gamma$ to itself.
Then $h$ acts on $Y$ and $Z$ as a rotation. We
claim that the two actions of $h$ rotate by the same number of
vertices.

Consider a representative $h_0 \in \Aut(A_\Gamma)$ of $\phi(h)$ which
 preserves $s$. For any such representative we have $h_0^{\ord(h)}$ equal to a conjugation which fixes $s$. Hence
\[h_0^{\ord(h)} = c(s^{K(h_0)} t_0   )\]
where \[t_0 \in A_{\lk(s)}\]
We know that $s \not\in Z(\Gamma)$ (otherwise
$\Sigma = \st_\Sigma(s)$ which contradicts our assumption).
Thus, the integer $K(h_0)$ is unique.

Since $\lk(s) \in \L^\phi$, the subgroup $h_0(A_{\lk(s)})$ is conjugate to $A_{\lk(s)}$. But the former subgroup must centralise $h_0(s) = s$, and so
\[ h_0(A_{\lk(s)}) \leqslant A_{\st(s)} \]
The only subgroup of $A_{\st(s)} = A_s \times A_{\lk(s)}$ conjugate to $A_{\lk(s)}$ is $A_{\lk(s)}$ itself, and therefore $h_0(A_{\lk(s)}) = A_{\lk(s)}$.

Let $h_1$ and $h_2$ be representatives of $\phi(h)$. There exists a unique integer $l$
such that
\[h_2 = c({s^l} t)\circ h_1\]
with $t \in A_{\lk(s)}$.
Note that
\[ c(x) \circ h_1 = h_1 \circ c(h_1^{-1}(x)) \]
and so, in particular, using $h_1(s)=s$ and $h_1(A_{\lk(s)}) = A_{\lk(s)}$, we get
\[ (c({s^l} t)\circ h_1)^{\ord(h)} = c(s^{\ord(h) l} t') \circ h_1^{\ord(h)}\]
where $t' \in A_{\lk(s)}$.
Thus
\begin{eqnarray*} c({s^{K(h_2)}} t_2) &=& h_2^{\ord(h)} \\ &=& (c({s^l} t)\circ h_1)^{\ord(h)} \\ &=&
c({s^{\ord(h)l}} t')\circ h_1^{\ord(h)} \\ &=& c({s^{\ord(h)l+K(h_1)}}t' t_1) \end{eqnarray*}
This shows that $K(h_1)$ mod $\ord(h)$
is independent of the representative, and so we can define $K(h) \in \Z / \ord(h) \Z$ in the obvious way.
This algebraic invariant will be the main tool in showing that $Y$ and $Z$ are $H$-equivariantly isometric.

Fix a
basepoint $p$ in $X_\Sigma$ lying in $\im(\iota_{E, \Sigma})$, and a basepoint $q$ in $X_\Theta$ lying in $\im(\iota_{E, \Theta})$.

Let $h_p \in \Aut(A_\Sigma)$ be the
geometric representative of $h$ (using the action of $H$ on $X_\Sigma$), obtained by taking the basepoint $p$ and a path $\gamma$
inside
\[\im(\iota_{E, \Sigma}) \cong \prod_{i=1}^k X_{\{s_i\}} \times X_{E'}\]
 which first travels orthogonally to $Y = X_{\{s_i\}}$, and then along the copy of $Y$ containing $p$ (in the negative direction).

If $h$ rotates $Y$ by $\mu$ vertices (in the positive direction), then
$(h_p)^{\ord(h)}$ is equal to the conjugation by $s^{\ord(h) \mu/m} t$ for some $t \in A_{\lk_\Sigma(s)}$. Since $\Sigma \neq \st_\Sigma(s)$, the number $\ord(h) \mu/m$ is unique. Hence, by taking any representative of $\phi(h)$ in $\Aut(A_\Gamma)$ which restricts to $h_p$ on $A_\Sigma$, we see that
\[K(h) = \ord(h) \mu /m \textrm{ mod } \ord(h)\]

Now we define $h_q$ in the analogous manner using $X_\Theta$ instead of $X_\Sigma$. Since $h_p$ and $h_q$ represent the same element $h$, the computation above shows that they rotate by the same number of vertices.
We have thus dealt with elements $h \in H$ which fix the conjugacy class of $s$.

If $h$ maps the conjugacy class of $s$ to the conjugacy class of $s^{-1}$, then $h$ must flip both $Y$ and $Z$, and therefore must have two fixed points on each loop. If another
element $g \in H$ flips $Y$, then $hg$ rotates $Y$,
and by the above rotates $Z$ by the same number of vertices. This implies that the fixed
points of $h$ and $g$ on $Y$
differ by the same number of vertices as the respective
fixed points on $Z$.

Hence there exists an identification between $Y$ and $Z$ which is $H$-equivariant.
\end{proof}

Now suppose that $\X'$ extends $\X_E$ strongly. Suppose further that we have fixed standard copies of $\t X_E$: one in $\t X_\Sigma$, called $\t P$ , and one in $\t X'_\Theta$, called $\t Q$. We can from a cube complex marked by $A_\Gamma$ from $\t X_\Sigma$ and $\t X'_\Theta$ by gluing $\t P$ and $\t Q$. Note that this is in general not unique, as there might be more than one marking-respecting isometry of $\t P$ and $\t Q$ such that the projections $P$ and $Q$ become $H$-equivariantly isomorphic. 

Let $\t Y$ denote the glued-up complex, and let $Y$ denote its
projection. Our gluing gives us an action of $H$ on the projection
$Y$. This induces an action $H \to \Out(A_\Gamma)$ in the obvious way;
but this action is in general not equal to $\phi$. We are now going to measure the difference of these two actions.

Let us choose a point $\t p \in \t P$ as a basepoint. For each $h \in H$ we choose a path $\gamma(h)$ in $P$ connecting $p$ to $h.p$. Since $P$ and $Q$ are standard copies of the same complex, they are isomorphic via a fixed isomorphism. This gives us a copy of $\t p$ and $\gamma(h)$ in $\t Q$ and $Q$ respectively; let us denote the former by $\t q$ and the latter by $\gamma'(h)$. The points $p$ and $q$ are naturally points in $Y$.

Now let $h_p$ and $h_q$ denote the respective geometric representatives of $h$. The former restricts to the same automorphism as $\phi(h)$ on $A_\Sigma$, the latter on $A_\Theta$. They represent the same outer automorphism, and agree on $A_E$. Thus we have
\[ h_p h_q^{-1} = c(x(h))\]
for some $x(h) \in C(A_E)$.

\begin{dfn}
We say that the gluing above is \emph{faulty within} $G \leqslant C(A_E)$ \iff  $x(h) \in G$ for all $h$ (with our choices of $\t p$ and $\gamma(h)$).
\end{dfn}

\begin{prop}[Gluing Lemma]
\label{prop: gluing}
Suppose that $\t Y$ is faulty within $Z(A_E)$. Then there exists another gluing as above, $\t X$, such that its projection $X$ realises $\phi$.
\end{prop}
\begin{proof}
First let us note that when $Z(\Gamma) = Z(E)$ then all the conjugations $c(x(h))$ are trivial, and $\t Y$ is already as desired. We will henceforth assume that $Z(\Gamma) \neq Z(E)$.

Take $h \in H$.
By assumption, the gluing $\t Y$ gives us geometric representatives $h_p$ and $h_q$ such that
\[ h_p h_q^{-1} = c(x(h))\]
with $x(h) \in Z(A_E) = A_{Z(E)}$.

We assume that $x(h) \in A_{Z(E) \s- Z(\Gamma)}$ unless $x(h)$ is the identity; we can always do this since conjugating by elements in $A_{Z(\Gamma)} = Z(A_\Gamma)$ is trivial.
Now we define $h'_p \in \Aut(A_\Gamma)$ so that $h_p' h_p^{-1}$ is a conjugation by an element of $Z(A_\Sigma)$ and that $h_p' h_q^{-1}$ is equal to a conjugation by an element in $A_{Z(E) \s- (Z(\Gamma) \cup Z(\Sigma))}$; we further define  $h'_q \in \Aut(A_\Gamma)$ so that $h_q' h_q^{-1}$ is a conjugation by an element of $Z(A_\Theta)$ and that
\[h_p' h_q'^{-1} = c(x'(h))\]
with $x'(h) \in A_{Z(E) \s- (Z(\Gamma) \cup Z(\Sigma) \cup Z(\Theta))}$.
Since $Z(\Sigma) \cap Z(\Theta) = Z(\Gamma)$, the elements $h'_p$ and $h'_q$ are unique.

Note that $h_p$ and $h'_p$ are identical when restricted to $A_\Sigma$; the analogous statement holds for $h_q$ and $h'_q$ restricted to $A_\Theta$.

Consider now ${h'_p}^{\ord (h)}$. It is equal to a conjugation
$c(y_p)$ where \[y_p \in N(A_\Sigma)\]
by construction.
Let $y_q \in N(A_\Theta)$ be the corresponding element for $h'_q$. Now
\[ c(y_p) = {h'_p}^{\ord (h)} = \big(c(x'(h)) {h'_q}\big)^{\ord (h)} = c(x'') {h'_q}^{\ord (h)} = c(x'') c(y_q) \]
where
\[x''  = \prod_{i=0}^{\ord(h)} (h'_q)^i(x'(h)) \in A_{Z(E) \s-(Z(\Sigma) \cup Z(\Theta))}\]
 since $Z(E) \s-(Z(\Sigma) \cup Z(\Theta)) \in \L^\phi$ and $h_q'(A_E) = A_E$ by construction.

 The element $y_p$ is determined up to $C(A_\Sigma)$ by its action by conjugation on
 $A_\Sigma$. Here however we immediately see that $c(y_p)$ is equal to conjugation by the element given by the loop
 obtained from concatenating images of our path $\gamma(h)$ under
 successive iterations of $h$.

We repeat the argument for $y_q$ and conclude that $y_p = y_q$ up to  $C(A_\Sigma) C(A_\Theta)$. But we have already shown that they differ by $x'' \in A_{Z(E) \s-(Z(\Sigma) \cup Z(\Theta))}$ up to $Z(A_\Gamma)$. Hence $c(x'') = 1$ as $Z(E)$ intersects $\lk(\Sigma)$ and $\lk(\Theta)$ trivially. Thus $x'' \in Z(A_\Gamma) \cap A_{Z(E) \s-(Z(\Sigma) \cup Z(\Theta))} = \{1\}$.
We obtain
\[1 = x''  = \prod_{i=0}^{\ord(h)} (h'_q)^i(x'(h))\]
and so
$x'(h)$ must lie in the subgroup of $A_{Z(E)}$ generated by
all vertices $s_i \in Z(E)$ such that $h$ flips the corresponding loop
$P_i = X_{s_i}$, since any other generator $s_j$ satisfies
\[\prod_{i=0}^{\ord(h)} (h'_q)^i(s_j) = {s_j}^{\ord(h)} \neq 1\]

The purpose of the proof so far was exactly to establish that $x'(h)$ lies in the subgroup of $A_{Z(E)}$ generated by
all vertices $s_i \in Z(E)$ such that $h$ flips the loop
$P_i$.

\smallskip
Now we are going to construct a whole family of gluings, and show that one of them is as desired.

To analyse the situation we need to look more closely at $\widetilde P = \widetilde Q$ (with the equality coming from the fact that $\X'$ extends $\X_E$ strongly). By the Product Axiom, we have $\widetilde P = \t P_0 \times \t P_1 \times \dots \times \t P_k$, with $\t P_0 = \t X_{E'}$, and $\t P_i = \t X_{s_i}$ for $i\geqslant 1$.
Our basepoint $\widetilde p$ satisfies
\[
\widetilde p = (\widetilde p_0, \dots, \widetilde p_k) \in \widetilde P_0  \times \dots \times \widetilde P_k
\]
Let $\widetilde q = (\widetilde q_0, \ldots, \widetilde q_k)$ be the corresponding expression for $\widetilde q$.

We construct complexes from $\t X_\Sigma$ and $\t X'_\Theta$ by gluing $\t P$ and $\t Q$ in a way respecting the markings, and so that the projections $P$ and $Q$ are glued in an $H$-equivariant fashion.
The resulting space is determined by the relative position of $\widetilde p$ and $\widetilde q$, now both seen as points in the glued-up complex (so in particular they do not need to coincide). We glue so that the images of $\t p$ and $\t q$ coincide if we project $\t X_E$ onto $\t X_{E'}$ -- this is in fact forced on us since $\t X_{E'}$ can be glued to itself only in one way. Hence any such gluing will give us a geodesic from $\t p$ to $\t q$, which will lie in a subcomplex of $\t X_E$ isomorphic to $\t P_1 \times \dots \times \t P_k$, and hence isometric to a Euclidean space.

We start by taking $x \in A_{Z(E)}$; we form a complex $\widetilde {X}^x$ (with projection $X^x$ as usual) by gluing as above, so that the geodesic $\t \delta$ we just discussed is such that if we extend it to twice its length (which is possible in a Euclidean space), still starting at $\t p$, the projection in $X^x$ becomes a closed loop equal to $x \in \pi_1(X^x,p)$.

Note that given $x \in Z(E)$ such a geodesic (and hence gluing) always exist: the Euclidean space we discussed is marked by $A_{Z(E)}$, and so the point $x(\t p)$ lies therein. We take the unique geodesic from $\t p$ to $x(\t p)$, cut it in half, and declare the first half to be $\t \delta$.

Let us now calculate the action on (conjugacy classes in) $A_\Gamma$ induced from the action
of $h$ on $X^x$. The element $h$ maps the local geodesic $\delta$ (the projection of $\t \delta$) to a
local geodesic $h. \delta$ connecting $h.p$ to $h.q$ such that the
loop obtained by following $\gamma(h)$ (starting at $p$), $h.\delta$,
the inverse of $\gamma'(h)$, and the inverse of $\delta$, gives an
element  \[\overline x^h \in \pi_1(P_1 \times \dots \times P_k,(p_1, \dots, p_k))\] which is the projection of $x$
onto the subgroup of $Z(E)$ generated by all the generators $s_i$ such
that $h'_p(s_i) = s_i^{-1}$.
 Hence the action of $h$ on $X^{x}$, followed by pushing the basepoint $p$ via $\gamma(h)$ as before, gives us an automorphism equal to $h'_p$ on the subgroup $\pi_1(X_{\Sigma},p) = A_{\Sigma}$, and to $c(\overline x^h) h'_q$ on the subgroup $\pi_1(X_ \Theta,p) = A_{\Theta}$.

It follows that the action of $h$ on $X^{x'(h)}$ is the correct one: by the observation above, $x'(h)$ lies in the subgroup of $A_{Z(E)}$ generated by all the vertices in $Z(E)$ which are mapped to their own inverse under $h'_p$. Hence $\overline{x'(h)}^h = x'(h)$, and it is enough to observe that $c(x'(h)) h'_q = h_p$.


We now need to specify a single $x$ that will work for all elements $h \in H$; we will denote such an $x$ by $x'(H)$. If there is a vertex in $Z(E)$ which is preserved by all elements $h'_p$ (for all $h \in H$), then we set the corresponding coordinate of $x'(H)$ to 0. If a generator is mapped to its inverse by some $h'_p$, then we set the corresponding coordinate of $x'(H)$ to be equal to the relevant coordinate in $x'(h)$.

Since this definition involves making choices (of elements $h$ that
flip a generator), we need to make sure that we indeed obtain the desired action.
Let $g \in H$ be any element. We need to confirm that $\overline{x'(g)}^g = \overline{x'(H)}^g$.

Suppose that this is not the case; then there exists a generator $s_i$ such that $g'(s_i) = {s_i}^{-1}$, and $x'(g)$ and $x'(H)$ disagree on the $s_i$-coordinate.
This means that the geometric representative $g'_p$ obtained from the action on $X^{x'(H)}$ is not a representative of $\phi(g)$. To make it such a representative we need to postcompose it with a partial conjugation of $A_\Theta$ by $\overline{x'(g)}^g \big(\overline{x'(H)}^g\big)^{-1}$, which has a non-trivial $s_i$-coordinate.

The construction of $x'(H)$ required an element $h \in H$ such that $h'_p$ flips the loop $P_i$.
We have $g  = f h$ with $f$ acting trivially on the conjugacy class of $s_i$. Hence, even though $f$ might not act correctly on $X^{x'(H)}$, the geometric representative $f'_p$ can be made into a representative of $\phi(f)$  by   postcomposing  it with a partial conjugation of $A_\Theta$ by an element of $Z(E)$ with a trivial $s_i$-coordinate. The exact same statement is true for $h$. Using the fact that $f'_p$ maps each $s_j$ either to itself or its inverse, we deduce that $f'_p h'_p$ can be made into a representative of $\phi(f h ) = \phi(g)$ by postcomposing  it with a partial conjugation of $A_\Theta$ by an element of $Z(E)$ with a trivial $s_i$-coordinate. But $f'_p h'_p$ differs from $(f h)'_p = g'_p$ by a conjugation, and hence $\overline{x'(g)}^g \big({\overline{x'(H)}^g}\big)^{-1}$ cannot have a non-trivial $s_i$-coordinate (as $\Gamma \neq Z(E)$).
\end{proof}

In particular we have

\begin{cor}
\label{cor: gluing}
If $\lk(E) = \emptyset$ then there exists a complex obtained from $\t X_\Sigma$ and $\t X'_\Theta$ by gluing $\t P$ to $\t Q$ such that its projection realises $\phi$.
\end{cor}
\begin{proof}
When $\lk(E) = \emptyset$ we have $C(A_E) = Z(A_E)$. Hence any gluing will be faulty within $Z(A_E)$, and then the existence of a desired gluing follows from the previous proposition.
\end{proof}


Let us record here a (very standard and) very useful result.

\begin{lem}
Let $X$ be a complete NPC space realising the action of a finite group $\phi \colon H \to \Out(A)$, where $A$ is a group. Suppose that there exists a lift $\psi \colon H \to \Aut(A)$ of $\phi$. Then the action of $H$ on $X$ has a fixed point $r$; moreover there exists a lift $\t r \in \t X$ of $r$ such that under the identification $\pi_1(X,r) = A$ induced by choosing $\t r$ as a basepoint, the induced action oh $H$ on $\pi_1(X,r)$ is $\psi$.
\label{lem: fixed point for liftable actions}
\end{lem}
\begin{proof}
Consider the action $A \curvearrowright \t X$ on the universal cover by deck transformations. The action of each $h \in H$ can be lifted from $X$ to $\t X$. The fact that $H \to \Out(A)$ lifts to $H \to \Aut(A)$ tells us that we can lift each $h$ in such a way that in fact the group $H$ acts on $\t X$. But $\t X$ is CAT(0), and hence every finite group of isometries has a fixed point. Thus in particular $H$ does; let $\t r$ denote this fixed point. Letting $r$ be the projection of $\t r$ finishes the proof.
\end{proof}

\section{Proof of the main theorem: preliminaries}
\label{sec:proof-main}

Our proof of the main theorem has an inductive character -- we will
induct on the dimension of the defining graph $\Gamma$. To emphasize
this (and to simplify statements) let us introduce the following
definitions.

\begin{dfn}
We say that \emph{Relative Nielsen Realisation holds} for an action $\phi \colon H \to \Out(A_\Gamma)$ \iff given $\Delta \in \L^\phi$ 
and any
cubical system $\Y$ for $\L^\phi_\Delta$, there exists a cubical system $\X$ for $\L^\phi$ extending $\Y$.
\end{dfn}

The rest of the paper is devoted to proving

\begin{thm}
\label{thm: case 2}
\label{thm: main}
Relative Nielsen Realisation holds for all link-preserving actions $\phi \colon H \to \Out(A_\Gamma)$ with $H$ finite.
\end{thm}

Before proceeding to the proof, let us record the following corollaries.

\begin{cor}
Let $\phi \colon H \to \Out^0(A_\Gamma)$ be a homomorphism with a finite domain. Then there exists a metric NPC cube complex $X$ realising $\phi$, provided that for any vertex $v \in \Gamma$, its link $\lk(v)$ is not a cone.
\end{cor}
\begin{proof}
By Corollary~\ref{cor:dim-2-case} $\phi$ is link-preserving. We take $X = X_\Gamma \in \X$ obtained by an application of the previous theorem.
\end{proof}

Note that in particular the statement above holds for all $\phi \colon H \to \Out(A_\Gamma)$, provided that in addition $\Gamma$ has no symmetries.

\begin{cor}
Let $\phi \colon H \to \U^0(A_\Gamma)$ be a homomorphism with a finite domain. Then there exists a metric NPC cube complex $X$ realising $\phi$.
\end{cor}
\begin{proof}
By \cref{lem: no adjacent transvections} the action  $\phi$ is link-preserving. We take $X = X_\Gamma \in \X$ obtained by an application of the previous theorem.
\end{proof}

\begin{lem}
\label{lem: case 1}
Assume that $\phi$ is link-preserving, and that Relative Nielsen Realisation holds for all link-preserving actions $\psi \colon H \to \Out(A_\Sigma)$ with $\dim \Sigma < \dim \Gamma$.
Suppose that $\Gamma = \Delta \ast (E \cup \Theta)$, where $\Delta$ and $\Theta$ are non-empty. Suppose that $\Delta \ast E \in \L^\phi$, and that we are given a cubical system $\X'$ for $\L^\phi_{\Delta \ast E}$. Then there exists a cubical system $\X$ for $\L^\phi$ extending $\X'$, which furthermore extends $\X'_\Delta$ strongly.
\end{lem}
\begin{proof}
Since $\L^\phi$ contains all links, we have $E \cup \Theta \in
\L^\phi$. Note that $E \cup \Theta$ has lower dimension than
$\Gamma$. Since
\[E = (E \cup \Theta) \cap (\Delta \ast E) \in \L^\phi\]
we can apply the assumption to the induced action on $A_{E \cup \Theta}$, and obtain a cubical system  $\X_{E \cup \Theta}$ extending $\X'_E$, the subsystem of $\X'$ corresponding to $\L^\phi_E$.  This last system also contains a subsystem $\X'_\Delta$ corresponding to $\L_\Delta$.

We now define $\X$ to be the product of $\X_{E \cup \Theta}$ and
$\X'_\Delta$ (compare Proposition~\ref{prop: product system}).
\end{proof}

\begin{proof}[Proof of Theorem~\ref{thm: main}]
We proceed by induction on the dimension of the defining graph $\Gamma$. Since we need to prove Relative Nielsen Realisation, let us fix $\Xi \in \L^\phi$, and a cubical system $\Y$ for $\L^\phi_\Xi$. Our aim is to construct a cubical system $\X$ for $\L^\phi$ extending $\Y$. If $\Xi = \Gamma$ then there is nothing to prove, so let us assume that $\Xi \subset \Gamma$ is a proper subgraph.

\smallskip
First we consider the case of $\Gamma$ being a join $\Gamma_1
\ast \Gamma_2$ for some non-empty subgraphs $\Gamma_1$ and $\Gamma_2$.
The dimension of $\Gamma_1$ is strictly smaller than that of $\Gamma$,
and 
\[\Gamma_1 = \lk(\Gamma_2) \in \L^\phi\]
 since $\phi$ is link-preserving. \link-preserving Thus,
by the inductive assumption, Relative Nielsen Realisation holds for
$\Gamma_1$ and yields a cubical system $\X_{\Gamma_1}$ for
$\L^\phi_{\Gamma_1}$ extending the system $\Y_{\Xi \cap \Gamma_1}$.
The same applies to $\Gamma_2$ and yields a cubical system
$\L^\phi_{\Gamma_2}$ extending $\Y_{\Xi \cap \Gamma_2}$. In this setting Proposition~\ref{prop: product system}
\link-preserving yields a cubical system $\X$ as required; we only need to observe that $\Y$ extends the product of $\Y_{\Xi \cap \Gamma_1}$ and $\Y_{\Xi \cap \Gamma_2}$.

\bigskip
For the rest of the proof we assume that $\Gamma$ is not a join.
We now proceed by induction on the \emph{depth} $k$ of
$\Gamma$. The depth is the length $k$ of a maximal chain
of proper inclusions
\[\emptyset = \Sigma_0 \subset \Sigma_1 \subset \dots \subset \Sigma_k
= \Gamma\]
where each $\Sigma_i \in \L^\phi$.

If $k=1$ then $\L^\phi = \{ \emptyset, \Gamma \}$. This in particular implies that $\Gamma$ is discrete: if it is not, let $v$ be a vertex of $\Gamma$ with non-empty link. Now $\widehat \st(v) \in \L^\phi$ is a
join of two non-empty graphs, and hence cannot be equal to $\Gamma$, which is not a join. But $\widehat \st(v) \neq \emptyset$, which contradicts the assumption on depth.
The depth $k$ being 1 also implies that $\Xi = \emptyset$. Thus our
theorem is reduced to the classical Nielsen Realisation for free
groups \cite{culler1984, khramtsov1985, Zimmermann1981, Henseletal2014}.  Note that we require our graphs to be leaf-free, so we might need to prune the leaves.

Suppose that $\Gamma$ is of depth $k\geqslant 2$, and that the theorem holds for all graphs of smaller depth. Let $\Gamma'$ denote the maximal (\wrt inclusion) proper subgraph of $\Gamma$ such that
\[ \Xi \subseteq \Gamma' \in \L^\phi \]
Since $\Gamma'$ is of smaller depth, our inductive assumption gives us a cubical system $\X'$ for $\L^\phi_{\Gamma'}$ extending $\Y$. The remainder of the proof will consist of a construction of a cubical system  $\X$ for $\L^\phi$, which extends $\X'$ (and thus $\Y$).

Whenever we speak about maximal subgraphs of $\Gamma$, we will always mean them to be maximal proper subgraphs in $\L^\phi$ (in particular not $\Gamma$ itself). 

\begin{claim}
\label{union of components or}
$\Gamma'$ is either a union of connected components of $\Gamma$, or it properly contains the union of all but one connected components of $\Gamma$.
\end{claim}
\begin{proof}
Suppose that $\Gamma'$ is not a union of connected components. Then there exists such a component, $\Gamma_0$ say, which intersect $\Gamma'$ in a non-empty proper subgraph of itself. In particular this implies that $\Gamma_0$ is not a singleton. This in turn implies that $\Gamma_0 \in \L^\phi$, and hence $\Gamma' \cup \Gamma_0 \in \L^\phi$. The maximality of $\Gamma'$ informs us that $\Gamma' \cup \Gamma_0 = \Gamma$, and so $\Gamma'$ contains all but one component of $\Gamma$ properly.
\end{proof}

Note that the above holds for any maximal subgraph of $\Gamma$.

We need to investigate two main cases, depending on whether $\Gamma'$
is a union of components (part I) or $\Gamma'$ properly contains the
union of all but one components (part II). Note that this second case always
occurs if $\Gamma$ is connected.

\renewcommand{\qedsymbol}{}
\end{proof}

\section{Proof of the main theorem: part I}
\label{sec:proof-main part I}
In this section we consider the case that $\Gamma'$ is a union of
connected components. 

\smallskip


Let us suppose first that any two maximal subgraphs of $\Gamma$ either coincide or are disjoint.
In this case we have
\[ \Gamma = \!\!\! \bigsqcup_{\Sigma \textrm{ maximal}}  \!\!\! \Sigma \sqcup \Theta \]
with $\Theta$ discrete, since each vertex with non-trivial link lies in its extended star, which is preserved (and not all of $\Gamma$, since the latter is not a join), and thus is contained in a maximal subgraph. Hence we also have
\[\L^\phi = \!\!\! \bigsqcup_{\Sigma \textrm{ maximal}} \!\!\! \L^\phi_{\Sigma} \sqcup \{\Gamma\} \]
Observe that each maximal $\Sigma$ is a union of components, as otherwise it would have to intersect $\Gamma'$ non-trivially (by \cref{union of components or}).
We conclude by applying Proposition~\ref{prop: sticking complexes together} to cubical systems for $\L^\phi_{\Sigma}$ (with $\Sigma$ maximal) obtained by the inductive assumption (taking $\X'$ for $\L^\phi_{\Gamma'}$).

\smallskip
We are left with the much more involved case, in which there exist maximal subgraphs of $\Gamma$ which are neither equal nor disjoint.

Let $\Theta = \Gamma \s- \Gamma'$. Note that $\Theta$ is also a union
of connected components. We define
\[ \S = \{ \Sigma \in \L^\phi \mid \Theta \subseteq \Sigma \} \]
\begin{claim}
 If $\Sigma \in \L^\phi$ intersects both $\Theta$ and $\Gamma'$
 non-trivially then $\Sigma \in \S$. In particular, any maximal
 $\Sigma$ different from $\Gamma'$ contains $\Theta$.
\end{claim}
\begin{proof}
Take $\Sigma$ as specified. Then $\Sigma \cup \Gamma' \in \L^\phi$ since $\st(\Sigma \cap \Gamma') \subseteq \Gamma'$ as $\Gamma'$ is a union of connected components. Thus $\Sigma \cup \Gamma' = \Gamma$ by maximality of $\Gamma'$, and so $\Theta \subseteq \Sigma$.  
\end{proof}

\begin{claim}
The system $\S$ is non-empty. 
\end{claim}
\begin{proof}
Suppose first that there exists a maximal $\Sigma \in \L^\phi$ which is not a union of components. Then $\Sigma \cap \Gamma'$ is non-empty (by \cref{union of components or}), and $\Sigma \neq \Gamma'$. Hence $\Sigma \in \S$.

Now suppose that all maximal subgraphs $\Sigma$ are unions of components. 
By our assumption there exist maximal $\Sigma, \Sigma'$ which intersect non-trivially and do not coincide. But then, in particular,  $\st(\Sigma \cap \Sigma') \subseteq \Sigma$, and so 
\[\Sigma \cup \Sigma' \in \L^\phi\]
 Maximality now yields $\Sigma \cup \Sigma' = \Gamma$, and so, without loss of generality, 
\[\Sigma \cap \Theta \neq \emptyset\]
 If $\Sigma \cap \Gamma' = \emptyset$, then we must have $\Gamma' \subseteq \Sigma'$, and so $\Gamma' = \Sigma'$. But then $\Sigma \cap \Gamma' = \Sigma \cap \Sigma' \neq \emptyset$ by assumption, and so $\Sigma \in \S$.
\end{proof}

Since $\Theta$ is a non-empty union of components, the link of
$\Theta$ is empty, and thus it follows by Lemma~\ref{lem: intersections in L} that $\S$ is closed under taking unions and
intersections. Hence so is 
\[\S_{\Gamma'} = \{ \Sigma \cap \Gamma' \mid \Sigma \in \S \} \]

Let us define
\[ \S' = \{ \st(\Sigma \cap \Gamma') \mid \Sigma \in \S \}\]
Note that $\S' \subseteq \S_{\Gamma'}$, since for every $\Sigma \in \S$ we have $\Sigma \cup \st(\Sigma \cap \Gamma') \in \S$ (see the proof of \cref{Delta' cup Theta} below).

Let 
\[\Delta = \bigcap \S_{\Gamma'}\]
and
\[ \Delta ' = \bigcap \S' \]

Apply the System Intersection Axiom in $\X'$ to the collection $\S_{\Gamma'}$. We obtain a collection of standard copies $\t Y_\Sigma$ of $\t X'_\Sigma$ in $\t X'_{\Gamma'}$ for each $\Sigma \in \S_{\Gamma'}$, such that the copies $\t Y_\Sigma$ intersect in $\t Y_\Delta$ (which is a point when $\Delta = \emptyset$).
Note that for each $\Sigma \in \S'$, the copy $\t Y_\Sigma$ is unique
(as $\lk(\Sigma) = \emptyset$), and hence fixed. Thus the intersection
of all such copies is also fixed. By the Intersection Axiom, this
intersection is $\t Y_{\Delta'}$.

Let us note that
\[ \Delta \cup \Theta  = \bigcap \S \in \L^\phi\]

\begin{claim}
\label{Delta' cup Theta}
$\Delta' \cup \Theta \in \L^\phi$.
\end{claim}
\begin{proof}
For each $\Sigma \in \S$ we have
\[ \st(\Sigma\cap \Gamma') \cup \Theta = \st(\Sigma\cap \Gamma') \cup \Sigma \in \L^\phi \]
since $\lk(\st(\Sigma\cap \Gamma') \cap \Sigma) \subseteq \lk(\Sigma\cap \Gamma') \subseteq \st(\Sigma\cap \Gamma')$, and due to part ii) of Lemma~\ref{lem: intersections in L}.
Thus
\[ \Delta' \cup \Theta =  \bigcap_{\Sigma \in \S} (\st(\Sigma\cap \Gamma') \cup \Theta) \in \L^\phi \]
by part i) of Lemma~\ref{lem: intersections in L}.
\end{proof}

\begin{claim}
  There exists a cubical system $\X''$ for $\L^\phi_{\Delta' \cup \Theta}$ which extends the subsystem $\X'_{\Delta'}$ of
  $\X'$ strongly.
\end{claim}
\begin{proof}
%
Since Relative Nielsen Realisation holds for $\Delta' \cup \Theta$, as it is a proper subgraph of $\Gamma$ and hence has lower depth, we obtain a complex $\X''$ for $\L^\phi_{\Delta' \cup \Theta}$ extending $\X'_{\Delta'}$.

We now need to look at vertices of $\Delta'$ which are singletons in
the join decomposition of $\Gamma'$; in other words we are looking at
vertices in $\Delta' \cap Z(\Gamma')$. If there are no such, then
Lemma~\ref{lem: rotating circles} implies that $\X''$ extends $\X'_{\Delta'}$ strongly, since $Z(\Theta \cup \Delta') = \emptyset$.

Now suppose that $s \in \Delta' \cap Z(\Gamma')$ exists. Note that
$\{s\} \in \L^\phi$ since $\phi$ is link-preserving. Then we use
Lemma~\ref{lem: replacing circle} and replace $\X'$ by another cubical
system which extends it (and so still extends the cubical system $\Y$
for $\L^\phi_\Xi$), and such that it extends $\X''_{\{s\}}$ strongly. 
Repeat this procedure for each vertex in $\Delta' \cap Z(\Gamma')$.
\end{proof}

Now we are ready to start the construction.

\subsection{Constructing the complexes}
\label{subsec: constructing complexes disconnected}

Let us begin by building up the necessary cube complexes. 

The basic idea is to glue the desired complexes from pieces lying in
$\Gamma'$ (where, by induction, we already have them in $\X'$) and pieces overlapping with $\Theta$. The details will be different depending on how
the invariant graph in question intersects $\Gamma'$ and $\Theta$, and thus
the construction has several steps.

\subsection*{Step 1: constructing $X_\Gamma$}

In this step we will construct various objects at once. First, we will build
$\t X_\Gamma$ from $\t \X'_{\Gamma'}$ and $\t \X''_{\Delta' \cup \Theta}$ by gluing $\t Y_{\Delta'}$ and $\t R_{\Delta'}$, a fixed standard copy of $\t \X''_{\Delta'}$ in $\t X''_{\Delta' \cup \Theta}$.

Since $\t X_\Gamma$ is obtained from a gluing procedure, we obtain at
the same time maps 
\[\t \iota_{\Gamma', \Gamma} \colon \t X'_{\Gamma'} \to \t X_\Gamma\] 
and 
\[\t \iota_{\Delta' \cup \Theta, \Gamma} \colon \t X''_{\Delta' \cup \Theta} \to \t X_\Gamma\]
 which are as required, since $\lk(\Gamma') = \lk(\Delta' \cup \Theta) = \emptyset$.
We will also construct a fixed standard copy $\t R_\Delta$ of $\t X''_\Delta$ and $\t R_{\Theta \cup \Delta} $ of $\t X''_{\Theta \cup \Delta}$ in $\t X''_{\Delta' \cup \Theta}$, with $\t R_{\Delta} \subseteq \t R_{\Theta \cup \Delta}$.
Furthermore, we will construct a collection of standard copies $\t Z_\Sigma$ of $\t X'_\Sigma$ in $\t X'_{\Gamma'}$ for each $\Sigma \in \S_{\Gamma'}$, such that the copies $\t Z_\Sigma$ intersect in $\t Z_\Delta$, and $\t Z_\Delta$ is identified with $\t R_\Delta$ under our gluing.

We need to consider two cases.

\subsubsection*{Case 1:} $\Delta'$ is empty.

In this case we have $\Theta  = \Theta \cup \Delta' \in \L^\phi$.

We apply Proposition~\ref{prop: sticking complexes together} to $\X'$ and $\X''$ and obtain a cubical system for $\L^\phi_{\Gamma'} \cup \L^\phi_\Theta \cup \{ \Gamma\}$. We are only interested in the cube complex associated to $\Gamma$ in this system -- let us call it $C$, and its universal cover  $\t C$. The complex $\t C$ is obtained from $\t X'_{\Gamma'}$ and $\t X''_\Theta$ by identifying a single point in each; let $\t q$ denote the point in $\t X''_\Theta$.

In particular $q$, the projection of $\t q$ in $C$, is fixed. Hence, picking $\t q$ as a basepoint of $\t C$ (and a trivial path), we get a geometric representative $h_q \in \Aut(A_\Gamma) $ of $\phi(h)$ for each $h \in H$. This representative satisfies
\[ h_q(A_{\Gamma'}) = A_{\Gamma'} \textrm{ and } h_q(A_\Theta) = A_\Theta \]
by construction.

Since $\Delta' = \emptyset$, the fixed standard copy $\t Y_{\Delta'}$ is simply a point $\t p \in \t X'_{\Gamma'}$ with a fixed projection $p \in X'_{\Gamma'}$. The point $\t p$  
can also be viewed as a point in $\t C$. Similarly $p$ becomes a point in $C$ which is also $H$-fixed. Thus we can use $\t p$ as a basepoint, and for each $\phi(h)$ obtain a geometric representative $h_p \in \Aut(A_\Gamma)$ such that $h_p(A_{\Gamma'}) = A_{\Gamma'}$.

We have $h_q h_p^{-1} = c(x)$ since they are representatives of the same outer automorphism, with $x \in A_\Gamma$.

Take $\Sigma \in \S$. Now
\[ A_\Theta = h_q(A_\Theta) \leq h_q(A_\Sigma) = {A_\Sigma}^y \]
for some $y \in A_\Gamma$, since $\Theta \in \L^\phi$.
Proposition~\ref{prop: ccv} implies that
\[y \in N(A_\Sigma) N(A_\Theta) = A_\Sigma \]
 and so $h_q(A_\Sigma) = A_\Sigma$.

Since $h_q(A_\Sigma) = A_\Sigma$ and $h_q(A_{\Gamma'}) = A_{\Gamma'}$,
we have $h_q(A_{\Sigma\cap {\Gamma'}}) = A_{\Sigma \cap {\Gamma'}}$
and so $h_q(A_{\st(\Sigma\cap {\Gamma'})}) = A_{\st(\Sigma \cap
  {\Gamma'})}$ (since $A_{\st(\Sigma \cap {\Gamma'})} =
N(A_{\st(\Sigma \cap {\Gamma'})})$). But we also have $h_p(A_{\st(\Sigma\cap {\Gamma'})}) = A_{\st(\Sigma \cap {\Gamma'})}$, since the basepoint $\t p$ lies in $\t Y_{\st(\Sigma \cap \Gamma')}$
for each $\Sigma \in \S$. Thus
\[  A_{\st(\Sigma \cap {\Gamma'})} = h_q h_p^{-1} (A_{\st(\Sigma\cap {\Gamma'})}) ={A_{\st(\Sigma\cap {\Gamma'})}}^x \]
and thus $x \in N(A_{\st(\Sigma \cap {\Gamma'})}) = A_{\st(\Sigma \cap {\Gamma'})}$ for each $\Sigma \in \S$. Hence 
\[x \in A_{\Delta'} = \{1\}\]
 and so $h_q = h_p$.
This equality means that the inherited action of $H$ on $X_\Gamma$, the cube complex obtained from $\t X'_{\Gamma'}$ and $\t X''_\Theta$ by gluing $\t p$ and $\t q$, induces $\phi$.

We now put $\t R_\Delta = \t R_{\Delta'} = \t q$, and set $\t Z_\Sigma = \t Y_\Sigma$ for each $\Sigma \in \S_{\Gamma'}$. We also put $\t R_{\Theta \cup \Delta} = \t X''_\Theta$.

\subsubsection*{Case 2:} $\Delta'$ is non-empty.


Let us observe the crucial property of $\Delta'$.

\begin{claim}
\label{crucial prop of Delta'}
Let $\Sigma \in \S$, and let $\t p \in \t Y_{\Delta'}$. Then there exists a standard copy $\t W_{\Sigma\cap \Gamma'}$ of $\t X'_{\Sigma\cap \Gamma'}$ in $\t X'_{\Gamma'}$ containing $\t p$
\end{claim}
\begin{proof}
When $\Sigma = \Theta$ the result follows trivially.

Suppose that $\Theta \subset \Sigma$ is a proper subgraph.
Note that we have a standard copy $\t Y_{\st(\Sigma\cap \Gamma')}$ which contains $\t Y_{\Delta'}$, and hence $\t p$. We have
\[ \st(\Sigma\cap \Gamma') = \Sigma\cap \Gamma' \ast \lk_{\Gamma'}(\Sigma\cap \Gamma') \]
and so the Product Axiom and Composition Axiom imply that there exists a standard copy $\t W_{\Sigma\cap \Gamma'}$ as required.
\end{proof}

We have $\lk(\Delta')_{\Delta' \cup \Theta} = \emptyset$, as $\Theta$ is a union of components, and so there is a unique (and hence fixed) standard copy of $X''_{\Delta'} = X'_{\Delta'}$ in $X''_{\Delta' \cup \Theta}$; we will denote it by $\t R_{\Delta'}$.
For the same reason there is a unique (and so fixed) standard copy $\t R_{\Delta \cup \Theta}$ of $\t X''_{\Delta \cup \Theta}$ in $\t X''_{\Theta \cup \Delta'}$. The Matching Property tells us that these intersect in a standard copy $\t R_\Delta$ of $\t X''_\Delta$; this copy is fixed, since it is the intersection of two fixed standard copies. 

We now obtain a complex $\t C$ from $\t X'_{\Gamma'}$ and $\t
X''_{\Delta' \cup \Theta}$ by gluing $\t Y_{\Delta'} = \t
R_{\Delta'}$. However, the projection $C$ of our complex $\t C$ might not
yet realise the desired action of $H$.

Pick $h \in H$. Take a point $\t r \in \t R_\Delta$ as a basepoint, let $r$ be its projection as usual. Take a path $\gamma(h)$ in $R_\Delta$ from $r$ to $h.r$. Note that $\t r$ is a point in $\t R_{\Delta'}$, and hence in $\t C$ (after the gluing). Let $\t p \in \t Y_{\Delta'} = \t R_{\Delta'}$ be the corresponding point; we view it as a point in $\t C$ as well, and denote its projection by $p$ as usual.
Let $h_r, h_p \in \Aut(A_\Gamma)$ denote the corresponding geometric representatives.

\begin{claim}
The gluing $\t C$ is faulty within $Z(A_{\Delta'})$.
\end{claim}
\begin{proof}
Let us first remark that $\lk(\Sigma) \subseteq \lk(\Theta) =  \emptyset$ for all $\Sigma \in \S$.

By construction we see that
\[ h_r(A_{\Delta \cup \Theta}) = A_{\Delta \cup \Theta} \]
Take $\Sigma \in \S$. Since $\Sigma$ lies in $\L^\phi$, we have
\[ h_r(A_\Sigma) = A_\Sigma^y\]
for some element $y \in A_\Gamma$. But then
\[ A_{\Delta \cup \Theta} = h_r(A_{\Delta \cup \Theta}) \leqslant h_r(A_\Sigma) = A_\Sigma^y \]
Proposition~\ref{prop: ccv} implies that $y \in N(A_\Sigma) N(A_{\Delta \cup \Theta}) = A_\Sigma A_{\Delta \cup \Theta} = A_\Sigma$, and thus
\[ h_r(A_\Sigma) = A_\Sigma\]
Now $\Sigma\cap{\Gamma'} \in \L^\phi$ and so
\[ h_r(A_{\Sigma\cap{\Gamma'}}) = A_{\Sigma\cap{\Gamma'}}^z\]
We have $A_{\Sigma\cap{\Gamma'}}^z \subseteq A_\Sigma$, and so (using
Proposition~\ref{prop: ccv} again) we get, without loss of generality,
$z \in N(A_{\Sigma}) = A_\Sigma$ since $\lk(\Sigma) = \emptyset$ (as $\Sigma \in \S$).

Let us now focus on $h_p$. By \cref{crucial prop of Delta'}, there exists a standard copy of $\t X'_{\Sigma\cap{\Gamma'}}$ containing $\t p$; it is clear that it will also contain the copy of the path $\gamma(h)$. Hence
\[ h_p(A_{\Sigma\cap{\Gamma'}}) = A_{\Sigma\cap{\Gamma'}}\]
and so
\[ A_{\Sigma\cap{\Gamma'}}^z = h_r(A_{\Sigma\cap{\Gamma'}}) = c(x(h)^{-1})  h_p(A_{\Sigma\cap{\Gamma'}}) = A_{\Sigma\cap{\Gamma'}}^{x(h)^{-1}} \]
where
\[ c(x(h)) = h_p h_r^{-1} \]
and $x(h) \in C(A_{\Delta'})$.

Now
\[ z = z \, x(h) \, {x(h)}^{-1} \in N(A_{\Sigma\cap{\Gamma'}}) C(A_{\Delta'}) \leqslant A_{\Gamma'} \]
and thus 
\[z \in A_\Sigma \cap A_{\Gamma'} = A_{\Sigma\cap{\Gamma'}} \]
and so
\[ h_r(A_{\Sigma\cap{\Gamma'}}) = {A_{\Sigma\cap{\Gamma'}}}^z = A_{\Sigma\cap{\Gamma'}}\]
Finally
\[ c(x(h)) = h_p h_r^{-1}(A_{\Sigma\cap{\Gamma'}}) = A_{\Sigma\cap{\Gamma'}}\]
for all $\Sigma \in \S$, which in turn implies
\[ x(h) \in \bigcap_{\Sigma \in \S} N(A_{\Sigma\cap{\Gamma'}}) = \bigcap_{\Sigma \in \S} A_{\st(\Sigma\cap{\Gamma'})} = A_{\Delta'} \]
Therefore $x(h) \in C(A_{\Delta'}) \cap A_{\Delta'} = Z(A_{\Delta'})$.

This statement holds for
each $h$, and thus
the fault of our gluing satisfies the claim.
\end{proof}

Now we are in a position to apply Proposition~\ref{prop: gluing} and obtain a new glued up complex, which we call $\t X_\Gamma$, which realises our action $\phi$.

Recall that we have a standard copy $\t R_{\Delta}$ in $\t X''_{\Delta' \cup \Theta}$. The gluing sends $\t R_\Delta$ to some standard copy of $\t X'_\Delta$ in $\t X'_{\Gamma'}$, which lies within $\t Y_{\Delta'}$ (we are using the Composition Property here); let us denote this standard copy by $\t Z_\Delta$. It is fixed since $\t R_\Delta$ is.

By \cref{crucial prop of Delta'}, we may pick a standard copy
$\t Z_\Sigma$ of $\t X'_{\Sigma}$ in $\t X'_{\Gamma'}$ for each $\Sigma \in \S_{\Gamma'}$
such that they will all intersect in $\t Z_\Delta$. Let us choose such a family of standard copies.

\subsection*{Step 2: Constructing the remaining complexes}

Let $\Sigma \in \L^\phi$ be a proper subgraph of $\Gamma$. When $\Sigma \subseteq \Gamma'$ we set $\t X_\Sigma = \X'_\Sigma$; when $\Sigma \subseteq \Theta \cup \Delta$, we set $\t X_\Sigma = \t X''_\Sigma$. Note that this is consistent for graphs $\Sigma \subseteq \Delta$ since $\X''$ extends $\X'_{\Delta'}$ (and thus $\X'_\Delta$) strongly.

Now let $\Sigma \in \S$. We have a standard copy $\t Z_{\Sigma \cap \Gamma'}$ in $\t X_{\Gamma'} =  \t X'_{\Gamma'}$. This last complex is embedded in $\t X_\Gamma$, and so we may think of $\t Z_{\Sigma \cap \Gamma'}$ as being embedded therein as well.

By construction $\t Z_{\Sigma \cap \Gamma'}$ contains $\t Z_\Delta$; this last copy is glued to $\t R_\Delta$, which in turn lies within $\t R_{\Theta \cup \Delta}$. We define $ X_\Sigma$ to be the union (taken in $X_\Gamma$) of $Z_{\Sigma \cap \Gamma'}$ and $R_{\Delta \cup \Theta}$.

Observe that when $\Sigma = \Theta \cup \Delta$ we have given two ways of constructing $\t X_\Sigma$; it is however immediate that the outcome of both methods is identical.

The projection $X_\Sigma$ carries the desired
marking by construction. The action of $H$ is also the desired one;
taking any $h \in H$, and looking at a geometric representative (in $\Aut(A_\Gamma)$)
obtained by choosing a basepoint and a path in the subcomplex
$X_\Sigma$, we get an automorphism of $A_\Gamma$ which preserves
$A_\Sigma$. This automorphism is a representative of $\phi(h)$, and so
the restriction to $A_\Sigma$ is the desired one. But this is equal to
the geometric representative of the action of $h$ on $X_\Sigma$
obtained using the same basepoint and path.

We now embed $\t X_\Sigma$ into $\t X_\Gamma$ in such a way that the image contains both $\t Z_{\Sigma \cap \Gamma'}$ and $\t R_{\Theta \cup \Delta}$, and so that this embedding gives the inclusion $X_\Sigma \subseteq X_\Gamma$ when we take the quotients.

Note that the construction gives us an embeddings $\t \iota_{\Sigma , \Gamma} \colon \t X_\Sigma \to \t X_\Gamma$, $\t \iota_{\Sigma\cap \Gamma', \Sigma} \colon \t X_{\Sigma \cap \Gamma'} \to \t X_\Sigma $ and $\t \iota_{\Delta \cup \Theta, \Sigma} \colon \t X_{\Delta \cup \Theta} \to \t X_\Sigma $, which are as required since $\lk(\Sigma) = \lk_{\Sigma}(\Sigma \cap \Gamma') = \lk_\Sigma(\Delta \cup \Theta) =\emptyset$.

\subsection{Constructing the maps}
\label{subsec: constructing maps disconnected} We now need to specify the maps $\widetilde{\iota}$. We take $\Sigma, \Sigma' \in \L^\phi$ with $\Sigma \subseteq \Sigma'$. When the two graphs are identical, we set $\t \iota_{\Sigma, \Sigma'}$ to be the identity. Let us now suppose that $\Sigma \neq \Sigma'$.

\begin{enumerate}
 \item
\textbf{ $\Sigma' \subseteq \Gamma'$ or $\Sigma' \subseteq \Delta \cup \Theta$ }

In this case the cube complexes $X_\Sigma$ and $X_{\Sigma'}$ are
  obtained directly from another cubical system ($\X'$ or $\X''_\Delta$), and
  we take $\t \iota_{\Sigma, \Sigma'}$ to be the map coming from that system.

For the remaining cases, we will assume that the hypothesis of $(1)$
is not satisfied, which implies $\Sigma' \in \S$. 

\item \textbf{$\Sigma \subseteq \Gamma'$}

In this case we have $\st(\Sigma) = \st(\Sigma) \cap \Gamma'$, since $\Gamma'$ is a union of components.
Therefore $\lk_{\Sigma'}(\Sigma) = \lk_{\Sigma'}(\Sigma) \cap \Gamma'$ as well.
We define
\[ \t \iota_{\Sigma, \Sigma'} = \t \iota_{\Sigma' \cap \Gamma', \Sigma'} \circ \t \iota_{\Sigma, \Sigma' \cap \Gamma'} \]
where the last map was defined in (1) above, and the map $\t
\iota_{\Sigma'\cap \Gamma', \Sigma'}$ was constructed together with the complex
$\t X_{\Sigma'}$ in Step~2 of Subsection~\ref{subsec: constructing complexes disconnected}.

\item \textbf{$\Sigma \subseteq \Theta \cup \Delta$}

In this case we have $\st(\Sigma) = \st(\Sigma) \cap (\Theta \cup \Delta)$, since $\Theta$ is a union of components.
Therefore $\lk_{\Sigma'}(\Sigma) = \lk_{\Sigma'}(\Sigma) \cap \Theta$ as before.
We define
\[ \t \iota_{\Sigma, \Sigma'} = \t \iota_{\Delta \cup \Theta, \Sigma'} \circ \t \iota_{\Sigma, \Delta \cup \Theta} \]
where the last map was defined in (1) above, and the map $\t
\iota_{\Sigma'\cap \Gamma', \Sigma'}$ was constructed together with the complex
$\t X_{\Sigma'}$ in Step~2 of Subsection~\ref{subsec: constructing complexes disconnected}.

\item \textbf{ $\Sigma \in \S$}

Observe that both $\t X_\Sigma$ and $\t X_{\Sigma'}$ are defined as subcomplexes of $\t X_\Gamma$. Since $\t Z_{\Sigma} \subset \t Z_{\Sigma'}$ (by the Intersection Axiom, as they both contain $\t Z_\Delta$), we see that $\t X_\Sigma \subset \t X_{\Sigma'}$; we define this embedding to be $\t \iota_{\Sigma, \Sigma'}$. Note that the map is as required since $\lk_{\Sigma'}(\Sigma) = \emptyset$ as $\Theta \subseteq \Sigma$, and $\Theta$ is a union of components.

Note that we have (again) given two constructions when 
\[\Sigma = \Delta \cup \Theta\] but again they are easily seen to be identical.
\end{enumerate}

\subsection{Verifying the axioms}

The first two axioms depend only on two subgraphs $\Sigma, \Sigma' \in
\L^\phi$ with $\Sigma \subseteq \Sigma'$. This is the same assumption
as in~\ref{subsec: constructing maps disconnected}, and hence the verification of the two axioms will follow the same structure as the construction of the maps -- we will consider four cases, and the assumption in each will be identical to the assumptions of the corresponding case above.

\subsubsection*{Product Axiom} Suppose that $\Sigma' = \st_{\Sigma'}(\Sigma)$.
\begin{enumerate}
 \item In this case the Product
Axiom follows from the Product Axiom in $\X'$ or $\X''$.

\item As in the analogous case of Subsection~\ref{subsec: constructing maps disconnected}, we see that $\st_{\Sigma'}(\Sigma) \subseteq \Gamma'$. But here $\Sigma' = \st_{\Sigma'}(\Sigma)$ and so we are in case (1) above.

\item As in the analogous case of Subsection~\ref{subsec: constructing maps disconnected}, we see that 
\[\st_{\Sigma'}(\Sigma) \subseteq \Delta \cup \Theta\]
 But here $\Sigma' = \st_{\Sigma'}(\Sigma)$ and so we are in case (1) above.

\item In this case we have $\lk(\Sigma) = \emptyset$, and so $\Sigma = \Sigma'$. But in such a case we defined $\t \iota_{\Sigma, \Sigma'}$ to be the identity.
\end{enumerate}

\subsubsection*{Orthogonal Axiom} Let $\Lambda = \lk_{\Sigma'}(\Sigma)$.

\begin{enumerate}
 \item Follows from the axiom in $\X'$ or $\X''$.

\item As before we have $\Lambda \subseteq \Gamma'$. The construction of the maps $\t \iota_{\Sigma, \Sigma'}$ and $\t \iota_{\Lambda, \Sigma'}$ immediately tells us that it is enough to verify the axiom for the pair $\Sigma, \Gamma'$. But this is covered by (1). 

\item As before we have $\Lambda \subseteq \Delta \cup \Theta$. The construction of the maps $\t \iota_{\Sigma, \Sigma'}$ and $\t \iota_{\Lambda, \Sigma'}$ immediately tells us that it is enough to verify the axiom for the pair $\Sigma, \Delta \cup \Theta$. But this is covered by (1). 

\item In this case we have $\Lambda = \emptyset$, and hence the axiom follows trivially, since $\t \iota_{\Lambda, \Sigma'}$ is onto (by the Product Axiom, which we have already shown).
\end{enumerate}

\subsubsection*{Intersection Axiom}
Let us now verify that $\X$ satisfies the Intersection Axiom. Take
$\Sigma, \Sigma', \Omega \in \L^\phi$ such that $\Sigma \subseteq
\Omega$ and $\Sigma' \subseteq \Omega$, and let $\t Y_{\Sigma}$ and $\t Y_{\Sigma'}$ be standard copies of, respectively, ${\t
  X}_{\Sigma}$ and $\t X_{\Sigma'}$ in ${\t X}_\Omega$ with non-empty intersection. We need
to show that the intersection is the image of a standard copy of
$\Sigma \cap\Sigma'$ in each.

As in the first two cases, the details depend on the inclusions $\Sigma, \Sigma' \subseteq \Omega$. The cases will thus be labeled by pairs of integers $(n,m)$, the first determining in which case the map $\t \iota_{\Sigma, \Omega}$ was constructed, and the second playing the same role for $\t \iota_{\Sigma', \Omega}$. By symmetry we only need to consider $n \leqslant m$.

\begin{itemize}
 \item[(1,1)] In this case the axiom follows from the Intersection Axiom in $\X'$ or $\X''$. 

In what follows we can assume that $\Omega \in \S$, and hence $\t X_\Omega$ is obtained by gluing $\t Z_{\Omega \cap \Gamma'}$ and $\t R_{\Delta \cup \Theta}$ along $\t Z_\Delta$.

\item[(2,2)] Here the problem is reduced to checking the axiom for the triple $\Sigma, \Sigma', \Omega \cap \Gamma'$, for which it holds by (1,1).

\item[(2,3)] By construction, the given standard
  copies $\t Y_{\Sigma}$ and $\t Y_{\Sigma'}$ must lie within the standard
  copies $\t Z_{\Omega \cap \Gamma'}$ and $\t R_{\Delta \cup \Theta}$ respectively, and
  hence intersect within $\t Z_{\Delta} = \t R_\Delta$.

We use the Intersection Axiom of $\X'$ for
$\t Z_{\Delta}$ and $\t Y_{\Sigma}$ inside $\t Z_{\Omega \cap \Gamma'}$ and see that the two copies intersect in a copy of $\t X_{\Delta \cap \Sigma}$, which is also the image of a standard copy of $\t X_{\Delta \cap  \Sigma}$ in $\t Z_{\Delta \cap \Omega}$.

We repeat the argument for $\Sigma'$ and obtain a standard copy
of $\t X_{\Delta \cap  \Sigma'}$ in $\t Z_{\Delta}$. Now
this copy intersects the one of $\t X_{\Delta \cap  \Sigma}$, and
hence, applying the Intersection Axiom again, they intersect in a
copy of $\t X_{\Delta \cap \Sigma \cap \Sigma'}$ in $\t Z_{\Delta}$. But $\Delta \cap \Sigma \cap \Sigma' = \Sigma \cap
\Sigma'$, and so we have found the desired standard copy in $\t
Z_{\Delta}$. Now the Composition Property (Lemma~\ref{lem: composition prop}) implies that
this is also a standard copy in $\t X_\Sigma$, $\t X_{\Sigma'}$ and
$\t X_{\Omega \cap \Gamma'}$, and thus this is also a standard copy in $\t
X_\Omega$ by construction.

\item[(2,4)] The non-trivial intersection of any standard copy of $\t X_\Sigma$ and any standard copy of $\t X_{\Sigma'}$ in $\t X_\Omega$ is in fact contained in the standard copy of $\t X_{\Omega \cap \Gamma'}$, since any copy of $\t X_{\Sigma}$ is contained therein. Therefore the intersection is also contained in a standard copy of $\t X_{\Sigma' \cap \Gamma'}$ by construction of $\t \iota_{\Sigma', \Omega}$.
We apply the Intersection Axiom in $\t X_{\Omega \cap \Gamma'}$, and observe that the standard copy of $\t X_{\Sigma \cap \Sigma'}$ in $\t X_{\Omega \cap \Gamma'}$ obtained this way is also a standard copy in $\t X_\Omega$ by construction.

\item[(3,3)] Here the problem is reduced to checking the axiom for the triple $\Sigma, \Sigma', \Delta \cup \Theta$, for which it holds by (1,1).

\item[(3,4)] The non-trivial intersection of any standard copy of $\t X_{\Sigma'}$ and any standard copy of $\t X_{\Sigma}$ in $\t X_\Omega$ is in fact contained in the standard copy of $\t X_{\Delta \cup \Theta}$, since any copy of $\t X_{\Sigma'}$ is contained therein. Therefore the intersection is also contained in a standard copy of $\t X_{\Delta \cup \Theta}$ by construction of $\t \iota_{\Sigma', \Omega}$.
We apply the Intersection Axiom in $\t X_{\Delta \cup \Theta}$, and observe that the standard copy of $\t X_{\Sigma \cap \Sigma'}$ in $\t X_{\Delta \cup \Theta}$ obtained this way is also a standard copy in $\t X_\Omega$ by construction.

\item[(4,4)] In this case $\t Y_\Sigma$ and
  $\t Y_{\Sigma'}$ are obtained from $\t Z_{\Sigma\cap \Gamma'}$, $\t R_{\Delta \cup \Theta}$, and $\t Z_{\Sigma'\cap \Gamma'}$, $\t R_{\Delta \cup \Theta}$ respectively.
Hence $\t Z_{\Sigma\cap \Gamma'} \cap \t Z_{\Sigma'\cap \Gamma'}$ contains $\t Z_\Delta$, and so $\t Z_{\Sigma\cap \Gamma'}$ and $\t Z_{\Sigma'\cap \Gamma'}$ intersect non-trivially.
Applying the Intersection Axiom in $\t X_{\Omega \cap \Gamma'}$ tells us that in fact $\t Z_{\Sigma\cap \Gamma'} \cap \t Z_{\Sigma'\cap \Gamma'} = \t Z_{\Sigma \cap \Sigma' \cap \Gamma'}$, which contains $\t Z_\Delta$. Thus $\t Y_\Sigma \cap \t Y_{\Sigma'}$ is equal to the unique standard copy of $\t X_{\Sigma \cap \Sigma'}$ in $\t X_\Gamma$, which lies within the unique standard copy of $\t X_\Omega$ by construction. It is also a standard copy in  $\t X_\Omega$, again by construction.
\end{itemize}

\subsubsection*{System Intersection Axiom}
Take a subsystem $\P \subseteq \L^\phi$ closed under taking unions. If
all elements of $\P$ lie in $\Gamma'$ or in $\Delta \cup \Theta$, then
we are done (from the System Intersection Axiom of $\X'$ or
$\X''$). So let us suppose this is not the case, that is, suppose that
there exists $\Sigma \in \P \cap \S$. Hence $\bigcup \P \in \S$.

Define
\[ \P' =  \{ \Sigma \in \P \mid \Theta \subseteq \Sigma \} \]
Note that each $\Sigma \in \P'$ has $\lk(\Sigma) = \emptyset$, and thus there exists a unique standard copy of $\t X_\Sigma$ in $\t X_{\bigcup \P}$ for each $\Sigma \in \P'$, and all these copies intersect non-trivially (they all contain $\t R_{\Delta \cup \Theta}$).

Note that $\P'$ is closed under taking unions, and that $\Theta \subseteq \bigcap \P'$. Thus we may apply Lemma~\ref{lem: extending intersections}. The construction is now finished by the application of the Intersection Axiom (which we have already proved).

\section{Proof of the main theorem: part II}
\label{sec:proof-main part II}

In this section we deal with the case that $\Gamma'$ properly contains the union of all but one connected components of $\Gamma$.
For notational convenience let $\Gamma_0$ denote the connected component not contained in $\Gamma'$.

\begin{claim}
\label{trivial link}
$\lk(\Gamma') = \emptyset$
\end{claim}
\begin{proof}
Note that $\st(\Gamma') \in \L^\phi$ since $\Gamma' \in \L^\phi$. If $\lk(\Gamma') \neq \emptyset$ we have $\Gamma' \subset \st(\Gamma')$ and so $\st(\Gamma') = \Gamma$. But $\Gamma$ is not a join, and $\st(\Gamma')$ is. This is a contradiction.
\end{proof}

We let $\Theta$ be the complement
\[\Theta = \Gamma \s- \Gamma' \subseteq \Gamma_0\]
%
Note that at this point we do not claim that $\Theta$ is contained in
$\L^\phi$.

Lemma~\ref{lem: boundary} (applied to $\Gamma_0$ and $\Gamma' \cap \Gamma_0$) implies that for all $w \in \partial
\Gamma'$ we have $\Theta \subseteq \lk(w)$. Thus we have
\[\partial \Gamma' = \lk(\Theta)\]

%
%

\subsection{Constructing the complexes}
\label{subsec: constructing complexes}
We are now ready to define the cube complexes forming $\X$.

To describe the cases, we need the following additional graphs. We let
$$\overline \Theta = \Theta \ast \partial \Gamma'$$
and we consider the subsystem of all invariant graphs which contain
$\overline\Theta$:
\[ \S = \{ \Sigma \in \L^\phi \mid \overline \Theta \subseteq \Sigma
\} \]
Since $\lk(\overline \Theta) = \emptyset$, the link of every element
in $\S$ is empty as well. Thus, $\S$ is closed under taking unions by Lemma~\ref{lem: intersections in L},  Hence
\[\S_{\Gamma'} = \{ \Sigma \cap \Gamma' \mid \Sigma \in \S \}\] is also
closed under taking unions.

We abbreviate
$$\Delta = \bigcap \S_{\Gamma'}$$
For any $\Sigma \in
\L^\phi$ we let
\[\Sigma_1 = \Sigma \cap \Gamma' \quad\mbox{and}\quad \Sigma_2 = \Sigma \cap
(\Delta \cup \Theta)\]
In other words, these are the part of $\Sigma$ which lie inside
$\Gamma'$ and inside $\Theta \cup \Delta$.


%

\subsection*{Step 1: Constructing $X_\Gamma$}

In this step we will actually do a little more. First we will
construct a suitable cubical system $\X''$ for $\L^\phi_{\st(\partial
  \Gamma ')}$ extending $\X'_{\Delta}$ strongly.
Then we will find a fixed standard copy $\t R_\Delta$ of $\t X''_\Delta$ in $\t X''_{\st(\partial \Gamma')}$. We will also construct a family of standard copies
$\t Z_\Sigma$ of $\t X'_{\Sigma}$ in $\t X'_{\Gamma'}$ for each $\Sigma \in \S_{\Gamma'}$, such that they all intersect in the copy $\t Z_\Delta$. We will show that this last copy is fixed, and obtain $\t X_\Gamma$ from $\t X'_{\Gamma'}$ and $\t X''_{\Delta \cup \Theta}$ (which will be shown to contain $\t R_\Delta$) by gluing $\t R_\Delta$ and $\t Z_\Delta$.

Note that the construction of a gluing then also gives us maps
\[\t \iota_{\Gamma', \Gamma} \colon \t X'_{\Gamma'} \to \t X_\Gamma \]
and
\[\t \iota_{\Delta \cup \Theta, \Gamma} \colon \t X''_{\Delta \cup \Theta} \to \t X_\Gamma \]
which will be as desired since $\lk(\Gamma') = \emptyset = \lk(\Delta \cup \Theta)$.

We need to consider two cases depending on whether $\lk(\partial \Gamma')$
intersects $\Gamma'$ or not.

\subsubsection*{Case 1: $\lk(\partial\Gamma')\cap\Gamma' = \emptyset$}

\begin{figure}
\label{fig: case 1}
\includegraphics{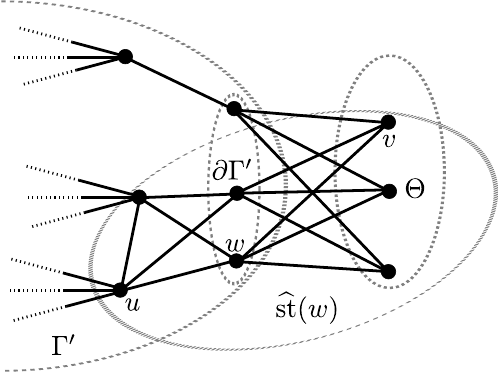}
\caption{The relevant subgraphs of $\Gamma$ in case 1}
\end{figure}

Recall that $\partial \Gamma' = \lk(\Theta)$. By the case assumption we
  therefore have
  \[\Theta = \lk({\partial \Gamma'}) \in \L^\phi \]

Since $\Theta$ has lower depth than $\Gamma$, our inductive assumption gives us a cubical system for $\L^\phi_\Theta$. 
By Proposition~\ref{prop: product system} there is a
system $\X''$ for $\L^\phi_{\overline \Theta}$, strongly extending
both this cubical system and the subsystem $\X'_{\partial
  \Gamma'}$ of $\X'$.

  Take any $v \in \Theta$. We have $\partial \Gamma' =
  \widehat \st(v)_1$ and thus $\partial \Gamma' \in \L^\phi$. Hence in particular $\Delta = \partial \Gamma'$. We also have $\st(\partial \Gamma') = \partial \Gamma' \ast \Theta \in \L^\phi$.

\begin{claim}
\label{claim: 9}
The induced action $H \to \Out(A_\Theta)$ lifts to an action 
\[H \to \Aut(A_\Theta)\]
\end{claim}
\begin{proof}
Since $\partial \Gamma' \in \L^\phi$, for each $h \in H$ there exists a representative 
\[h_1 \in \Aut(A_\Gamma)\]
 of $\phi(h)$ such that $h_1(A_{\partial \Gamma'}) = A_{\partial \Gamma'}$. Now
\[ h_1(A_{{\overline \Theta}}) = {A_{{\overline \Theta}}}^x \]
since ${\overline \Theta} \in \L^\phi$, with $x \in A_\Gamma$. Proposition~\ref{prop: ccv} tells us that
\[ x \in N(A_{{\overline \Theta}}) N(A_{\partial \Gamma'}) = A_{{\overline \Theta}} \]
and so
\[ h_1(A_{{\overline \Theta}}) = A_{{\overline \Theta}} \]
for each $h$. Since $\Theta \in \L^\phi$ we also have
\[ h_1(A_{\Theta}) = {A_{\Theta}}^y \]
with $y \in N(A_\Theta) N(A_{{\overline \Theta}}) = A_{{\overline \Theta}}$, again by Proposition~\ref{prop: ccv}. Hence
\[ h_1(A_{\Theta}) = A_{\Theta} \]
Finally, we have 
\[ h_1(A_{\Gamma'}) = {A_{\Gamma'}}^z \]
with $z \in N(A_{\Gamma'}) N(A_{\Theta}) = A_{\partial \Gamma'} A_\Theta$. Without loss of generality we take $z \in A_\Theta$, and so we construct new representatives $h_2 \in \Aut(A_\Gamma)$ for $\phi(h)$ (by multiplying $h_1$ with the conjugation by $z^{-1}$ in the appropriate way) which satisfy
\[ h_2(A_{\partial \Gamma'}) = A_{\partial \Gamma'} , \ h_2(A_{\Theta}) = A_{\Theta}, \textrm{ and }  h_2(A_{\Gamma'}) = A_{\Gamma'} \]
Now, given $h,g \in H$, we know that $h_2 g_2 {(hg)_2}^{-1}$ is equal to a conjugation which preserves $A_{\Gamma'}$ and $A_{\Theta}$. The former fact implies that the conjugating element lies in $A_{\Gamma'}$, and the latter that it lies in $A_\Theta \times A_{\partial \Gamma'}$. Hence the conjugating element lies in $A_{\partial \Gamma'}$. But such a conjugation is trivial when we restrict the action to $\Out(A_\Theta)$, and hence our chosen representatives lift $H \to \Out(A_\Theta)$ to $H \to \Aut(A_\Theta)$ as required.
\end{proof}

Claim~\ref{claim: 9} allows us to use Lemma~\ref{lem: fixed point for liftable actions}, and conclude that the cube complex $X''_\Theta$ contains an $H$-fixed point $r$, with a chosen lift $\t r$ in the universal cover.


Let $\widetilde R$ denote the standard copy of $\widetilde
X''_{{\partial \Gamma'} }$ in $\t X''_{\overline \Theta}$
determined by $\widetilde r$. Its projection $R$ in $X_{\overline
  \Theta}$ is preserved by $H$, since $r$ is.

On the other hand, let $\widetilde Z$ be the standard copy of
$\widetilde X'_{\partial \Gamma'}$ in $\widetilde X'_{\Gamma'}$; this is unique since $\lk_{\Gamma'}(\partial \Gamma') = \emptyset$ by the assumption of Case~1.
The standard copies $\widetilde R$ and $\widetilde Z$ are naturally
isometric since both are standard copies of the same cube complex.
We now form $\widetilde X_\Gamma$ by gluing $\widetilde X'_{\Gamma'}$ to
$\widetilde X''_{\overline \Theta} $ via the natural isomorphism
$\widetilde Z = \widetilde R$.
\begin{claim}
  The $H$-action on $X_\Gamma$ inherited from the gluing is the correct one.
\end{claim}
\begin{proof}
Take $h \in H$. Let us pick a basepoint $\widetilde p \in \widetilde Z = \widetilde R$, and let $p$ denote its projection in $Z = R$.
Note that choosing $\widetilde p$ fixes an isomorphism between
fundamental groups of $Z, X'_{\Gamma'}$, and
$X''_{{\overline \Theta}}$ (based at $p$), and the groups $A_{\partial
  \Gamma'}, A_{\Gamma'}$, and $A_{{\overline
    \Theta}}$ respectively.

Choose a path $\gamma(h)$ in $Z$ connecting $p$ to $h.p$. Let
$h_{\partial \Gamma'} \in \Aut(A_{\partial \Gamma'})$ be the geometric representative of the
restriction of $\phi(h)$.
Note that we can construct any representative of the restriction of $\phi(h)$ this way, so let us choose the path so that $h_{\partial \Gamma'}$ is equal to the restriction of $h_2 \in \Aut(A_\Gamma)$ from the previous claim.

Using the same basepoint and path
we obtain geometric representatives $h_{\overline \Theta } \in \Aut(A_{\overline \Theta})$ and
$h_{\Gamma'} \in \Aut(A_{\Gamma'})$. Each of these representatives can be
(algebraically) extended to a representative of $\phi(h)$ in
$\Aut(A_\Gamma)$, but this is in general not unique; two such
extensions will differ by conjugation by an element of the
centraliser of $A_{\overline \Theta }$ and $A_{\Gamma'}$ respectively.

Since $\lk(\partial \Gamma') \cap \Gamma' = \emptyset$ (which is the
assumption of Case~1 we are in), $Z(\Gamma') \subseteq \partial
\Gamma'$. Therefore each vertex in $\Theta$ is connected to each vertex in $Z(\Gamma')$ by a single edge. Hence, if $Z(\Gamma')$ is non-empty, then $\Gamma$ is a cone
(and hence a join) over any vertex of it. As this would be a
contradiction, we see that $\Gamma'$ has trivial centre. It also has a
trivial link, as we have shown in \cref{trivial link}.
Therefore there is a unique way of extending $h_{\Gamma'}$ to an
automorphism of $A_\Gamma$; we will continue to denote this extension
by $h_{\Gamma'}$.

Note that $h_{\Gamma'}$ and $h_2$ both preserve $A_{\Gamma'}$, and by construction they agree when restricted to $\Aut(A_{\partial \Gamma'})$. These two facts imply that $h_{\Gamma'} {h_2}^{-1} \in \Aut(A_\Gamma)$ is equal to the conjugation by an element of $A_{\Gamma'} \cap C(A_{\partial \Gamma'}) = Z(A_{\partial \Gamma'})$.

\smallskip
Now let us look more closely at $h_{\overline \Theta}$. When
restricted to $A_{\partial \Gamma'}$, it is equal to the restriction of $h_2$. The same is true when restricted to $A_\Theta$, since in this case this is exactly the geometric representative constructed using $\t r$ and no path at all. Hence, picking any representative $h_3 \in \Aut(A_\Gamma)$ of $h_{\overline \Theta}$, we see that $h_3 {h_2}^{-1}$ is equal to the conjugation by an element of $Z(A_{\overline \Theta}) = A_{Z(\partial \Gamma')} \times A_{Z(\Theta)}$. Hence the identical statement holds for $h_{\Gamma'} {h_3}^{-1}$. 

\smallskip
Let us now go back to the action of $h$ on the glued-up complex
$X_\Gamma$. Again we use $\t p$ and $\gamma(h)$ to
obtain a geometric representative in $\Aut(A_\Gamma)$. On the
subgroup $A_{\Gamma'}$ this automorphism is equal to $h_{\Gamma'}$. On the subgroup $A_{\overline \Theta}$, it
is equal to $h_3$, but also to $h_{\Gamma'}$, since
conjugation by elements in $Z(A_{\overline \Theta})$ is trivial here.
Thus $h$ acts as the outer automorphism $\phi(h)$, as required.
\end{proof}

The System Intersection Axiom gives a standard copy $\t Z_\Sigma$ of $\t X'_{\Sigma}$ in $\t X'_{\Gamma'}$ for each $\Sigma \in \S_{\Gamma'}$. Note that this includes $\Delta \in \S_{\Gamma'}$; we have $\t Z = \t Z_\Delta$ since $\lk_{\Gamma'}(\Delta) \subseteq \lk_{\Gamma'}(\partial \Gamma') = \emptyset$ (by the assumption of Case~1 we are in), and so there is only one standard copy of $\t X'_\Delta$ in $\t X'_{\Gamma'}$.

We have thus completed the construction in Case~1.

\subsubsection*{Case 2: $\lk(\partial\Gamma')\cap\Gamma' \neq \emptyset$}

\begin{figure}
\label{fig: case 2}
\includegraphics{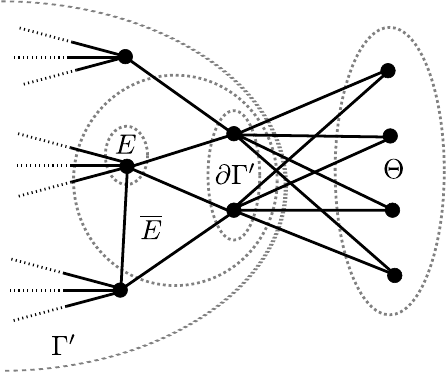}
\caption{The relevant subgraphs of $\Gamma$ in case 2}
\end{figure}

Let
$$E = \lk(\partial \Gamma')_1$$
be the part of the link
of $\partial \Gamma$ lying in $\Gamma'$. We define
$$\overline E = E \ast \partial \Gamma'$$
analogous to the definition of $\overline \Theta = \Theta
\ast \partial \Gamma'$. In general we put
\[ \overline \Sigma = \Sigma \cup \partial \Gamma' \]
for any subgraph $\Sigma \subseteq \Gamma$.

Note that we may assume that $\overline E \neq \Gamma'$, as otherwise
$\Gamma$ is a join, namely $\st(\partial \Gamma')$. We have $\partial
\Gamma' = \lk(\Theta \cup E)$ and thus $\partial\Gamma' \in \L^\phi$,
since it is a link of a non-cone. Thus $\overline E = \st(\partial \Gamma')_1 \in \L^\phi$, by part (4) of Lemma~\ref{lem: extended stars in L} and by Lemma~\ref{lem: intersections in L}.

\begin{claim}
  There exists a cubical system $\X''$ for $\L^\phi_{\st(\partial
    \Gamma')}$ which extends the subsystem $\X'_{\overline E}$ of
  $\X'$ strongly.
\end{claim}
\begin{proof}
  We begin by applying Lemma~\ref{lem: case 1}, using the decomposition
  \[ \st(\partial \Gamma') = \partial \Gamma' \ast (E \cup \Theta)\]
   to obtain a cubical system
  $\X''$ for $\L^\phi_{\st(\partial \Gamma')}$ which extends the subsystem
  $\X'_{\overline E}$ of $\X'$, in such a way that it extends
  $\X'_{\partial \Gamma'}$ strongly.  If the join decomposition
  (compare Definition~\ref{def:join-decomp}) of $E$ does not contain
  singletons, then in fact $\X''$ extends $\X'_{\overline E}$ strongly
  by definition.

  Suppose that the decomposition does contain singletons; note that each such singleton belongs to $\L^\phi$, since $\phi$ is link-preserving.
  Moreover, such a singleton is also a singleton in the join decomposition of $\overline E = E \ast \partial \Gamma'$.

   Let $s$ be such a singleton which is
  not connected to every other vertex in $\Gamma'$. Then
  Lemma~\ref{lem: rotating circles} implies that the action on the
  associated loop in $\X'$ and $\X''$ is the same.

  The remaining case
  occurs when $E$ contains a singleton $s$ in its join
  decomposition, and $\st(s) = \Gamma'$. In this situation we will
  replace $\X'$ by the system given by Lemma~\ref{lem: replacing circle}, so that $\X'_{\{s\}}$ and $\X''_{\{s\}}$ are
  $H$-equivariantly isometric; we will continue to denote this
  slightly altered system by $\X'$.  Repeating this operation for each
  singleton as described guarantees that $\X''$ extends
  $\X'_{\overline E}$ strongly.
\end{proof}

We need to introduce one more graph. Let
\[ \Delta ' = \bigcap_{\Sigma \in \S} \st(\Sigma_1)_1 \]
Note that $\Delta' \subseteq \Gamma'$. The significance of $\Delta'$ will be explained shortly. For now let us observe the following.

\begin{claim}
$\Delta' \cup \Theta \in \L^\phi$.
\end{claim}
\begin{proof}
We have
\[ \st(\Sigma_1)_1 \cup \Theta = \st(\Sigma_1) \cup \Sigma \in \L^\phi \]
since $\lk(\st(\Sigma_1) \cap \Sigma) \subseteq \lk(\Sigma_1) \subseteq \st(\Sigma_1) = \st(\st(\Sigma_1))$, and by part ii) of Lemma~\ref{lem: intersections in L}.
Thus
\[ \Delta' \cup \Theta =  \bigcap_{\Sigma \in \S} (\st(\Sigma_1)_1 \cup \Theta) \in \L^\phi \]
by part i) of Lemma~\ref{lem: intersections in L}.
\end{proof}

By construction we have $\overline \Theta \subseteq \Delta' \cup \Theta$ and so $\Delta' \in \S_{\Gamma'}$. We also have $\Delta' \subseteq \overline E$ since $\overline E \in \S_{\Gamma'}$ satisfies $\st(\overline E)_1 = \overline E$.

\begin{claim}
There exist fixed standard copies $\t R_\Delta$ and $\t R_{\Delta'}$ of, respectively,  $\t X''_\Delta$ and $\t X''_{\Delta'}$ in $\t X''_{\st(\partial \Gamma')}$ such that $\t R_\Delta \subseteq \t R_{\Delta'}$.
\end{claim}
\begin{proof}
We have \[\lk(\Delta' \cup \Theta) \subseteq \lk(\Theta) \s- \Delta' \subseteq \partial \Gamma' \s- \partial \Gamma' = \emptyset\]
Thus there is a unique (and therefore fixed) standard copy of $\t X''_{\Delta' \cup \Theta}$ in $\t X''_{\st(\partial \Gamma')}$. Since
\[\lk(\overline E) = \lk(\partial \Gamma') \cap \lk(E) \subseteq \lk(\partial \Gamma') \s- (E \cup \Theta) = \emptyset\]
there is a unique (and therefore again fixed) standard copy of $\t X''_{\overline E}$ in $\t X''_{\st(\partial \Gamma')}$. These two copies intersect in a standard copy of $\t X''_{\Delta'}$ (by the Matching Property); let us denote it by $\t R_{\Delta'}$. Since the two copies are fixed, so is $\t R_{\Delta'}$.

By noting that
\[\lk(\Delta \cup \Theta) \subseteq \lk(\Theta) \s- \Delta \subseteq \partial \Gamma' \s- \partial \Gamma' = \emptyset\]
and applying the same procedure we obtain a fixed standard copy $\t R_{\Delta}$ of $\t X''_\Delta$ in $\t X''_{\st(\partial \Gamma')}$.

Since the unique standard copy of $\t X''_{\Delta \cup \Theta}$ must be contained in the unique standard copy of $\t X''_{\Delta' \cup \Theta}$ (by the Composition Property), we have \[\t R_\Delta \subseteq \t R_{\Delta'}\]
as claimed.
\end{proof}

The System Intersection Axiom gives a standard copy $\t Y_\Sigma$ of
$\t X'_{\Sigma}$ in $\t X'_{\Gamma'}$ for each $\Sigma \in
\S_{\Gamma'}$ such that all these copies intersect in $\t Y_\Delta$.
For each $\Sigma \in \S_{\Gamma'}$ we have
\[ \lk_{\Gamma'}(\st(\Sigma)_1) \subseteq \lk_{\Gamma'}(\Sigma) \s- \st(\Sigma)_1 \subseteq \lk(\Sigma)_1 \s- \lk(\Sigma)_1 = \emptyset\]
and thus the copy $\t Y_{\st(\Sigma)_1}$ is fixed.
Note that $\Delta' \in \S_{\Gamma'}$, and by the Intersection Axiom we have
\[ \t Y_{\Delta'} = \bigcap_{\Sigma \in \S_{\Gamma'}} \t Y_{\st(\Sigma)_1} \]
Thus, the standard copy $\t Y_{\Delta'}$ is fixed.

Let us now observe the crucial property of $\Delta'$.

\begin{claim}
\label{claim: 14}
Let $\Sigma \in \S$, and let $\t p \in \t Y_{\Delta'}$. Then there exists a standard copy $\t W_{\Sigma_1}$ of $\t X'_{\Sigma_1}$ in $\t X'_{\Gamma'}$ such that $\t p \in \t W_{\Sigma_1}$.
\end{claim}
\begin{proof}
Note that we have a standard copy $\t Y_{\st(\Sigma_1)_1}$ which contains $\t Y_{\Delta'}$, and hence $\t p$. We have
\[ \st(\Sigma_1)_1 = \Sigma_1 \ast \lk_{\Gamma'}(\Sigma_1) \]
and so the Product Axiom and Composition Axiom tell us that there exists a standard copy $\t W_{\Sigma_1}$ as required.
\end{proof}

We obtain a complex $\t C$ from $\t X'_{\Gamma'}$ and $\t
X''_{\st(\partial \Gamma')}$ by gluing $\t Y_{\Delta'} = \t
R_{\Delta'}$. However, the projection $C$ of our complex $\t C$ might not
yet realise the desired action of $H$.

Pick $h \in H$. Take a point $\t r \in \t R_\Delta$ as a basepoint, let $r$ be its projection as usual. Take a path $\gamma(h)$ in $R_\Delta$ from $r$ to $h.r$. Note that $\t r$ is a point in $\t R_{\Delta'}$, and hence in $\t C$ (after the gluing). Let $\t p \in \t Y_{\Delta'} = \t R_{\Delta'}$ be the corresponding point; we view it as a point in $\t C$ as well, and denote its projection by $p$ as usual.

We obtain two geometric representatives of $\phi(h)$ this way, $h_r$ obtained using $r$ and $\gamma(h)$, and $h_p$ obtained using $p$ and the corresponding path. 

\begin{claim}
The gluing $\t C$ is faulty within $Z(A_{\Delta'})$.
\end{claim}
\begin{proof}
Let us first remark that $\lk(\Sigma) \subseteq \lk(\Theta) \s- \partial \Gamma' = \emptyset$ for all $\Sigma \in \S$.

By construction we see that
\[ h_r(A_{\Delta \cup \Theta}) = A_{\Delta \cup \Theta} \]
Take $\Sigma \in \S$. Since $\Sigma$ lies in $\L^\phi$, we have
\[ h_r(A_\Sigma) = {A_\Sigma}^y\]
for some element $y \in A_\Gamma$. But then
\[ A_{\Delta \cup \Theta} = h_r(A_{\Delta \cup \Theta}) \leqslant h_r(A_\Sigma) = {A_\Sigma}^y \]
Proposition~\ref{prop: ccv} implies that $y \in N(A_\Sigma) N(A_{\Delta \cup \Theta}) = A_\Sigma A_{\Delta \cup \Theta} = A_\Sigma$, and thus
\[ h_r(A_\Sigma) = A_\Sigma\]
Now $\Sigma_1 \in \L^\phi$ and so
\[ h_r(A_{\Sigma_1}) = {A_{\Sigma_1}}^z\]
We have ${A_{\Sigma_1}}^z \subseteq A_\Sigma$, and so (using
Proposition~\ref{prop: ccv} again) we get, without loss of generality,
$z \in N(A_{\Sigma}) = A_\Sigma$ since $\lk(\Sigma) = \emptyset$ (as $\Sigma \in \S$).

\smallskip
Let us now focus on $h_p$. \cref{claim: 14} tells us that there exists a standard copy of $\t X'_{\Sigma_1}$ containing $\t p$; it is clear that it will also contain the copy of the path $\gamma(h)$. Hence
\[ h_p(A_{\Sigma_1}) = A_{\Sigma_1}\]
and so
\[ {A_{\Sigma_1}}^z = h_r(A_{\Sigma_1}) = c(x(h)^{-1})  h_p(A_{\Sigma_1}) = {A_{\Sigma_1}}^{x(h)^{-1}} \]
where
\[ c(x(h)) = h_p h_r^{-1} \]
and $x(h) \in C(A_{\Delta'})$.

The construction of $h_r$ also tells us that $A_{\overline E} = h_r(A_{\overline E})$, since $\t r$ and $\gamma(h)$ lie in the fixed standard copy of $\t X''_{\overline E}$ in $\t X''_{\st(\partial \Gamma')}$. We have $A_{\overline E} = h_p(A_{\overline E})$ as well, since $\overline E \in \S_{\Gamma'}$. So
\[ x(h) \in N(A_{\overline E}) \leqslant A_{\Gamma'}\]
Thus, by Proposition~\ref{prop: ccv},  $z \in N(A_{\Sigma_1}) A_{\Gamma'}$.

We claim that ${A_{\Sigma_1}}^z = A_{\Sigma_1}$. If $\Theta \cap \lk(\Sigma_1) \neq \emptyset$, then $\Sigma_1 = \partial \Gamma'$, and so
\[ z \in A_\Sigma \leq N(A_{\Sigma_1})\]
which yields the desired conclusion.

Otherwise we have $\st(\Sigma_1) \subseteq \Gamma'$, and so $N(A_{\Sigma_1}) \leqslant A_{\Gamma'}$, which in turn implies $z \in A_{\Gamma'}$.
Now $z \in A_{\Sigma} \cap A_{\Gamma'} = A_{\Sigma_1}$ and so ${A_{\Sigma_1}}^z = A_{\Sigma_1}$ as claimed.

\smallskip
We have
\[{A_{\Sigma_1}}^{x(h)^{-1}} = {A_{\Sigma_1}}^z = A_{\Sigma_1} \]
for all $\Sigma \in \S$. Hence
\[ x(h) \in \bigcap_{\Sigma \in \S} A_{\st(\Sigma_1)}\]
We have already shown that $x(h) \in A_{\Gamma'}$ and so
\[ x(h) \in \bigcap_{\Sigma \in \S} A_{\st(\Sigma_1)_1} = A_{\Delta'}\]
Therefore $x(h) \in C(A_{\Delta'}) \cap A_{\Delta'} = Z(A_\Delta')$.
This statement holds for
each $h$, and thus
the fault of our gluing satisfies the claim.
\end{proof}

Now we are in a position to apply Proposition~\ref{prop: gluing} and obtain a new glued up complex, which we call $\t X_\Gamma$, which realises our action $\phi$.

Recall that we have a standard copy $\t R_{\Delta}$ in $\t
X''_{\st(\partial \Gamma')}$. The gluing sends $\t R_\Delta$ to some
standard copy of $\t X'_\Delta$ in $\t X'_{\Gamma'}$, which lies
within $\t Y_{\Delta'}$ (we are using the Composition Property here);
let us denote this standard copy by $\t Z_\Delta$. It is fixed since
$\t R_\Delta$ is fixed.

By Claim~\ref{claim: 14}, we may pick a standard copy
$\t Z_\Sigma$ of $\t X'_{\Sigma}$ in $\t X'_{\Gamma'}$ for each $\Sigma \in \S_{\Gamma'}$
such that they will all intersect in $\t Z_\Delta$. Let us choose such a family of standard copies.

To finish the construction in this case we need to remark that the complex $\t X_\Gamma$ we constructed is equal to a complex obtained from
 $\t X'_{\Gamma'}$ and $\t X''_{\Delta \cup \Theta}$ by gluing $\t R_\Delta$ and $\t Z_\Delta$.

\subsection*{Step 2: 
  Constructing $X_\Sigma$ for $\Sigma \subseteq
  \Gamma'$ or $\Sigma \subseteq \Delta\cup\Theta$}

Since $\X'$ and $\X''_{\Delta \cup \Theta}$ strongly extend
$\X'_\Delta$, and $\Delta = \Gamma' \cap (\Delta \cup \Theta)$, we
simply define the complexes in $\X$ for graphs $\Sigma \subseteq
\Gamma'$ or $\Sigma \subseteq \Delta\cup\Theta$ to be the ones in $\X'$
or $\X''_{\Delta \cup \Theta}$, respectively.

\subsection*{Interlude}

Before we begin Step~3, we record the following
\begin{claim}
  For any graph $\Sigma \in \L^\phi$ such that $\Sigma \not\subseteq \Gamma'$ and
  $\Sigma \not\subseteq \overline \Theta$ we have $\Theta \subseteq
  \Sigma$.

  Additionally, every graph $\Sigma \in \L^\phi$ with $\Theta \subseteq
  \Sigma$ satisfies
  \begin{equation}
    \label{eqn: star}
    \lk(\partial \Gamma') \cap (\Theta \cup \Delta) \subseteq \Sigma
    \tag{$\ast$}
  \end{equation}
\end{claim}
In particular, $\overline \Sigma = \Sigma \cup \Delta$ for subgraphs
$\Sigma$ with \eqref{eqn: star}, and so $\overline \Sigma \in \S$ in this case.

\begin{proof}
First assume that $\Sigma \not\subseteq \Gamma'$ and $\Sigma
\not\subseteq \overline \Theta$. In this case $\lk(\Sigma_1) \subseteq
\Gamma'$, since otherwise $\Sigma_1 \subseteq \partial \Gamma' $,
which would force $\Sigma \subseteq \overline \Theta$, a contradiction.
Therefore, by Lemma~\ref{lem: intersections in L}, $\overline \Sigma =
\Sigma \cup \Gamma' \in \L^\phi$. Thus
\[\Theta \subseteq \Sigma\] since there is no
 subgraph in $\L^\phi$ which is properly contained in
$\Gamma$ and properly contains $\Gamma'$ (recall that $\Gamma'$ is maximal among proper subgraphs of $\Gamma$ in $\L^\phi$, and this is why we had to first check that $\overline \Sigma \in \L^\phi$).

\smallskip
Now let us suppose that $\Sigma \in \L^\phi$ satisfies $\Theta
\subseteq \Sigma$.
Then
\[\lk(\Sigma \cap (\Delta \cup \Theta)) \subseteq \lk(\Theta) \subseteq \Delta \cup \Theta\]
and so $\Sigma \cup \Delta \in \L^\phi$ by Lemma~\ref{lem: intersections in L}.

We now claim that $\lk(\partial \Gamma') \cap \Delta \subseteq \Sigma$. Note that $E \cup \Theta  = \lk(\partial \Gamma') \in \L^\phi$ since $\phi$ is link-preserving.\link-preserving Now $\Sigma \cap (E \cup \Theta) \in \L^\phi$ and thus
\[ \partial \Gamma' \ast (\Sigma \cap (E \cup \Theta)) = \st(\Sigma \cap (E \cup \Theta)) \in \L^\phi \]
where the equality follows from the observation that $\Theta \subseteq \Sigma \cap (E \cup \Theta)$.
But $\partial \Gamma' \ast (\Sigma \cap (E \cup \Theta)) \in \S$, and therefore
\[ \Delta \subseteq \partial \Gamma' \ast (\Sigma \cap (E \cup \Theta))\]
This in turn implies that
\[ \Delta \subseteq \partial \Gamma' \cup \Sigma \]
and so $\Delta \cap \lk(\partial \Gamma') \subseteq \Sigma$. We assumed that $\Theta \subseteq \Sigma$, and so
\[ \lk(\partial \Gamma') \cap (\Theta \cup \Delta) \subseteq \Sigma \]
that is $\Sigma$ satisfies \eqref{eqn: star}.

Lastly, suppose that $\Sigma$ satisfies \eqref{eqn: star}. Then
\[ \overline \Sigma = \Sigma \cup \partial \Gamma' = \Sigma \cup (\Delta \s- (\Delta \cap \lk(\partial \Gamma')) = \Sigma \cup \Delta \]
as required.
\end{proof}

Since  $\Sigma \not\subseteq \Delta \cup \Theta$ implies  $\Sigma \not\subseteq \overline \Theta$, we see that any $\Sigma \in \L^\phi$ not covered by Step~2 satisfies \eqref{eqn: star} and so $\overline \Sigma \in \S$. Hence given such a $\Sigma$ we have
\[\Sigma \subseteq \st_{\overline \Sigma}(\Sigma) \subseteq \overline \Sigma\]
 We will deal with each graph in this chain of inclusions in turn.



\subsection*{Step 3: Constructing $X_\Sigma $ for $\Sigma \in \S \s- \{ \Delta \cup \Theta \}$}

In this case
\[\Sigma_1 = \Sigma \cap \Gamma' \in \S_{\Gamma'}\]
 and so (by construction) we have $\t Z_{\Sigma_1}$ in $\t X_\Gamma$ containing $\widetilde Z_\Delta$ (which in particular implies that it is fixed).
We also have the image under our gluing map $\t \iota_{\Delta \cup
  \Theta,\Gamma}$ of $\t X''_{\Delta \cup \Theta} = \t X_{\Delta \cup \Theta}$; let us call it $\t R_{\Delta \cup \Theta}$. This subcomplex contains $\widetilde Z_\Delta$ by construction.
We obtain $\t X_\Sigma$ from the two complexes by gluing the two
instances of $\widetilde Z_\Delta$. Its projection carries the desired
marking by construction. The action of $H$ is also the desired one;
taking any $h \in H$, and looking at a geometric representative obtained by choosing a basepoint and a path in the subcomplex $X_\Sigma$, we get an automorphism of $A_\Gamma$ which preserves $A_\Sigma$. This automorphism is a representative of $\phi(h)$, and so the restriction to $A_\Sigma$ is the desired one. But this is equal to the geometric representative of the action of $h$ on $X_\Sigma$ obtained using the same basepoint and path.

Our construction also gives us maps $\widetilde \iota_{\Sigma_1, \Sigma}$ and $\widetilde \iota_{\Delta \cup \Theta, \Sigma}$.
These maps are as required, since  $\Delta \cup \Theta$ and $\Sigma_1$ both have trivial links in $\Sigma$; the former statement is clear, and the latter follows from the observation that
\[\lk_\Sigma(\Sigma_1) \neq \emptyset\]
 implies that $\Sigma_1 \subseteq \partial \Gamma'$ and hence $\Sigma_1 = \Delta = \partial \Gamma'$, which in turn gives $\Sigma = \Delta \cup \Theta$, contradicting our assumption.

We also get a map $\t \iota_{\Sigma, \Gamma}$, since we define $\t X_\Sigma$ as a subcomplex of $\t X_\Gamma$. It is as required since $\lk(\Sigma) = \emptyset$ for all $\Sigma \in \S$.

\subsection*{Step 4: Constructing $X_\Sigma $ for $\Sigma$ with \eqref{eqn: star} and such that $\lk_{\overline \Sigma}(\Sigma) = \emptyset$}

By the Composition Property there exists a standard copy of $\t X_{\partial
  \Gamma'}$ in $\t X_{\Delta \cup \Theta}$ which lies within $\t R_\Delta$; let us denote it by $\t R_{\partial \Gamma'}$. The gluing
$\t R_\Delta = \t Z_\Delta$ gives us the corresponding standard copy $\t Z_{\partial \Gamma'}$ in $\t Z_\Delta$.

Note that the assumption
implies that $\lk_{\overline \Sigma_1}(\Sigma_1) = \emptyset$.
Let us take the unique standard copy of $\widetilde X_{\Sigma_1}$ in
$\widetilde X_{\overline \Sigma_1}$; we will denote it by $\t Z_{\Sigma_1}$.
By the Matching Property, it intersects $\t Z_\Delta$ in a standard copy of $\widetilde X_{\Sigma \cap \Delta}$.
Using the Matching Property again, this time in $\t Z_\Delta$, we see that this copy of $\widetilde X_{\Sigma \cap \Delta}$ intersects $\t Z_{\partial \Gamma'}$ in a copy of $\t X_{\Sigma \cap \partial \Gamma'}$; we will denote thus copy by $\t Z_{\Sigma \cap \partial \Gamma'}$, and the corresponding one in $\t R_{\partial \Gamma'}$ by $\t R_{\Sigma \cap \partial \Gamma'}$.

Recall that $\Sigma_2 = \Sigma \cap (\Delta \cup \Theta)$.
Since $\X''$ is a product of $\X''_{\partial \Gamma'}$ and $\X''_{\lk_{\Delta \cup \Theta}(\partial \Gamma')}$, there exists a standard copy of $\t X_{\Sigma_2}$  in $\t X_{\Delta \cup \Theta}$ which contains $\t R_{\Sigma \cap \partial \Gamma'}$; we will denote it by $\t R_{\Sigma_2}$.
We define $\t X_\Sigma$ to be the subcomplex of $\t X_{\overline \Sigma}$ obtained by gluing $\t Z_{\Sigma_1}$ to $\t R_{\Sigma_2}$. Note that these two copies overlap in a copy of $\t X_{\Sigma \cap \Delta}$, which is the unique such copy containing $\t R_{\Sigma \cap \partial \Gamma'}$.

Note that again our gluing procedure determines maps $\widetilde
\iota_{\Sigma_1, \Sigma}$ and $\widetilde \iota_{\Sigma_2, \Sigma}$ of
the required type.

From this construction we also obtain a map $\t \iota_{\Sigma,
  \overline \Sigma}$, since we define $\t X_\Sigma$ as a subcomplex of $\t X_{\overline \Sigma}$.

\subsection*{Step 5: Constructing the remaining complexes.}

As remarked above we are left with graphs $\Sigma \in \L^\phi$ satisfying \eqref{eqn: star} and such that $\Sigma \subset \st_{\overline \Sigma}(\Sigma)$ is a proper subgraph. Let $\Sigma' = \st_{\overline \Sigma}(\Sigma)$, and note that $\Sigma'$ is covered by the previous step. We need to exhibit a product structure on $\t X_{\Sigma'}$, one factor of which will be the desired complex for $\Sigma$, the other for $\Lambda = \lk_{\overline \Sigma}(\Sigma)$.

The complex $\t X_{\Sigma'}$ is obtained from complexes $\t X_{\Sigma'_1}$ and $\t X_{\Sigma'_2}$ by gluing them along a copy of $\t X_{\Sigma' \cap \Delta}$. Since $\Lambda \subseteq \partial \Gamma'$, each of these three complexes is a product of $\t X_\Lambda$ and some other complex by the Product Axiom.
Moreover, the embeddings $\t X_{\Sigma' \cap \Delta} \to \t X_{\Sigma'_i}$ with $i \in \{1,2\}$ respect the product structure, that is the image of any standard copy of $\t X_\Lambda$ in $\t X_{\Sigma' \cap \Delta}$ is still a standard copy in $\t X_{\Sigma'_i}$ by the Composition Property. Hence the glued-up complex $\t X_{\Sigma'}$ is a product of $\t X_\Lambda$ and a complex obtained by gluing some standard copies of $\t X_{\Sigma_i}$ in $\t X_{\Sigma'_i}$ along $\t X_{\Sigma \cap \Delta}$; we call this latter complex $\t X_\Sigma$.
For notational convenience we pick some such standard copies of $\t X_{\Sigma_1}$ and $\t X_{\Sigma_2}$, and denote them by $\t Z_{\Sigma_1}$ and $\t R_{\Sigma_2}$ respectively.

The gluing of the standard copies of $\t X_{\Sigma_i}$ gives us maps $\t \iota_{\Sigma_i, \Sigma}$ for both values of $i$, which are as required.

The construction also gives us a map \[\t \iota_{\Sigma, \Sigma'} \colon \t X_\Sigma \times \t X_\Lambda \to \t X_{\Sigma'}\]
which is again as wanted.

\subsection{Constructing the maps}

Let $\Sigma, \Sigma' \in \L^\phi$ be such that $\Sigma \subseteq \Sigma'$. We need to construct a map $\t \iota_{\Sigma, \Sigma'}$. We will do it in several steps.

\begin{enumerate}
 \item
\textbf{ $\Sigma' = \Sigma'_i$ for some $i\in \{1,2\}$}

In this case the cube complexes $X_\Sigma$ and $X_{\Sigma'}$ are
  obtained directly from another cubical system (in Step~2), and
  we take $\t \iota_{\Sigma, \Sigma'}$ to be the map coming from that system.

We will now assume that the hypothesis of this step is not satisfied, which implies that $\Sigma'$ satisfies \eqref{eqn: star} and that $\Sigma' \neq \Delta \cup \Theta$.

\item \textbf{$\Sigma = \Sigma_i$ and $\Sigma \neq \Sigma_j$ with $\{i,j\} = \{1,2\}$}

In this case we have $\st(\Sigma) = \st(\Sigma)_i$, since if $\lk(\Sigma) \neq \lk(\Sigma)_i$ then (knowing that $\Sigma = \Sigma_i$) we must have $\Sigma \subseteq \partial \Gamma'$, and thus $\Sigma=\Sigma_j$ which contradicts the assumption.
Therefore $\lk_{\Sigma'}(\Sigma) = \lk_{\Sigma'}(\Sigma)_i$ as well.
We define
\[ \t \iota_{\Sigma, \Sigma'} = \t \iota_{\Sigma'_i, \Sigma'} \circ \t \iota_{\Sigma, \Sigma'_i} \]
where the last map was defined in the previous step, and the map $\t
\iota_{\Sigma'_i, \Sigma'}$ was constructed together with the complex
$\t X_{\Sigma'}$ in Step~3,~4~or~5 of Subsection~\ref{subsec: constructing complexes}.

\item \textbf{ $\Sigma$ satisfies \eqref{eqn: star} and $\Sigma \neq \Delta \cup \Theta$}



Observe that $\st_{\overline \Sigma}(\Sigma) = \st_{\overline {\Sigma'}}(\Sigma)$ since $\lk(\Sigma) \subseteq \partial \Gamma'$.
Hence Step~5 above gives us the map $\t \iota_{\Sigma, \st_{\overline {\Sigma'}}(\Sigma)}$.

Let $\Omega = \st_{\overline {\Sigma'}}(\Sigma) \cup \st_{\overline {\Sigma'}}(\Sigma')$

\begin{claim}
$\Omega \in \L^\phi$.
\end{claim}
\begin{proof}
We use part (3) of Lemma~\ref{lem: intersections in L}. Thus we only need to observe that
\[ \lk\big(\st_{\overline {\Sigma'}}(\Sigma) \cap \st_{\overline
  {\Sigma'}}(\Sigma')\big) \subseteq \lk(\Sigma) \subseteq \st(\Sigma)
= \st_{\overline {\Sigma'}}(\Sigma)\]
since $\lk(\Sigma)\subseteq \partial \Gamma'$ as $\Sigma$ satisfies \eqref{eqn: star}.
\end{proof}

We have
\[ \lk(\st_{\overline \Sigma}(\Sigma)) \subseteq \lk(\Sigma) \s- \st_{\overline \Sigma}(\Sigma) = \emptyset\]
since $\lk(\Sigma) \subseteq \partial \Gamma'$ and so $\lk(\Sigma) = \lk_{\overline \Sigma}(\Sigma)$.
Analogously we have $\lk(\st_{\overline
  {\Sigma'}}(\Sigma')) = \emptyset$. It immediately follows that $\lk(\Omega) = \emptyset$ as well.
Now the complex $\t X_{\st_{\overline \Sigma}(\Sigma)}$ is formed by gluing $\t Z_{\st_{\overline \Sigma}(\Sigma)_1}$ and $\t R_{\st_{\overline \Sigma}(\Sigma)_2}$; we have similar statements for $\t X_{\st_{\overline {\Sigma'}}(\Sigma')}$ and $\t X_{\Omega}$.
By construction (see Step 4), $\t Z_{\st_{\overline \Sigma}(\Sigma)_1}$ and $\t Z_{{\Omega}_1}$ are both unique standard copies of the corresponding complexes in $\t Z_{\overline \Sigma_1}$. The Composition Property  in $\X'$ implies that there exists a standard copy of $\t X_{\st_{\overline \Sigma}(\Sigma)_1}$ contained in  $\t Z_{{\Omega}_1}$, and thus
\[ \t Z_{\st_{\overline \Sigma}(\Sigma)_1} \subseteq \t Z_{{\Omega}_1} \]
Similarly
\[ \t Z_{\st_{\overline {\Sigma'}}(\Sigma')_1} \subseteq \t Z_{{\Omega}_1} \]
Now the Matching Property (in $\X'$) tells us that
$\t Z_{\st_{\overline {\Sigma'}}(\Sigma')_1}$ and $\t Z_{\st_{\overline {\Sigma'}}(\Sigma')_1}$
intersect non-trivially.
Since both have a product structure, we find standard copies $\t Z_{\Sigma_1}$ and $\t Z_{\Sigma'_1}$ of $\t X_{\Sigma_1}$ and $\t X_{\Sigma'_1}$ respectively (in $\t Z_{\Omega_1}$) which also intersect non-trivially.

Let us look more closely at the product structure of $\t Z_{\st_{\overline {\Sigma'}}(\Sigma)_1}$. It is isomorphic to a product of $\t X_{\Sigma_1}$ and $\t X_{\lk_{\overline {\Sigma'}}(\Sigma)_1}$. The latter complex contains a standard copy of $\t X_{\lk_{\Sigma'}(\Sigma)_1}$. Hence there is a standard copy $\t Z_{\st_{{\Sigma'}}(\Sigma)_1}$ of $\t X_{\st_{ {\Sigma'}}(\Sigma)_1}$ in $\t Z_{\st_{\overline {\Sigma'}}(\Sigma)_1}$ containing the chosen standard copy $\t Z_{\Sigma_1}$ (by the Intersection Axiom). Thus $\t Z_{\st_{{\Sigma'}}(\Sigma)_1}$ and $\t Z_{\Sigma'_1}$ intersect non-trivially, and so
the intersection Axiom tells us that
\[\t Z_{\st_{{\Sigma'}}(\Sigma)_1} \subseteq \t Z_{\Sigma'_1}\]

After the gluing this yields a map $\t \iota_{\st_{\Sigma'}(\Sigma), \Sigma'}$, and since $\t Z_{\st_{\Sigma}(\Sigma')_1}$ had a product structure, so does $\t X_{\st_{{\Sigma'}}(\Sigma)}$ (arguing as in Step~5 above). This gives us the desired map $\t \iota_{\Sigma, \Sigma'}$.

\item \textbf{$\Sigma \subseteq \partial \Gamma'$}

In this (last) case we observe that $\lk_{\Sigma'}(\Sigma)$ satisfies \eqref{eqn: star}, and so we have already constructed the map
\[\t \iota_{\lk_{\Sigma'}(\Sigma), \Sigma'} \colon \t X_{\lk_{\Sigma'}(\Sigma)} \times \t X_\Sigma \to \t X_{\Sigma'}\]
We define $\t \iota_{\Sigma, \Sigma'}$ by reordering the factors in the domain.
\end{enumerate}

\subsection{Verifying the axioms}

The first two axioms depend only on two subgraphs $\Sigma, \Sigma' \in \L^\phi$ with $\Sigma \subseteq \Sigma'$. This is the same assumption as in the maps part of our proof, and hence the verification of the two axioms will follow the structure as the construction of maps -- we will consider four cases, and the assumption in each will be identical to the assumptions of the corresponding case above.

\subsubsection*{Product Axiom} Suppose that $\Sigma' = \st_{\Sigma'}(\Sigma)$.
\begin{enumerate}
 \item In this case the Product
Axiom follows from the Product Axiom in $\X'$ or $\X''$. Otherwise we assume that $\Sigma'$ satisfies \eqref{eqn: star} and $\Sigma' \neq \Delta \cup \Theta$.

\item If $\Sigma = \Sigma_i$ and $\Sigma \neq \Sigma_j$ with $\{i,j\} = \{1,2\}$, then $\Sigma' = \st_{\Sigma'}(\Sigma)$ also satisfies $\Sigma' = \Sigma'_i$, and so we are in the previous case.

\item The standard copy $\t Z_{\st_{\Sigma'}(\Sigma)_1}$
used in case (3) above is equal to the image of $\t \iota_{\Sigma_1, \st_{\Sigma'}(\Sigma)_1}$, and hence the corresponding statement is still true after the gluing.

\item In this case we defined the map $\t \iota_{\Sigma, \Sigma'}$ using $\t \iota_{\lk_{\Sigma'}(\Sigma), \Sigma'}$, and the graph $\lk_{\Sigma'}(\Sigma)$ is covered by the previous cases.
\end{enumerate}

\subsubsection*{Orthogonal Axiom} Let $\Lambda = \lk_{\Sigma'}(\Sigma)$.

\begin{enumerate}
 \item If $\Sigma' \subseteq \Gamma'$ or $\Sigma' \subseteq \Delta \cup \Theta$ then the axiom is satisfied, since it is satisfied in $\X'$ and $\X''$. Otherwise we assume that $\Sigma'$ satisfies \eqref{eqn: star} and $\Sigma' \neq \Delta \cup \Theta$.

\item If $\Sigma = \Sigma_i$ and $\Sigma \neq \Sigma_j$ with $\{i,j\}
  = \{1,2\}$, then also $\Lambda = \Lambda_i$. Suppose that $\Lambda \neq \Lambda_j$.
In this case both maps $\t \iota_{\Sigma, \Sigma'}$ and $\t
\iota_{\Lambda, \Sigma'}$ factorise through $\t
\iota_{\st_{\Sigma'}(\Sigma), \Sigma'}$, and thus it is enough to
verify the axiom within $\t X_{\st_{\Sigma'}(\Sigma)}$. But
$\st_{\Sigma'}(\Sigma) = \st_{\Sigma'}(\Sigma)_i$ and so we are done by the previous case.

Now suppose that $\Lambda \subseteq \partial \Gamma'$.
In this case the axiom follows trivially from the construction of the map $\t \iota_{\Lambda, \Sigma'}$.

\item In this case we have $\Lambda \subseteq \partial \Gamma'$ and so we are done as above.

\item When $\Sigma \subseteq \partial \Gamma'$ we are again done by construction.
\end{enumerate}

\subsubsection*{Intersection Axiom}
Let us now verify that $\X$ satisfies the Intersection Axiom. Take
$\Sigma, \Sigma', \Omega \in \L^\phi$ such that $\Sigma \subseteq
\Omega$ and $\Sigma' \subseteq \Omega$, and let $\t Y_{\Sigma}$ and $\t Y_{\Sigma'}$ be standard copies of, respectively, ${\t
  X}_{\Sigma}$ and $\t X_{\Sigma'}$ in ${\t X}_\Omega$ with non-empty intersection. We need
to show that the intersection is the image of a standard copy of
$\Sigma \cap\Sigma'$ in each.

As in the first two cases, the details depend on the inclusions $\Sigma, \Sigma' \subseteq \Omega$. The cases will thus be labeled by pairs of integers $(n,m)$, the first determining in which step the map $\t \iota_{\Sigma, \Omega}$ was constructed, and the second playing the same role for $\t \iota_{\Sigma', \Omega}$. By symmetry we only need to consider $n \leqslant m$.

\begin{itemize}
 \item[(1,1)] In this case $\Omega = \Omega_i$, and so the axiom follows from the Intersection Axiom in $\X'$. In what follows we can assume that $\Omega \neq \Delta \cup \Theta$ satisfies \eqref{eqn: star}. Hence $\t X_\Omega$ is obtained by gluing $\t Z_{\Omega_1}$ and $\t R_{\Omega_2}$ along a subcomplex of $\t Z_\Delta$ which is a
  standard copy of $\t X_{\Delta \cap \Omega}$; let us denote it by $\t Z_{\Delta \cap \Omega}$.

\item[(2,2)] This splits into two cases. If $\Sigma = \Sigma_i$ and $\Sigma' = \Sigma'_i$ then both maps $\t \iota$ factor through $\t \iota_{\Omega_i, \Omega}$, and so the problem is reduced to checking the axiom for the triple $\Sigma, \Sigma', \Omega_i$, for which it holds.

In the other case we have, without loss of generality, $\Sigma = \Sigma_1$ and $\Sigma' = \Sigma'_2$. By construction, the given standard
  copies $\t Y_{\Sigma}$ and $\t Y_{\Sigma'}$ must lie within the standard
  copies $\t Z_{\Omega_1}$ and $\t R_{\Omega_2}$ respectively, and
  hence intersect within $\t Z_{\Delta \cap
    \Omega}$.

We use the Intersection Axiom of $\X'$ for
$\t Z_{\Delta \cap \Omega}$ and $\t Y_{\Sigma}$ inside $\t Z_{\Omega_1}$ and see that the two copies intersect in a copy of $\t X_{\Delta \cap \Sigma}$, which is also the image of a standard copy of $\t X_{\Delta \cap  \Sigma}$ in $\t Z_{\Delta \cap \Omega}$.

We repeat the argument for $\Sigma'$ and obtain a standard copy
of $\t X_{\Delta \cap  \Sigma'}$ in $\t Z_{\Delta \cap \Omega}$. Now
this copy intersects the one of $\t X_{\Delta \cap  \Sigma}$, and
hence, applying the Intersection Axiom again, they intersect in a
copy of $\t X_{\Delta \cap \Sigma \cap \Sigma'}$ in $\t Z_{\Delta \cap
  \Omega}$. But $\Delta \cap \Sigma \cap \Sigma' = \Sigma \cap
\Sigma'$, and so we have found the desired standard copy in $\t
Z_{\Delta \cap \Omega}$. Now the Composition Property (Lemma~\ref{lem: composition prop}) implies that
this is also a standard copy in $\t X_\Sigma$, $\t X_{\Sigma'}$ and
$\t X_{\Omega_i}$ for any $i \in \{1,2\}$, and thus this is also a standard copy in $\t
X_\Omega$ by construction.

\item[(2,3)] The non-trivial intersection of any standard copy of $\t X_\Sigma$ and any standard copy of $\t X_{\Sigma'}$ in $\t X_\Omega$ is in fact contained in the standard copy of $\t X_{\Omega_i}$, since any copy of $\t X_{\Sigma}$ is contained therein. Therefore the intersection is also contained in a standard copy of $\t X_{\Sigma_i'}$ by construction of $\t \iota_{\Sigma', \Omega}$.
We apply the Intersection Axiom in $\t X_{\Omega_i}$, and observe that the standard copy of $\t X_{\Sigma \cap \Sigma'}$ in $\t X_{\Omega_i}$ obtained this way is also a standard copy in $\t X_\Omega$ by construction.

\item[(2,4)] In this case $\t Y_\Sigma$ must in fact be contained in $\t Q$, where $\t Q$ is either
  $\t Z_{\Omega_1}$  or $\t R_{\Omega_2}$.

We have $\t Y_\Sigma \cap \t Y_{\Sigma'} \subseteq \t Q$, and hence we only need to prove that $\t Y_{\Sigma'}$ is a standard copy in $\t Q$.

Let $\Lambda = \lk_{\Omega}(\Sigma')$. Note that $\lk(\partial \Gamma') \cap (\Delta \cup \Theta) \subseteq \Lambda$.
By definition of $\t \iota_{\Sigma', \Omega}$, we have
\[\t Y_{\Sigma'} = \t \iota_{\Lambda, \Omega}( \{ \t x\} \times \t X_{\Sigma'})\]
 for some point $\t x \in \t X_\Lambda$.
Since $\t Y_{\Sigma'}$ contains a point in $\t Q$, there exists $\t y \in \t X_{\Sigma'}$ such that $\t \iota_{\Lambda, \Omega}(\t x , \t y) \in \t Q$.

If $\Lambda \not \subseteq \Delta \cup \Theta$ then this is only possible if $\t x$ lies in $\t Z_{\Lambda_i}$ or $\t R_{\Lambda_2}$ (depending on what $\t Q$ is), by the construction of $\t X_{\Lambda}$ and $\t \iota_{\Lambda, \Omega}$. But then, again by the construction of $\t X_{\Lambda}$, we have $\t Y_{\Sigma'} \subseteq \t Q$ being a standard copy as claimed.

We still need to check what happens when $\Lambda \subseteq \Delta \cup \Theta$. Suppose that $\t Q = \t R_{\Omega_2}$. In this case $\t Y_{\Sigma'}$ is a standard copy in $\t Q$ by the Composition Property of $\X''$. Lastly, let us suppose that $\t Q = \t Z_{\Omega_1}$. Since $\im (\t \iota_{\Lambda, \Omega}) \subseteq \t R_{\Omega_2}$ by construction, the Intersection Axiom in $\X''$ tells us that $\t Y_{\Sigma'}$ is a standard copy in $\t Z_{\Omega \cap \Delta}$. But then it is also a standard copy in $\t Q = \t Z_{\Omega_1}$ by the Composition Property in $\X'$.

\item[(3,3)] In this case $\t Y_\Sigma$ and
  $\t Y_{\Sigma'}$ are obtained from $\t Z_{\Sigma_1}$, $\t R_{\Sigma_2}$, and $\t Z_{\Sigma'_1}$, $\t R_{\Sigma'_2}$ respectively.

Suppose that $\t Z_{\Sigma_1}$  and $\t Z_{\Sigma'_1}$ do not intersect.
Then $\t R_{\Sigma_2}$  and $\t R_{\Sigma'_2}$ do intersect, and the Intersection Axiom in $\X''$ tells us that they intersect in a standard copy of $\t X_{\Sigma_2 \cap \Sigma'_2}$. The graph $\Sigma_2 \cap \Sigma'_2$ satisfies \eqref{eqn: star}, and so $(\Sigma_2 \cap \Sigma'_2) \cup \Delta = \Delta \cup \Theta$. This implies that the standard copy of $\t X_{\Sigma_2 \cap \Sigma'_2}$ intersects $\t Z_\Delta$ non-trivially (by the Matching Property in $\X''$). But this intersection lies in  $\t Z_{\Sigma_1}$  and $\t Z_{\Sigma'_1}$, and hence they did intersect.

Now
the Intersection Axiom in $\X'$ tells us that $\t Z_{\Sigma_1}$  and $\t Z_{\Sigma'_1}$ intersect in a standard copy of $\t X_{\Sigma_1 \cap \Sigma'_1}$; let us call it $\t Z_{\Sigma_1 \cap \Sigma'_1}$.

Note that the standard copy $\t Z_{\overline \Sigma_1}$ which contains $\t Z_{ \Sigma_1}$ by construction; similarly $\t Z_{\overline{\Sigma'}_1}$ contains $\t Z_{{\Sigma'}_1}$. Now $\t Z_{\overline{\Sigma}_1}$  and $\t Z_{\overline{\Sigma'}_1}$ intersect in the copy $\t Z_{\overline {\Sigma}_1 \cap \overline{\Sigma'}_1}$, which contains
 $\t Z_{\Sigma_1 \cap \Sigma'_1}$, since $\t Z_{\Sigma_1 \cap \Sigma'_1}$ lies in both $\t Z_{\overline{\Sigma}_1}$ and $\t Z_{\overline{\Sigma'}_1}$.
The copy $\t Z_{\overline
   \Sigma_1 \cap \overline{\Sigma'}_1}$ intersects $\t Z_\Delta$, and
 so the Matching Property implies that $\t Z_{\Sigma_1
   \cap \Sigma'_1}$ intersects $\t Z_\Delta$ as well. The Intersection Axiom implies that this intersection is a copy of $\t X_{ \Sigma_1 \cap \Sigma'_1 \cap \Delta}$. This in turn implies that $\t R_{\Sigma_2}$ and $\t R_{\Sigma'_2}$ intersect, and we have already shown above that in this case they intersect in a standard copy of $\t X_{\Sigma_2 \cap \Sigma'_2}$. The union of this copy with $\t Z_{\Sigma_1 \cap \Sigma'_1}$ is by construction a standard copy of $\t X_{\Sigma \cap \Sigma'}$ in $\t X_\Omega$, and again by construction it is the image of a standard copy of $\t X_{\Sigma \cap \Sigma'}$ in $\t X_{\Sigma }$ and $\t X_{\Sigma'}$.

\item[(3,4)] The standard copy $\t Y_\Sigma$ is obtained from $\t Z_{\Sigma_1}$ and $\t R_{\Sigma_2}$. If $\t Y_{\Sigma'}$ intersects $\t Z_{\Sigma_1}$, then we apply case (2,4) to the triple $\Sigma_1, \Sigma', \Omega$, and see that $\t Y_{\Sigma'}$ intersects $\t Z_{\Sigma_1}$ in a copy of $\t X_{\Sigma_1 \cap \Sigma'} = \t X_{\Sigma \cap \Sigma'}$.

If $\t Y_{\Sigma'}$ intersects $\t R_{\Sigma_2}$, then we apply case (2,4) to the triple $\Sigma_2, \Sigma', \Omega$, and see that $\t Y_{\Sigma'}$ intersects $\t R_{\Sigma_2}$ in a copy of $\t X_{\Sigma_1 \cap \Sigma'} = \t X_{\Sigma \cap \Sigma'}$.

If $\t Y_{\Sigma'}$ intersects both $\t Z_{\Sigma_1}$ and $\t R_{\Sigma_2}$, then the two copies we obtained intersect. But they are copies of the same complex, and hence they coincide.

\item[(4,4)] In case (2,4) we have shown that if $\t Y_{\Sigma'}$ intersects $\t Z_{\Omega_1}$, then it lies within as a standard copy; the analogous statement holds for $\t R_{\Omega_2}$. Now the standard copies $\t Y_\Sigma$ and $\t Y_{\Sigma'}$ intersect, and hence they both lie in $\t Q$ as standard copies, where $\t Q$ is $\t Z_{\Omega_1}$ or $\t R_{\Omega_2}$. But now we just need to apply the Intersection Axiom in $\X'$ or $\X''$.
\end{itemize}

\subsubsection*{System Intersection Axiom}
Take a subsystem $\P \subseteq \L^\phi$ closed under taking unions. If
all elements of $\P$ lie in $\Gamma'$ or in $\Delta \cup \Theta$, then
we are done (from the System Intersection Axiom of $\X'$ or
$\X''$). So let us suppose this is not the case, that is suppose that
there exists $\Sigma \in \P$ satisfying \eqref{eqn: star} and $\Sigma \neq \Delta \cup \Theta$. Hence $\bigcup \P \neq \Delta \cup \Theta$ satisfies \eqref{eqn: star}.

Define
\begin{eqnarray*} \P' &=&  \{ \Sigma \in \P \mid \Theta \subseteq \Sigma \} \\
\overline \P' &=& \{ \overline \Sigma = \Sigma \cup \partial \Gamma' \mid \Sigma \in \P'  \}
\end{eqnarray*}
and $\overline \P = \P \cup \overline \P'$. Observe that $\overline \P' \subseteq \L^\phi$, since all graphs in $\P'$ satisfy \eqref{eqn: star}, and so for any $\Sigma' \in \P'$ we have $\overline {\Sigma'} \in \S \subseteq \L^\phi$.

\begin{claim}
$\overline \P$ is closed under taking unions.
\end{claim}
\begin{proof}
Take $\Sigma, \Sigma' \in \overline \P$. If both lie in $\P$ then we are done. Let us first suppose that $\Sigma \in \P$ and $\Sigma' \in \overline \P'$. Then $\Sigma' = \overline {\Sigma''}$ for some $\Sigma'' \in \P'$. Now
\[ \Sigma \cup \Sigma' = \Sigma \cup (\Sigma'' \cup \partial \Gamma')  = (\Sigma \cup \Sigma'') \cup \partial \Gamma' \in \overline \P' \]
since $\Theta \subseteq \Sigma \cup \Sigma'' \in \P$.
If both $\Sigma, \Sigma' \in \overline \P'$ then an analogous argument shows that $\Sigma \cup \Sigma' \in \overline \P'$.
\end{proof}

Now observe that $\overline P'$ is a subsystem of $\overline \P$, which is closed under taking unions.
Hence the same is true for systems $(\overline P)_{\Gamma'}$ and $(\overline P')_{\Gamma'}$.

\smallskip

We are now going to construct standard copies for all elements in $\P$, such that they all intersect non-trivially.

Observe that $\overline P' \subseteq \S$, and so for each element $\overline \Sigma \in (\overline P')_{\Gamma'}$ we are given a standard copy $\t Z_{\overline \Sigma}$ in $\t X_{\Gamma'}$ which contains $\t Z_\Delta$.

Let $\Sigma \in \overline P$. Then $\Sigma \cup \bigcap \overline P' \in \overline P'$, and hence $\Sigma_1 \cup \bigcap (\overline P')_{\Gamma'} \in (\overline P')_{\Gamma'}$, and so we are able to apply Lemma~\ref{lem: extending intersections} to the collection of standard copies we just discussed, and extend it by adding copies $\t Z_{\Sigma_1} $ of $\t X_{\Sigma_1}$ with $\Sigma \in \overline P$, such that all these copies intersect non-trivially.
Moreover, for every $\Sigma \in \P'$, the copy $\t Z_{\Sigma_1}$ will intersect $\t Z_\Delta$ (thanks to the Matching Property in $\t Z_{\overline \Sigma_1}$).

By this point we have constructed standard copies $\t Z_{\Sigma_1}$ for each $\Sigma \in \P$ which all intersect non-trivially. We will now extend these copies to copies of $\t X_\Sigma$.

Let $\Sigma \in \P'$.
If $\Sigma \not\subseteq \Delta \cup \Theta$, then  we can extend $\t Z_{\Sigma_1}$ to a standard copy of $\t X_\Sigma$ in $\t X_\Gamma$ by construction.

When $\Sigma \subseteq \Delta \cup \Theta$ (but still $\Sigma \in \P'$) we need to show that there exists a standard copy of $\t X_\Sigma$ in $\t X_\Gamma$ which contains $\t Z_{\Sigma_1}$.
By construction it is enough to find such a copy of $\t X_\Sigma$ in $\t X_{\Delta \cup \Theta}$, bearing in mind that $\t Z_{\Sigma_1}$ lies in $\t Z_\Delta = \t R_\Delta$.

By \eqref{eqn: star} we know that $\Delta
\cap \lk(\partial \Gamma') \subseteq \Sigma$.

The map $\t \iota_{\lk_{\Delta \cup \Theta}(\partial \Gamma'), \Delta \cup \Theta}$ is onto by the Product Axiom, and so there exists a standard copy of  $\t X_{\lk_{\Delta \cup \Theta}(\partial \Gamma')}$ which intersects $\t Z_{\Sigma_1}$.
But we have $\lk_{\Delta \cup \Theta}(\partial \Gamma') \subseteq
\Sigma$, and thus there exists a copy of $\t X_\Sigma$ in $\t X_{\Delta \cup \Theta}$ (since $\t X_{\Delta \cup \Theta}$ is a product) which contains the given copy of $\t X_{\lk_{\Delta \cup \Theta}(\partial \Gamma')}$, and thus is as required by the Intersection Axiom.

We have finished extending the copies for all $\Sigma \in \P'$. For all $\Sigma \in \P$ with $\Sigma = \Sigma_1$ we do not even need to extend.

\smallskip
It is still possible that there exists $\Sigma \in \P \s- \P'$ such that $\Sigma \not\subseteq \Gamma'$.
Such a $\Sigma$ must satisfy $\Sigma \subseteq \overline \Theta$ (by \eqref{eqn: star}) and $\Sigma \cap \Theta \not\in \{\emptyset, \Theta \}$. But then for all $\Sigma' \in \P$ we have $\Theta \subseteq \Sigma'$ or $\Sigma' \subseteq \overline \Theta$, as otherwise $\Sigma \cup \Sigma'$ would violate \eqref{eqn: star}. So in this situation $\P \s- \P' \subseteq \L^\phi_{\Delta \cup \Theta}$.

Consider the subsystem $\P'_{\Delta \cup \Theta}$ of $\P_{\Delta \cup \Theta}$. It is closed under taking unions, since $\P'$ is, and for each $\Sigma \in \P$ we have $\Sigma_2 \cup \bigcap \P'_{\Delta \cup \Theta} \in \P'_{\Delta \cup \Theta}$. Now our standard copies of $\t X_\Sigma$ for $\Sigma \in \P'$ give us (by the Intersection Axiom) standard copies of $\t X_{\Sigma_2}$ in $\t X_{\Delta \cup \Theta}$ which intersect in a standard copy of $\t X_{(\bigcap \P')_2}$, and hence non-trivially. Lemma~\ref{lem: extending intersections} gives us a collection of standard copies of $\t X_{\Sigma_2}$ in $\t X_{\Delta \cup \Theta}$ for all $\Sigma \in \P$, which contains the previously discussed collection, and such that all of these standard copies intersect non-trivially. For each $\Sigma \in \P \s- \P'$ we have $\Sigma = \Sigma_2$, and the standard copy of $\t X_\Sigma$ in $\t X_{\Delta \cup \Theta}$ becomes a standard copy in $\t X_\Gamma$ by construction. \qed

\bibliographystyle{math}
\bibliography{raags}

\end{document}